\documentclass[a4paper]{amsart}
\usepackage[top=3cm, bottom=3cm, left=3cm, right=3cm]{geometry}

\usepackage{enumerate}
\usepackage{amsmath}
\usepackage{amssymb}
\usepackage{amsthm}

\usepackage{lpic}
\newlength{\figurewidth}
\newlength{\annotationwidth}

\usepackage{math}
\usepackage{thm}
\usepackage{ref}

\title
	[Dynamic Outflow Boundary Conditions]
	{Strong Well-Posedness for a Class of Dynamic Outflow Boundary Conditions for Incompressible Newtonian Flows}

\dedicatory
	{Dedicated to Jan Pr{\"u}ss on the occasion of his $65^{\textrm th}$ birthday.}

\author
	[Dieter Bothe]
	{Dieter Bothe}
	
\address
	{Center of Smart Interfaces \newline\indent
	 Technische Universit{\"a}t Darmstadt, \newline\indent
	 Alarich-Weiss-Str.~10, D-64287 Darmstadt, Germany}
	
\email
	{bothe@csi.tu-darmstadt.de}

\author
	[Takahito Kashiwabara]
	{Takahito Kashiwabara}

\address
	{Department of Mathematics \newline\indent
	Tokyo Institute of Technology \newline\indent
	2-12-1 Ookayama, Meguro, Tokyo 152-8511, Japan}

\email
	{tkashiwa@math.titech.ac.jp}

\author
	[Matthias K{\"o}hne]
	{Matthias K{\"o}hne}

\address
	{Mathematisches Institut \newline\indent
	 Heinrich-Heine-Universit{\"a}t D{\"u}sseldorf, \newline\indent
	 Universit{\"a}tsstr.~1, D-40225 D{\"u}sseldorf, Germany}
	
\email
	{koehne@math.uni-duesseldorf.de}



\date
	{\today}

\begin{document}
\begin{abstract}
	Based on energy considerations, we derive a class of dynamic outflow boundary conditions for the incompressible Navier-Stokes equations,
	containing the well-known convective boundary condition but incorporating also the stress at the outlet.
	As a key building block for the analysis of such problems, we consider the Stokes equations with such dynamic outflow boundary conditions
	in a halfspace and prove the existence of a strong solution in the appropriate Sobolev-Slobodeckij-setting with $L_p$ (in time and space) as the base space for the momentum balance.
	For non-vanishing stress contribution in the boundary condition, the problem is actually shown to have $L_p$-maximal regularity under the natural compatibility conditions.
	Aiming at an existence theory for problems in weakly singular domains, where different boundary conditions apply on different parts of the boundary such that these surfaces meet orthogonally,
	we also consider the prototype domain of a wedge with opening angle $\frac{\pi}{2}$ and different combinations of boundary conditions:
	Navier-Slip with Dirichlet and Navier-Slip with the dynamic outflow boundary condition.
	Again, maximal regularity of the problem is obtained in the appropriate functional analytic setting and with the natural compatibility conditions.
\end{abstract}
\setlength{\figurewidth}{10cm}
\setlength{\annotationwidth}{\textwidth-\figurewidth-1cm}
\renewcommand{\baselinestretch}{1.125}
\normalsize
\maketitle

\section*{Introduction}
In the numerical modeling of fluid flows from real world applications it is often not possible to model the complete flow domain up to physical boundaries.
Instead, artificial boundaries usually need to be introduced into the problem description.
In such cases the formulation of sensible boundary conditions, so-called artificial boundary conditions (ABCs, for short),
is a non-trivial task since the flow can enter and, more problematic, leave the domain through open parts of the boundary.
We speak of an ``outflow boundary'' if the mean flow points outwards, while locally a backflow -- with fluid entering the domain -- is allowed.
One important class of ABCs at such outflow boundaries are ``convective'' boundary conditions like
\begin{equation}
	\eqnlabel{convectiveBC}
	\partial_t \phi + (a \cdot \nabla) \phi = 0
\end{equation}
with a prescribed velocity $a$, where $\phi$ denotes a transported quantity, say a velocity component.
Such dynamic ABCs are known since long in the area of hyperbolic problems, also called Sommerfeld radiation condition in this context.
While $a$ usually denotes the phase velocity of the waves, which is hard to be known a priori,
Orlanski used a local velocity $a$ in his numerical studies in \cite{Orlanski:Simple-Boundary-Condition}.
In \cite{Halpern:ABC-Advection-Diffusion}, using Fourier techniques and approximations in the transformed space similar to \cite{Engquist:Absorbing-Boundary-Conditions},
the convective ABC above was derived as an approximation to the non-local exact boundary condition for a linear advection-diffusion equation.
In \cite{Halpern:ABC-Incompressible-Viscous-Flows}, different approximations to the symbol of the exact boundary operator for the linearized incompressible Navier-Stokes equations have been derived,
but these approximations often lead to non-local boundary conditions.
One local condition given there for 2D flow is the combination of \eqnref{convectiveBC} for the normal velocity component
with a homogeneous Neumann condition for the tangential part.
The full incompressible Navier-Stokes equations are also treated in \cite{Jin-Braza:Non-Reflecting-Outlet-Condition},
where the resulting ABC is chosen to contain an additional (viscous) diffusion term acting in the tangential direction.

Since the derivation of local ABCs of convective type are not strictly feasible for the incompressible Navier-Stokes or Stokes equations,
we adopt a different approach based on energy considerations, somewhat in the spirit of \cite{Bothe-Koehne-Pruess:Energy-Preserving-Boundary-Conditions}.
These considerations also motivate the incorporation of additional stress terms and, moreover, lead to several variants of such dynamic outflow boundary conditions.
Since it is very important also for the numerical applications that the chosen boundary conditions lead to wellposed initial-boundary-value problems,
the main focus of the present work is the analysis of the resulting PDE system concerning the local-in-time wellposedness
in appropriate Bessel potential and Sobolev-Slobodeckij spaces.
To our knowledge, at least in the context of strong solutions such an analysis has not been done so far.
But let us note that for other classes of ABCs, also employed at outflow boundaries, some analytical results are known;
cf.\ \cite{Bothe-Koehne-Pruess:Energy-Preserving-Boundary-Conditions} and the references given there.

Let us finally note that in the numerical description of real world flow problems,
the computational domains usually contain edges at which different boundary conditions meet.
Such mixed-type initial-boundary-value problems for the Navier-Stokes or Stokes equations in singular domains are very challenging concerning,
e.g., their rigorous analysis.
In some prototype cases, like a flow in a system of pipes, the flow domain can be chosen such that it is only weakly singular,
meaning that if different boundary parts meet at a common edge, they locally form a right angle there. This is illustrated in Figure~1.
\begin{figure}[t]
	\vspace*{-0.5cm}
	\begin{minipage}[t]{\figurewidth}
		\rotatebox{-90}{
			\begin{lpic}{Y-Tube-Figure(,10cm)}
				\lbl{90,7.5,90;$\Gamma^{\textrm{in}}$}
				\lbl{45,50,90;$\Gamma^{\textrm{wall}}$}
				\lbl{27.5,127.5,90;$\Gamma$}
				\lbl{47.5,162.5,90;$\Gamma$}
			\end{lpic}
		}
	\end{minipage}
	\hspace*{\fill}
	\begin{minipage}[t]{\annotationwidth}
		\vspace*{1cm}
		Figure 1: Example for a weakly singular domain; here: a smooth tube with one inlet $\Gamma^{\textrm{in}}$ and two outlets $\Gamma \ (= \Gamma^{\textrm{out}})$.
		The lateral boundary is denoted by $\Gamma^{\textrm{wall}}$.
		The arrows indicate the principal flow direction.
	\end{minipage}
\end{figure}

There, the flow enters the domain via an inlet, while two outlets are available for the fluid to leave the domain.
All in-/outlets are ``connected'' by an impermeable wall which forms the lateral boundary of the tube.
Such a tube is a typical example of a weakly singular domain $\Omega \subseteq \bR^n$,
whose boundary may be decomposed into several smooth parts that meet each other orthogonally.
For the example in Figure 1, the smooth parts of the boundary are the inlet $\Gamma^{\textrm{in}}$,
the lateral boundary of the tube $\Gamma^{\textrm{wall}}$, and the outlets $\Gamma$.

For the right combinations of boundary conditions, such weakly singular domains can be treated for a variety of admissible boundary conditions
as has been shown in \cite{Koehne:Incompressible-Newtonian-Flows}.
The key model problem required to be treatable for such weakly domains are the corresponding PDE systems in a wedge of opening angle $\frac{\pi}{2}$.
For this reason, the analysis for this prototype geometry is included in the present paper.

\section{Dynamic Outflow Boundary Conditions}
\seclabel{Modeling}
We aim at deriving physically meaningful boundary conditions at outflow boundaries
which render the artificial boundary transparent in the sense that the boundary condition does not introduce unphysical dissipation into the system.
While our motivation mainly stems from so-called non-reflecting boundary conditions developed for partial differential equations
of hyperbolic character such as wave equations or compressible flows, their derivation requires a different approach
because we aim at applications to flow problems for incompressible Newtonian fluids.
The basic idea of our derivation is the preservation of kinetic energy in the following sense:
if an outflow boundary $\Gamma^{\rm out}$ is observed at arbitrary time $t = t_0$,
the (infinitesimally thin) layer of fluid exiting the flow domain $\Omega$ at this time instant should not endure a change of its kinetic energy.
In mathematical terms this means that
\begin{equation}
	\label{outflow-dissipation1}
	\frac{\mbox{d}}{\mbox{d}t} \!\! \int_{\Gamma^{\rm out}} \!\! \rho\,\frac{v(t,x(t;\,t_0,\,x_0))^2}{2}\,\mbox{d}\sigma(x_0)_{\,|t = t_0} = 0,
\end{equation}
where $\rho$ is the constant mass density and $x(\cdot) = x(\cdot;\,t_0,\,x_0)$ denotes the unique solution of
\begin{equation}
	\label{equation-of-motion}
	\dot{x}(t) = v(t,\,x(t)), \quad x(t_0) = x_0.
\end{equation}
Let us note that the rate of change
of kinetic energy given by the left-hand side of \eqref{outflow-dissipation1} is, in general, not the same as
\begin{equation}
	\label{outflow-dissipation2}
	\frac{\mbox{d}}{\mbox{d}t} \!\! \int_{\Sigma(t)} \!\! \rho \,\frac{v(t,\,x)^2}{2}\,\mbox{d}\sigma(x)_{\,|t = t_0} = 0,
\end{equation}
where $\Sigma(t)$ is the surface composed of all fluid particles at time $t$ which exit through $\Gamma^{\rm out}$ at $t_0$.
The reason are the different surface measures in \eqref{outflow-dissipation1}, resp.\ \eqref{outflow-dissipation2}.
To decide which expression is the physically correct one,
notice first that the integral in \eqref{outflow-dissipation1} stands for a thin layer of fluid of a given constant thickness $\delta > 0$, say,
since kinetic energy is stored in the mass of the fluid which requires a volume instead of an area to support for it.
If this sheet of fluid is traced backwards along the flow trajectories, the thickness as well as the local surface area change.
While the different surface measure in \eqref{outflow-dissipation2} accounts for the local area changes,
the effect of a changing thickness is not included.
On the other side, the surface measure in \eqref{outflow-dissipation1} assigns to a fluid parcel with base area
$\mbox{d}\sigma$ and thickness $\delta$ a constant weight during its motion, as it should be due to the conserved volume
corresponding to vanishing divergence of the velocity field.

Notice that \eqref{equation-of-motion} typically is an end value instead of an initial value problem,
since the right-hand side in \eqref{equation-of-motion} is only defined for $t\leq t_0$ if the fluid trajectory is leaving the domain.
But for a bounded and locally Lipschitz velocity field $v$,
say, an extension of $v$ with the same regularity to a neighborhood of $\Omega$ is possible
such that \eqref{equation-of-motion} then has unique solutions at least on $(t_0-\epsilon,t_0+\epsilon)$ for some $\epsilon>0$.
Then the derivative in \eqref{outflow-dissipation1} is also well-defined if the fluid locally enters the domain via $\Gamma^{\rm out}$.
Computing the derivative in \eqref{outflow-dissipation1} yields
\begin{equation}
	\label{outflow-diss2}
	\int_{\Gamma^{\rm out}} \!\! v(t_0,\,x_0) \cdot \rho \big( \partial_t v(t_0,\,x_0) + \nabla_x v(t_0,\,x_0) \cdot v(t_0,\,x_0) \big)\,\mbox{d}\sigma(x_0) = 0.
\end{equation}
Since this should hold for any time and any velocity field, the only appropriate local condition to assure \eqref{outflow-diss2} is the condition
\begin{equation*}
	\label{outflow-BC0}
	v \cdot \big( \partial_t v  + ( v \cdot \nabla) v \big) =0 \quad \mbox{ on } \Gamma^{\rm out}.
\end{equation*}
Evidently, the dynamic boundary condition
\begin{subequations}
\begin{equation}
	\label{outflow-DBC1}
	\partial_t v + (v \cdot \nabla) v = 0 \quad \mbox{ on } \Gamma^{\rm out}
\end{equation}
on the full velocity is sufficient for this to hold.
In cases when it is reasonable to assume the outgoing flow to be perpendicular to the outflow boundary,
the dynamic condition only needs to hold for the normal velocity component, i.e.\ the following variant is also sufficient:
\begin{equation}
	\label{outflow-DBC2}
	P_\Gamma v = 0, \qquad (\partial_t v + (v \cdot \nabla) v) \cdot \nu = 0 \quad \mbox{ on } \Gamma^{\rm out},
\end{equation}
where $\nu: \Gamma \longrightarrow \bR^n$ denotes the outer unit normal field at $\Gamma^{\rm out}$
and $P_\Gamma := 1 - \nu \otimes \nu$ denotes the projection onto the tangent bundle.
Another variant describes the normal (outgoing, say) velocity component and imposes only the tangential part of the dynamic condition, i.e.\
\begin{equation}
\label{outflow-DBC-prescribed-outflow}
	v \cdot \nu = v^{\rm out}(t,\,x) \cdot \nu, \qquad
	P_\Gamma  (\partial_t v + (v \cdot \nabla) v)  = 0 \quad \mbox{ on } \Gamma^{\rm out}
\end{equation}
\end{subequations}
with a given outflow velocity $v^{\rm out}(t,\,x)$.

These dynamic ABCs are nonlinear boundary conditions which, in particular for numerical purpose,
might be approximated by the linearized versions.
For example, the linearized version of \eqref{outflow-DBC1} reads as
\begin{subequations}
\begin{equation}
	\label{outflow-DBC3}
	\partial_t v + (v^{\rm out}(t,\,x) \cdot \nabla) v = 0 \quad \mbox{ on } \Gamma^{\rm out}
\end{equation}
with a given outflow velocity $v^{\rm out}(t,\,x)$.
In practice, the latter velocity will also be unknown,
but certain additional assumptions may be reasonable like perpendicular outflow velocity.
Then \eqref{outflow-DBC3} becomes
\begin{equation}
	\label{outflow-DBC4}
	\partial_t v + V(t,\,x) \partial_\nu v = 0 \quad \mbox{ on } \Gamma^{\rm out}
\end{equation}
\end{subequations}
with a scalar function $V(t,\,x)$ which is assumed to be known.
In case the mean flow across the outflow boundary is known,
it is of special interest to consider \eqref{outflow-DBC4} with $V(t,\,x) \equiv V^{\rm out}$,
where $V^{\rm out}$ is either constant or a known function of time.

At this point it is important to mention that the analysis below will also show that the Stokes problem
in a half-space together with the dynamic ABC \eqref{outflow-DBC3}, or even \eqref{outflow-DBC4},
is not well-posed in the considered Sobolev-Slobodeckij-setting; cf.\ \Remref{Normal-Pressure-Effect}.
Therefore, an appropriate modification of this condition is required.

For this purpose, recall first that the kinetic energy
\begin{equation*}
	\label{Ekin}
	E_{\rm kin}:=\int_\Omega \rho \frac{v^2}{2}\,\mbox{d}x
\end{equation*}
contained in the full domain changes at the rate
\begin{equation*}
	\dot{E}_{\rm kin} = - 2 \eta \int_\Omega D : \nabla v\,\mbox{d}x + \int_\Omega \rho b \cdot v\,\mbox{d}x
		+ \int_{\partial \Omega} v \cdot S \nu\,\,\mbox{d}\sigma - \int_{\partial \Omega} \rho \frac{v^2}{2} v \cdot \nu\,\mbox{d}\sigma,
\end{equation*}
where $\eta$ is the dynamic viscosity of the fluid,
$D = \frac{1}{2}(\nabla v + \nabla v^\sfT)$ denotes the symmetric velocity gradient, $S = 2 \eta D - pI$ denotes the stress tensor,
$p$ is the pressure, and $b$ are the body force densities.
We decompose the full boundary into disjoint parts according to $\partial \Omega = \Gamma^{\rm in} \cup \Gamma^{\rm wall} \cup \Gamma^{\rm out}$,
where we assume that $v \cdot S\nu = 0$ on $\Gamma^{\rm wall}$.
Hence, we obtain
\begin{align}
	\dot{E}_{\rm kin} = & - 2 \eta \int_{\Omega} D : \nabla v\,\mbox{d}x + \int_{\Omega} \rho b \cdot v\,\mbox{d}x \nonumber \\[1ex]
		& - \int_{\Gamma^{\rm in}} \rho \frac{v^2}{2} v \cdot \nu\,\mbox{d}\sigma + \int_{\Gamma^{\rm in}} v \cdot S \nu\,\mbox{d}\sigma
			- \int_{\Gamma^{\rm out}} \rho \frac{v^2}{2} v \cdot \nu\,\mbox{d}\sigma + \int_{\Gamma^{\rm out}} v \cdot  S \nu\,\mbox{d}\sigma \nonumber
\end{align}
as the  rate of change of this energy functional.
On the boundaries, the terms with $\rho v^2/2$ describe convective in- and output to the open domain $\Omega$,
hence are not related to dissipation.
Therefore, the condition for a non-dissipative outflow boundary becomes
\begin{equation*}
	v \cdot S \nu = 0 \quad \mbox{ on } \Gamma^{\rm out},
\end{equation*}
which is satisfied if, e.g., the homogeneous Neumann condition holds, i.e.\ $S\nu=0$ on $\Gamma^{\rm out}$.
Other variants, analogous to the variants above,
are $P_\Gamma v = 0$ and $S\nu \cdot \nu = 0$ or $v \cdot \nu = 0$ and $P_\Gamma S \nu = 0$.
These boundary conditions are natural conditions in the sense that they eliminate the corresponding boundary
term in the variational formulation. In a Finite Element context, the omission of the boundary term is also refered
to as the ``do-nothing condition''; see \cite{Heywood-Rannacher-Turek:Artificial-Boundaries}.
Let us also note that well-posedness as well as $L_p$-maximal regularity are known for the Stokes and for the
Navier-Stokes equations with Neumann boundary condition;
see the remarks and references in \cite{Bothe-Koehne-Pruess:Energy-Preserving-Boundary-Conditions, Koehne:Incompressible-Newtonian-Flows}.

At this point, we have two different sets of artificial boundary conditions, which are all motivated from energy considerations.
Somewhat similar to the Robin boundary condition as a linear combination of a Dirichlet and a
Neumann condition, we consider the following types of dynamics outflow boundary conditions, obtained by linear combination
of a convective-type linearized dynamic condition and the corresponding variant of the Neumann-type condition:
The fully dynamic condition
\begin{subequations}
\begin{equation}
	\label{outflow-DBC5}
	\alpha (\partial_t v + (v^{\rm out}(t,\,x) \cdot \nabla) v) + S \nu = 0 \quad \mbox{ on } \Gamma^{\rm out},
\end{equation}
the normally dynamic variant
\begin{equation}
	\label{outflow-DBC6}
	P_\Gamma v = 0, \qquad \alpha (\partial_t v + (v^{\rm out}(t,\,x) \cdot \nabla) v) \cdot \nu + S \nu \cdot \nu = 0 \quad \mbox{ on } \Gamma^{\rm out}
\end{equation}
and the tangentially dynamic variant
\begin{equation}
	\label{outflow-DBC7}
	v \cdot \nu = 0, \qquad \alpha P_\Gamma  (\partial_t v + (v^{\rm out}(t,\,x) \cdot \nabla) v) + P_\Gamma S \nu = 0 \quad \mbox{ on } \Gamma^{\rm out}.
\end{equation}
\end{subequations}
In all three ABCs above, $\alpha>0$ is a model parameter.
Let us note in passing that the new ABCs \eqref{outflow-DBC5}-\eqref{outflow-DBC7} could also be derived
directly from a combined energy functional. In this case, also the nonlinear variants with $v$ instead of
$v^{\rm out}(t,\,x)$ would be reasonable choices.

\subsection*{A Complete Model}
We now pass to the dimensionless form, writing $u$ for the non-dimensional velocity.
Moreover, in order to economize the notation we write $\Gamma = \Gamma^{\rm out}$ for an outflow boundary as in Figure~1.
This yields
\begin{equation*}
	\eqnlabel{Interior}\tag*{$(\textrm{\upshape NS})^f_{\textrm{\upshape Re}}$}
	\begin{array}{rclll}
		\partial_t u + (u \cdot \nabla u) - \frac{1}{\textrm{Re}} \Delta u + \nabla p & = & f & \quad \mbox{in} & J \times \Omega, \\[0.5em]
		                                                                \mbox{div}\,u & = & 0 & \quad \mbox{in} & J \times \Omega
	\end{array}
\end{equation*}
as the well-known dimensionless form of the Navier-Stokes equation inside the domain.
Here $J := (0,\,a)$ with $a > 0$ denotes the time interval within which the flow is to be modeled,
and ${\rm Re} > 0$ is the Reynolds number.
At the outflow boundary, we first record the full dynamic outflow boundary condition, i.e.
\begin{equation*}
	\alpha (\partial_t u + (v^{\textrm{out}} \cdot \nabla) u) + S \nu = 0 \quad \mbox{on} \ J \times \Gamma.
\end{equation*}
Since the normal and the tangential parts are treated differently below, we also write the full dynamic outflow condition in the form
\begin{equation*}
	\eqnlabel{Fully-NRBC}\tag*{$(\textrm{FDO})^{v^{\textrm{out}}}_{\alpha, \textrm{\upshape Re}}$}
		\begin{array}{rclll}
		       \alpha P_\Gamma (\partial_t u + (v^{\textrm{out}} \cdot \nabla) u) + \frac{2}{\textrm{Re}} P_\Gamma D \nu & = & 0 & \quad \mbox{on} & J \times \Gamma, \\[0.5em]
		\alpha (\partial_t u + (v^{\textrm{out}} \cdot \nabla) u) \cdot \nu + \frac{2}{\textrm{Re}} D \nu \cdot \nu -  p & = & 0 & \quad \mbox{on} & J \times \Gamma.
	\end{array}
\end{equation*}
A variant of this ABC imposes the dynamic condition on the normal part, only, and reads as
\begin{equation*}
	\eqnlabel{Normal-NRBC}\tag*{$(\textrm{NDO})^{v^{\textrm{out}}}_{\alpha, \textrm{\upshape Re}}$}
	\begin{array}{rclll}
		                                                                                                     P_\Gamma u & = & 0 & \quad \mbox{on} & J \times \Gamma, \\[0.5em]
		\alpha (\partial_t u + (v^{\textrm{out}} \cdot \nabla) u) \cdot \nu + \frac{2}{\textrm{Re}} D \nu \cdot \nu - p & = & 0 & \quad \mbox{on} & J \times \Gamma.
	\end{array}
\end{equation*}
Finally, there is a third version which imposes the dynamic condition on the tangential component and reads as
\begin{equation*}
	\eqnlabel{Tangential-NRBC}\tag*{$(\textrm{TDO})^{v^{\textrm{out}}}_{\alpha, \textrm{\upshape Re}}$}
	\begin{array}{rclll}
		\alpha P_\Gamma (\partial_t u + (v^{\textrm{out}} \cdot \nabla) u) + \frac{2}{\textrm{Re}} P_\Gamma D \nu & = & 0 & \quad \mbox{on} & J \times \Gamma, \\[0.5em]
		                                                                                              u \cdot \nu & = & 0 & \quad \mbox{on} & J \times \Gamma.
	\end{array}
\end{equation*}
Note that the homogeneous version of the boundary condition above actually assumes an impermeable boundary,
but the theorems to follow treat the nonhomogeneous case as well.
For the last ABC, this means a prescribed outgoing normal velocity component.
In all boundary conditions above, we assume the velocity $v^{\textrm{out}}$ to be a priori given and to be of the form
\begin{equation*}
	\eqnlabel{Parameter-Compatibility}\tag{CP}
	v^{\textrm{out}} = V \nu, \qquad \mbox{ where $V=V(t,x)$ satisfies } \alpha V + \frac{1}{\textrm{Re}} > 0.
\end{equation*}
Let us note that in the main results to follow, we actually assume $V$ to be constant, since the considered prototype
model problems result by a localization process.

Finally, in order to provide a full model for weakly singular domains like the tube in Figure~1,
boundary conditions have to be prescribed for the other parts of the boundary as well.
For an inlet like $\Gamma^{\textrm{in}}$ it is reasonable to assume an inflow condition
\begin{equation*}
	\eqnlabel{Inflow-BC}\tag*{$(\textrm{\upshape IF})^{u^{\textrm{in}}}$}
	u = u^{\textrm{in}} \quad \mbox{on} \ J \times \Gamma^{\textrm{in}}
\end{equation*}
with a prescribed velocity profile $u^{\textrm{in}}$.
On a lateral wall like $\Gamma^{\textrm{wall}}$ a Navier type condition
\begin{equation*}
	\eqnlabel{Wall-BC}\tag*{$(\textrm{\upshape W})_{\sigma, \textrm{Re}}$}
	\begin{array}{rclll}
		\sigma P_\Gamma u + \frac{2}{\textrm{Re}} P_\Gamma D \nu & = & 0 & \quad \mbox{on} & J \times \Gamma^{\textrm{wall}}, \\[0.5em]
		                                             u \cdot \nu & = & 0 & \quad \mbox{on} & J \times \Gamma^{\textrm{wall}}
	\end{array}
\end{equation*}
with some friction/slip-length $\sigma \geq 0$ is suitable to describe the frictional flow along a wall.

\section{Main Results}
The remaining part of the paper is devoted to the analysis of the Stokes equations $(\textrm{S})^{f, g, u_0}_{\textrm{\upshape Re}}$
subject to a dynamic outflow boundary condition $(\textrm{BDO})^{v^{\textrm{out}}, h}_{\alpha, \textrm{\upshape Re}}$
with $B \in \{\,T,\,N,\,F\,\}$ in several prototype situations.
Our approach is based on $L_p$-maximal regularity for suitable linearizations of the models.
A generic approach to analyze the Stokes and Navier-Stokes equations subject to a large class of different boundary conditions in this setting
has been developed in \cite{Bothe-Koehne-Pruess:Energy-Preserving-Boundary-Conditions, Koehne:Incompressible-Newtonian-Flows}.
In these sources the focus is set on so-called {\itshape energy preserving boundary conditions} which are of local and non-dynamic nature.
However, this generic approach together with generic results on parabolic problems subject to dynamic boundary conditions
as developed in \cite{Denk-Pruess-Zacher:Dynamic-Boundary-Conditions} may be adapted to the Stokes equations
subject to dynamic outflow boundary conditions $(\textrm{BDO})^{v^{\textrm{out}}, h}_{\alpha, \textrm{\upshape Re}}$ with $B \in \{\,T,\,N,\,F\,\}$.

\subsection{Prototype Models}
Here we focus on two prototype models:
Again, we set $J := (0,\,a)$ with $a > 0$.
We first study the fully inhomogeneous Stokes equations
\begin{equation*}
	\eqnlabel{Interior-Linear}\tag*{$(\textrm{\upshape S})^{f, g, u_0}_{\textrm{\upshape Re}}$}
	\begin{array}{rclll}
		\partial_t u - \frac{1}{\textrm{Re}} \Delta u + \nabla p & = & f   & \quad \mbox{in} & J \times \Omega, \\[0.5em]
		                                           \mbox{div}\,u & = & g   & \quad \mbox{in} & J \times \Omega, \\[0.5em]
		                                                    u(0) & = & u_0 & \quad \mbox{in} & \Omega
	\end{array}
\end{equation*}
in a halfspace $\Omega = \bR^n_+ := \{\,(x,\,y) \in \bR^n\,:\,x \in \bR^{n - 1},\ y > 0\,\}$,
subject to a fully inhomogeneous linear dynamic outflow boundary condition on $\Gamma = \partial \Omega$,
i.\,e.\ we either consider the tangentially dynamic outflow boundary condition
\begin{equation*}
	\eqnlabel{Tangential-NRBC-Linear}\tag*{$(\textrm{TDO})^{v^{\textrm{out}}, h}_{\alpha, \textrm{Re}}$}
	\begin{array}{rclll}
		\alpha P_\Gamma (\partial_t u + (v^{\textrm{out}} \cdot \nabla) u) + \frac{2}{\textrm{Re}} P_\Gamma D \nu & = & P_\Gamma h  & \quad \mbox{on} & J \times \Gamma, \\[0.5em]
		                                                                                              u \cdot \nu & = & h \cdot \nu & \quad \mbox{on} & J \times \Gamma,
	\end{array}
\end{equation*}
or the normally dynamic outflow boundary condition
\begin{equation*}
	\eqnlabel{Normal-NRBC-Linear}\tag*{$(\textrm{NDO})^{v^{\textrm{out}}, h}_{\alpha, \textrm{\upshape Re}}$}
	\begin{array}{rclll}
		                                                                                                     P_\Gamma u & = & P_\Gamma h  & \quad \mbox{on} & J \times \Gamma, \\[0.5em]
		\alpha (\partial_t u + (v^{\textrm{out}} \cdot \nabla) u) \cdot \nu + \frac{2}{\textrm{Re}} D \nu \cdot \nu - p & = & h \cdot \nu & \quad \mbox{on} & J \times \Gamma,
	\end{array}
\end{equation*}
or the fully dynamic outflow boundary condition
\begin{equation*}
	\eqnlabel{Fully-NRBC-Linear}\tag*{$(\textrm{FDO})^{v^{\textrm{out}}, h}_{\alpha, \textrm{\upshape Re}}$}
	\begin{array}{rclll}
		      \alpha P_\Gamma (\partial_t u + (v^{\textrm{out}} \cdot \nabla) u) + \frac{2}{\textrm{Re}} P_\Gamma D \nu & = & P_\Gamma h  & \quad \mbox{on} & J \times \Gamma, \\[0.5em]
		\alpha (\partial_t u + (v^{\textrm{out}} \cdot \nabla) u) \cdot \nu + \frac{2}{\textrm{Re}} D \nu \cdot \nu - p & = & h \cdot \nu & \quad \mbox{on} & J \times \Gamma.
	\end{array}
\end{equation*}
Here, $\nu: \Gamma \longrightarrow \bR^n$ again denotes the outer unit normal at the boundary
while we denote by $P_\Gamma := 1 - \nu \otimes \nu$ the projection onto the tangent bundle at the boundary.
Based on our $L_p$-maximal regularity result \Thmref{Linear-Halfspace} the localization procedure presented in
\cite{Bothe-Koehne-Pruess:Energy-Preserving-Boundary-Conditions, Koehne:Incompressible-Newtonian-Flows} leads to corresponding results
for the fully inhomogeneous linear problem in bounded, smooth domains.
However, the details of this localization procedure shall not be presented here.

As a second prototype problem we study the fully inhomogeneous Stokes equations $(\textrm{S})^{f, g, u_0}_{\textrm{\upshape Re}}$
in a wedge \mbox{$\Omega = \bR^n_{+\!\!\!+} := \{\,(x,\,y,\,z) \in \bR^n\,:\,x \in \bR^{n - 2},\ y > 0,\ z > 0\,\}$}.
This prototype domain has two smooth boundary parts which we denote by
\begin{equation*}
	\partial_y \bR^n_{+\!\!\!+} := \Big\{\,(x,\,y,\,z) \in \bR^n\,:\,x \in \bR^{n - 2},\ y = 0,\ z > 0\,\Big\},
\end{equation*}
and $\partial_z \bR^n_{+\!\!\!+}$, respectively.
In order to be able to study domains like the tube in Figure~1, we consider the situation
$\Gamma^{\textrm{wall}} := \partial_y \bR^n_{+\!\!\!+}$ with a fully inhomogeneous Navier condition
\begin{equation*}
	\eqnlabel{Wall-BC-Linear}\tag*{$(\textrm{\upshape W})^{h^{\textrm{wall}}}_{\sigma, \textrm{\upshape Re}}$}
	\begin{array}{rclll}
		\sigma P_\Gamma u + \frac{2}{\textrm{Re}} P_\Gamma D \nu & = & P_\Gamma h^{\textrm{wall}}  & \quad \mbox{on} & (0,\,a) \times \Gamma^{\textrm{wall}}, \\[0.5em]
		                                             u \cdot \nu & = & h^{\textrm{wall}} \cdot \nu & \quad \mbox{on} & (0,\,a) \times \Gamma^{\textrm{wall}}
	\end{array}
\end{equation*}
in combination with $\Gamma^{\textrm{in}} := \partial_z \bR^n_{+\!\!\!+}$ with the inflow condition $(\textrm{IF})^{u^{\textrm{in}}}$.
Moreover, we consider the situation $\Gamma^{\textrm{wall}} := \partial_y \bR^n_{+\!\!\!+}$ with a fully inhomogeneous
Navier condition $(\textrm{W})^{h^{\textrm{wall}}}_{\sigma, \textrm{\upshape Re}}$ in combination with $\Gamma := \partial_z \bR^n_{+\!\!\!+}$
with one of the fully inhomogeneous dynamic outflow boundary conditions $(\textrm{BDO})^{v^{\textrm{out}}, h}_{\alpha, \textrm{\upshape Re}}$ with $B \in \{\,T,\,N,\,F\,\}$.
Based on our $L_p$-maximal regularity results Theorems~\ref{thm:Linear-Wedge-Inflow} and \ref{thm:Linear-Wedge-Non-Reflecting}, the localization procedure presented in
\cite[Chapter~8]{Koehne:Incompressible-Newtonian-Flows} leads to corresponding results for the fully inhomogeneous linear problem in weakly singular domains like the one shown in Figure~1.
However, the fully general notion of weakly singular domains is not needed in the present work. Moreover, due to
space limitations, the details of the localization procedure are also not given here.

\subsection{Necessary Regularity/Compatibility Conditions}
Our approach leads to $L_p$-maximal regular solutions to $(\textrm{S})^{f, g, u_0}_{\textrm{\upshape Re}}$,
i.\,e.\ we assume $f \in L_p(J \times \Omega)^n$ and obtain
\begin{equation*}
	u \in H^1_p(J,\,L_p(\Omega)^n) \cap L_p(J,\,H^2_p(\Omega)^n), \qquad p \in L_p(J,\,\dot{H}^1_p(\Omega)),
\end{equation*}
where $J = (0,\,a)$, $\Omega \in \{\,\bR^n_+,\,\bR^n_{+\!\!\!+}\,\}$,
and $[\,H^s_p(J,\,\cdot\,),\ H^s_p(\Omega)\,:\,s \geq 0,\ 1 < p < \infty\,]$ denotes the scale of (vector-valued) Bessel-potential spaces.
Moreover, in order to handle the pressure we denote by $[\,\dot{H}^s_p(\Omega)\,:\,s \geq 0,\ 1 < p < \infty\,]$ the scale of homogeneous Bessel-potential spaces.
However, if the pressure does not appear in the boundary condition, then it is only unique up to an additive constant.
Hence, in some situations we obtain a unique pressure $p \in L_p(J,\,\hat{H}^1_p(\Omega))$
within the quotient space $\hat{H}^1_p(\Omega) := H^1_p(\Omega) / \bR$ for $1 < p < \infty$.
Now, standard trace theory leads to velocity traces at time $t = 0$, and on smooth parts $\Sigma \subseteq \partial \Omega$ of the boundary within
the scale $[\,W^s_p(J,\,\cdot\,),\ W^s_p(\Omega),\ W^s_p(\Sigma)\,:\,s \geq 0,\ 1 < p < \infty\,]$ of (vector-valued) Sobolev-Slobodeckij spaces.
This implies regularity conditions for the initial velocity $u_0$, and the right-hand side $h$ of the boundary condition,
while the mapping properties of the operator $\mbox{div}$ imply a regularity condition for the right-hand side $g$ of the divergence equation.

Besides the obvious compatibility conditions between $g$ and $u_0$ as well as between $h$ and $u_0$, there is a hidden compatibility condition between $g$ and $h$.
To formulate this condition we argue as in \cite[Section~2]{Bothe-Koehne-Pruess:Energy-Preserving-Boundary-Conditions}:
For $\Omega = \bR^n_+$, and $\Gamma = \partial \Omega$ we define a linear functional $F(\psi,\,\eta)$ for $\psi \in L_p(\Omega)$, and $\eta \in L_p(\Gamma)$ as
\shrinkdisplayskips
\begin{equation*}
	\langle \phi,\,F(\psi,\,\eta) \rangle := \int_\Gamma [\phi]_\Gamma\,\eta\,\mbox{d}\sigma - \int_\Omega \phi\,\psi\,\mbox{d}x, \qquad \phi \in H^1_{p^\prime}(\Omega),
\end{equation*}
\unshrinkdisplayskips
where $1 < p^\prime < \infty$ with $\frac{1}{p} + \frac{1}{p^\prime} = 1$, and $[\,\cdot\,]_\Gamma$ denotes the trace of a quantity defined in $\Omega$ on the boundary $\Gamma$.
Then we have
\shrinkdisplayskips
\begin{equation*}
	\langle \phi,\,F(\mbox{div}\,u,\,[u]_\Gamma \cdot \nu) \rangle = \int_\Omega \nabla \phi \cdot u\,\mbox{d}x, \qquad \phi \in H^1_{p^\prime}(\Omega),
\end{equation*}
\unshrinkdisplayskips
which implies
\begin{equation*}
	|\langle \phi,\,\partial^m_t F(\mbox{div}\,u,\,[u]_\Gamma \cdot \nu) \rangle| \leq \|\partial^m_t u\|_{L_p(J \times \Omega)^n}\,\|\nabla \phi\|_{L_{p^\prime}(\Omega)^n}, \qquad \phi \in H^1_{p^\prime}(\Omega),
\end{equation*}
for $m = 0,\,1$.
Since a solution $u$ to $(\textrm{S})^{f, g, u_0}_{\textrm{\upshape Re}}$ satisfies $\mbox{div}\,u = g$,
this leads to a compatibility condition between $g$ and $[u]_\Gamma \cdot \nu$, which may be a prescribed quantity depending on the boundary condition.
To be precise, we have
\begin{equation*}
	F(g,\,[u]_\Gamma \cdot \nu) \in H^1_p(J,\,\hat{H}^{-1}_p(\Omega))
\end{equation*}
with $\hat{H}^{-1}_p(\Omega) := (H^1_{p^\prime}(\Omega),\,|\,\!\cdot\,\!|_{\dot{H}^1_{p^\prime}(\Omega)})^\prime$, and
$|\,\!\cdot\,\!|_{\dot{H}^1_p(\Omega)} = \|\nabla\,\!\cdot\,\!\|_{L_p(\Omega)^n}$.
Analogously, for $\Omega = \bR^n_{+\!\!\!+}$, $\Sigma = \partial_y \bR^n_{+\!\!\!+}$, and $\Gamma = \partial_z \bR^n_{+\!\!\!+}$
we define the linear functional $F(\psi,\,\eta_\Sigma,\,\eta_\Gamma)$ for $\psi \in L_p(\Omega)$, $\eta_\Sigma \in L_p(\Sigma)$, and $\eta_\Gamma \in L_p(\Gamma)$ as
\begin{equation*}
	\langle \phi,\,F(\psi,\,\eta_\Sigma,\,\eta_\Gamma) \rangle := \int_\Sigma [\phi]_\Sigma\,\eta_\Sigma\,\mbox{d}\sigma + \int_\Gamma [\phi]_\Gamma\,\eta_\Gamma\,\mbox{d}\sigma - \int_\Omega \phi\,\psi\,\mbox{d}x, \qquad \phi \in H^1_{p^\prime}(\Omega),
\end{equation*}
and obtain
\begin{equation*}
	|\langle \phi,\,\partial^m_t F(\mbox{div}\,u,\,[u]_\Sigma \cdot \nu,\,[u]_\Gamma \cdot \nu) \rangle| \leq \|\partial^m_t u\|_{L_p(J \times \Omega)^n}\,\|\nabla \phi\|_{L_{p^\prime}(\Omega)^n}, \quad \phi \in H^1_{p^\prime}(\Omega),
\end{equation*}
for $m = 0,\,1$.
As above, this leads to a compatibility condition between $g$, $[u]_\Sigma \cdot \nu$ and $[u]_\Gamma \cdot \nu$,
which may be prescribed quantities depending on the boundary condition.
In this case we have
\begin{equation*}
	F(g,\,[u]_\Sigma \cdot \nu,\,[u]_\Gamma \cdot \nu) \in H^1_p(J,\,\hat{H}^{-1}_p(\Omega))
\end{equation*}
with $\hat{H}^{-1}_p(\Omega) := (H^1_{p^\prime}(\Omega),\,|\,\!\cdot\,\!|_{\dot{H}^1_{p^\prime}(\Omega)})^\prime$ as above.

Finally, for the model problems in $\Omega = \bR^n_{+\!\!\!+}$ there are compatibility conditions
on the edge $\cE = \overline{\partial_y \bR^n_{+\!\!\!+}} \cap \overline{\partial_z \bR^n_{+\!\!\!+}}$ which have to be satisfied
by the right-hand sides of the boundary conditions.
First, if we impose a Navier condition $(\textrm{W})^{h^{\textrm{wall}}}_{\sigma, \textrm{\upshape Re}}$ on $\Gamma^{\textrm{wall}} = \partial_y \bR^n_{+\!\!\!+}$
in combination with an inflow condition $(\textrm{IF})^{u^{\textrm{in}}}$ on $\Gamma^{\textrm{in}} = \partial_z \bR^n_{+\!\!\!+}$,
then we necessarily have
\begin{equation*}
	\eqnlabel{Inflow-Wall-Compatibility}\tag*{$(\textrm{\upshape IF/W})^{u^{\textrm{in}}, h^{\textrm{wall}}}_{\sigma, \textrm{\upshape Re}}$}
	\begin{array}{rclll}
		                                 \sigma P_\cE u^{\textrm{in}} + \frac{1}{\textrm{\upshape Re}} \partial_{\nu_{\Gamma^{\textrm{wall}}}} (P_\cE u^{\textrm{in}}) + \frac{1}{\textrm{\upshape Re}} \nabla_\cE (h^{\textrm{wall}} \cdot \nu_{\Gamma^{\textrm{wall}}}) & = & P_\cE h^{\textrm{wall}}                              & \quad \mbox{on} & J \times \cE, \\[0.5em]
		                                                                                                                                                                                                               u^{\textrm{in}} \cdot \nu_{\Gamma^{\textrm{wall}}} & = & h^{\textrm{wall}} \cdot \nu_{\Gamma^{\textrm{wall}}} & \quad \mbox{on} & J \times \cE, \\[0.5em]
		\sigma u^{\textrm{in}} \cdot \nu_\Gamma + \frac{1}{\textrm{\upshape Re}} \partial_{\nu_{\Gamma^{\textrm{wall}}}} (u^{\textrm{in}} \cdot \nu_\Gamma) + \frac{1}{\textrm{\upshape Re}} \partial_{\nu_\Gamma} (h^{\textrm{wall}} \cdot \nu_{\Gamma^{\textrm{wall}}}) & = & h^{\textrm{wall}} \cdot \nu_\Gamma                   & \quad \mbox{on} & J \times \cE,
	\end{array}
\end{equation*}
where we denote by $P_\cE$ the projection onto the tangent bundle of $\cE$,
and by $\nabla_\cE$ the surface gradient.
Note that this is a simplified form of the necessary compatibility conditions which is valid for the simple geometry of the wedge $\bR^n_{+\!\!\!+}$.
For a generic weakly singular domain additional curvature related terms appear in the first and last lines
which stem from tangential derivatives of the normal fields $\nu_\Sigma$, and $\nu_\Gamma$.

Second, if we impose a Navier condition $(\textrm{W})^{h^{\textrm{wall}}}_{\sigma, \textrm{\upshape Re}}$ on $\Gamma^{\textrm{wall}} = \partial_y \bR^n_{+\!\!\!+}$
in combination with a dynamic outflow boundary condition $(\textrm{BDO})^{v^{\textrm{out}}, h}_{\alpha, \textrm{\upshape Re}}$
with $B \in \{\,T,\,N,\,F\,\}$ on $\Gamma = \partial_z \bR^n_{+\!\!\!+}$, then we necessarily have an analogous compatibility condition,
where, however, the velocity profile on $\Gamma$ is (in part) not prescribed.
For $B = T$ this leads to
\begin{equation*}
	\eqnlabel{Non-Reflecting-Wall-Compatibility-Tangential}\tag*{$(\textrm{\upshape TDO/W})^{h, h^{\textrm{wall}}, \xi}$}
	\begin{array}{rclll}
		                             \sigma P_\cE \xi + \frac{1}{\textrm{\upshape Re}} \partial_{\nu_{\Gamma^{\textrm{wall}}}} (P_\cE \xi) + \frac{1}{\textrm{\upshape Re}} \nabla_\cE (h^{\textrm{wall}} \cdot \nu_{\Gamma^{\textrm{wall}}}) & = & P_\cE h^{\textrm{wall}}                              & \quad \mbox{on} & J \times \cE, \\[0.5em]
		                                                                                                                                                                                               \xi \cdot \nu_{\Gamma^{\textrm{wall}}} & = & h^{\textrm{wall}} \cdot \nu_{\Gamma^{\textrm{wall}}} & \quad \mbox{on} & J \times \cE, \\[0.5em]
		\sigma h \cdot \nu_\Gamma + \frac{1}{\textrm{\upshape Re}} \partial_{\nu_{\Gamma^{\textrm{wall}}}} (h \cdot \nu_\Gamma) + \frac{1}{\textrm{\upshape Re}} \partial_{\nu_\Gamma} (h^{\textrm{wall}} \cdot \nu_{\Gamma^{\textrm{wall}}}) & = & h^{\textrm{wall}} \cdot \nu_\Gamma                   & \quad \mbox{on} & J \times \cE, \\[0.5em]
		                                                                          \alpha (\partial_t (h^{\textrm{wall}} \cdot \nu_{\Gamma^{\textrm{wall}}}) + V \partial_{\nu_\Gamma} (h^{\textrm{wall}} \cdot \nu_{\Gamma^{\textrm{wall}}}))                                                                                              \\[0.5em]
		                                                                                                                                                                 + \ h^{\textrm{wall}} \cdot \nu_\Gamma - \sigma (h \cdot \nu_\Gamma) & = & h \cdot \nu_{\Gamma^{\textrm{wall}}}                 & \quad \mbox{on} & J \times \cE
	\end{array}
\end{equation*}
for some function
\begin{equation*}
	\xi \in W^{3/2 - 1/2p}_p(J,\,L_p(\Gamma,\,T\Gamma)) \cap H^1_p(J,\,W^{1 - 1/p}_p(\Gamma,\,T\Gamma)) \cap L_p(J,\,W^{2 - 1/p}_p(\Gamma,\,T\Gamma))
\end{equation*}
that is compatible with $u_0$.
For $B = N$ we obtain
\begin{equation*}
	\eqnlabel{Non-Reflecting-Wall-Compatibility-Normal}\tag*{$(\textrm{\upshape NDO/W})^{h, h^{\textrm{wall}}, \eta}$}
	\begin{array}{rclll}
		   \sigma P_\cE h + \frac{1}{\textrm{\upshape Re}} \partial_{\nu_{\Gamma^{\textrm{wall}}}} (P_\cE h) + \frac{1}{\textrm{\upshape Re}} \nabla_\cE (h^{\textrm{wall}} \cdot \nu_{\Gamma^{\textrm{wall}}}) & = & P_\cE h^{\textrm{wall}}                              & \quad \mbox{on} & J \times \cE, \\[0.5em]
		                                                                                                                                                                   h \cdot \nu_{\Gamma^{\textrm{wall}}} & = & h^{\textrm{wall}} \cdot \nu_{\Gamma^{\textrm{wall}}} & \quad \mbox{on} & J \times \cE, \\[0.5em]
		\sigma \eta + \frac{1}{\textrm{\upshape Re}} \partial_{\nu_{\Gamma^{\textrm{wall}}}} \eta + \frac{1}{\textrm{\upshape Re}} \partial_{\nu_\Gamma} (h^{\textrm{wall}} \cdot \nu_{\Gamma^{\textrm{wall}}}) & = & h^{\textrm{wall}} \cdot \nu_\Gamma                   & \quad \mbox{on} & J \times \cE
	\end{array}
\end{equation*}
for some function
\begin{equation*}
	\eta \in H^1_p(J,\,W^{1 - 1/p}_p(\Gamma)) \cap L_p(J,\,W^{2 - 1/p}_p(\Gamma))
\end{equation*}
that is compatible with $g$, and $u_0$.
For $B = F$ we have
\begin{equation*}
	\eqnlabel{Non-Reflecting-Wall-Compatibility-Full}\tag*{$(\textrm{\upshape FDO/W})^{h, h^{\textrm{wall}}, \xi, \eta}$}
	\begin{array}{rclll}
		\sigma P_\cE \xi + \frac{1}{\textrm{\upshape Re}} \partial_{\nu_{\Gamma^{\textrm{wall}}}} (P_\cE \xi) + \frac{1}{\textrm{\upshape Re}} \nabla_\cE (h^{\textrm{wall}} \cdot \nu_{\Gamma^{\textrm{wall}}}) & = & P_\cE h^{\textrm{wall}}                              & \quad \mbox{on} & J \times \cE, \\[0.5em]
		                                                                                                                                                                  \xi \cdot \nu_{\Gamma^{\textrm{wall}}} & = & h^{\textrm{wall}} \cdot \nu_{\Gamma^{\textrm{wall}}} & \quad \mbox{on} & J \times \cE, \\[0.5em]
		 \sigma \eta + \frac{1}{\textrm{\upshape Re}} \partial_{\nu_{\Gamma^{\textrm{wall}}}} \eta + \frac{1}{\textrm{\upshape Re}} \partial_{\nu_\Gamma} (h^{\textrm{wall}} \cdot \nu_{\Gamma^{\textrm{wall}}}) & = & h^{\textrm{wall}} \cdot \nu_\Gamma                   & \quad \mbox{on} & J \times \cE, \\[0.5em]
		                                             \alpha (\partial_t (h^{\textrm{wall}} \cdot \nu_{\Gamma^{\textrm{wall}}}) + V \partial_{\nu_\Gamma} (h^{\textrm{wall}} \cdot \nu_{\Gamma^{\textrm{wall}}}))                                                                                              \\[0.5em]
		                                                                                                                                                    + \ h^{\textrm{wall}} \cdot \nu_\Gamma - \sigma \eta & = & h \cdot \nu_{\Gamma^{\textrm{wall}}}                 & \quad \mbox{on} & J \times \cE
	\end{array}
\end{equation*}
for some functions $\xi$, and $\eta$ as above.
Again these are simplified forms of the necessary compatibility conditions which are valid for the simple geometry of the wedge $\bR^n_{+\!\!\!+}$
and have to be modified for a generic weakly singular domain by additional curvature related terms.

\subsection{Main Results}
With the above preparations, we now formulate our main results,
the proofs of which are carried out in Sections~\ref{sec:Halfspace} and \ref{sec:Wedge}.
\begin{theorem}
	\thmlabel{Linear-Halfspace}
	Let $a > 0$, let $J := (0,\,a)$ and let $\Omega = \bR^n_+$ with $\Gamma := \partial \Omega$.
	Let $1 < p < \infty$ with $p \neq \frac{3}{2},\,3$.
	Moreover, let $B \in \{\,T,\,N,\,F\,\}$, and let $\alpha,\,\textrm{\upshape Re} > 0$.
	Furthermore, let $v^{\textrm{out}} = V \nu$ with $V > - \frac{1}{\alpha \textrm{\upshape Re}}$, and let
	\begin{itemize}
		\item $f \in L_p(J \times \Omega)^n$,
		\item $g \in H^{1/2}_p(J,\,L_p(\Omega)) \cap L_p(J,\,H^1_p(\Omega))$,
		\item $h \in L_p(J,\,W^{1 - 1/p}_p(\Gamma))^n$,
		\item $u_0 \in W^{2 - 2/p}_p(\Omega)^n$ with $\textrm{\upshape div}\,u_0 = g(0)$ in $\Omega$ for $p \geq 2$.
	\end{itemize}
	If $B = T$, let
	\begin{itemize}
		\item $P_\Gamma h \in W^{1/2 - 1/2p}_p(J,\,L_p(\Gamma))^n$, $P_\Gamma [u_0]_\Gamma \in W^{2 - 2/p}_p(\Gamma)^n$,
		\item $h \cdot \nu \in W^{1 - 1/2p}_p(J,\,L_p(\Gamma)) \cap L_p(J,\,W^{2 - 1/p}_p(\Gamma))$,
		\item $F(g,\,h \cdot \nu) \in H^1_p(J,\,\hat{H}^{-1}_p(\Omega))$,
		\item $[u_0]_\Gamma \cdot \nu = h(0) \cdot \nu$ for $p > \frac{3}{2}$;
	\end{itemize}
	if $B = N$, let
	\begin{itemize}
		\item $P_\Gamma h \in W^{1 - 1/2p}_p(J,\,L_p(\Gamma))^n \cap L_p(J,\,W^{2 - 1/p}_p(\Gamma))^n$,
		\item $F(g,\,\eta) \in H^1_p(J,\,\hat{H}^{-1}_p(\Omega))$ for some \\
			$\eta \in H^1_p(J,\,W^{1 - 1/p}_p(\Gamma)) \cap L_p(J,\,W^{2 - 1/p}_p(\Gamma))$
			with $[u_0]_\Gamma \cdot \nu = \eta(0)$ for $p > \frac{3}{2}$,
		\item $P_\Gamma [u_0]_\Gamma = P_\Gamma h(0) \cdot \nu$ for $p > \frac{3}{2}$;
	\end{itemize}
	if $B = F$, let
	\begin{itemize}
		\item $P_\Gamma h \in W^{1/2 - 1/2p}_p(J,\,L_p(\Gamma))^n$, $P_\Gamma [u_0]_\Gamma \in W^{2 - 2/p}_p(\Gamma)^n$,
		\item $F(g,\,\eta) \in H^1_p(J,\,\hat{H}^{-1}_p(\Omega))$ for some \\
			$\eta \in H^1_p(J,\,W^{1 - 1/p}_p(\Gamma)) \cap L_p(J,\,W^{2 - 1/p}_p(\Gamma))$
			with $[u_0]_\Gamma \cdot \nu = \eta(0)$ for $p > \frac{3}{2}$;
	\end{itemize}
	Then the system \eqnref*{Interior-Linear}, $(\textrm{\upshape BDO})^{v^{\textrm{out}}, h}_{\alpha, \textrm{\upshape Re}}$ admits a unique maximal regular solution
	\begin{itemize}
		\item $u \in H^1_p(J,\,L_p(\Omega))^n \cap L_p(J,\,H^2_p(\Omega))^n$,
		\item $p \in L_p(J,\,\hat{H}^1_p(\Omega))$ for $B = T$, or \\
			$p \in L_p(J,\,\dot{H}^1_p(\Omega))$ with $[p]_\Gamma \in L_p(J,\,W^{1 - 1/p}_p(\Gamma))$ for $B \in \{\,N,\,F\,\}$.
	\end{itemize}
	If $B \in \{\,T,\,F\,\}$, then we additionally have
	\begin{itemize}
		\item $P_\Gamma [u]_\Gamma \in W^{3/2 - 1/2p}_p(J,\,L_p(\Gamma))^n \cap H^1_p(J,\,W^{1 - 1/p}_p(\Gamma))^n \cap L_p(J,\,W^{2 - 1/p}_p(\Gamma))^n$;
	\end{itemize}
	if $B \in \{\,N,\,F\,\}$, then we additionally have
	\begin{itemize}
		\item $[u]_\Gamma \cdot \nu \in H^1_p(J,\,W^{1 - 1/p}_p(\Gamma)) \cap L_p(J,\,W^{2 - 1/p}_p(\Gamma))$.
	\end{itemize}
	The solutions depend continuously on the data in the corresponding spaces.
\end{theorem}
The proof of \Thmref{Linear-Halfspace}, which is based on a precise analysis of the corresponding boundary symbols, is carried out in \Secref{Halfspace}.
Here, however, some remarks seem to be in order.
\begin{remark}
	\remlabel{Linear-Halfspace}
	There are some immediate corollaries of \Thmref{Linear-Halfspace},
	which we want to mention without elaborate proofs.
	\begin{enumerate}[(a)]
		\item If $B \in \{\,N,\,F\,\}$, then the assumptions on the right-hand side of the boundary condition may be relaxed to
			$h \cdot \nu \in L_p(J,\,\dot{W}^{1 - 1/p}_p(\Gamma))$ to obtain a maximal regular solution
			as in \Thmref{Linear-Halfspace} with $[p]_\Gamma \in L_p(J,\,\dot{W}^{1 - 1/p}_p(\Gamma))$.
			Indeed, one first constructs an auxiliary pressure $q \in L_p(J,\,\dot{H}^1_p(\Omega))$ as a weak solution to
			\begin{equation*}
				\begin{array}{rclll}
					- \Delta q & = & 0             & \qquad \mbox{in} & J \times \Omega, \\[0.5em]
					         q & = & - h \cdot \nu & \qquad \mbox{on} & J \times \Gamma,
				\end{array}
			\end{equation*}
			and then solves $(\textrm{S})^{f^\prime, g, u_0}_{\textrm{\upshape Re}}$, $(\textrm{BDO})^{v^{\textrm{out}}, h^\prime}_{\alpha, \textrm{\upshape Re}}$ via \Thmref{Linear-Halfspace}
			with the adjusted data $f^\prime = f - \nabla q$, $P_\Gamma h^\prime = P_\Gamma h$, and $h^\prime \cdot \nu = 0$ to obtain a solution $(u^\prime,\,p^\prime)$ in the maximal regularity class.
			Then $u = u^\prime$, $p = p^\prime + q$ constitutes the unique maximal regular solution to the model problem with relaxed regularity assumptions.
			Conversely the relaxed version of \Thmref{Linear-Halfspace} obviously implies \Thmref{Linear-Halfspace}, i.\,e.\ both formulations of the theorem are equivalent.
		\item One may assume $v^{\textrm{out}}$ to be given based on
			\begin{equation*}
				V \in W^{1 - 1/2p}_p(J,\,L_p(\Gamma)) \cap L_p(J,\,W^{2 - 1/p}_p(\Gamma))
			\end{equation*}
			such that \eqnref{Parameter-Compatibility} is satisfied.
			Indeed, this problem may be reduced to \Thmref{Linear-Halfspace} via a localization procedure.
	\end{enumerate}
	Of course, Corollaries (a) and (b) are independent of each other and may be applied simultaneously.
\end{remark}
\begin{remark}
	\remlabel{Non-Linear-Halfspace}
	\Thmref{Linear-Halfspace} and its variants in \Remref{Linear-Halfspace} are the cornerstones to obtain corresponding results for bounded, smooth domains $\Omega \subseteq \bR^n$
	via well-known localization procedures as presented e.\,g.\ in \cite{Bothe-Koehne-Pruess:Energy-Preserving-Boundary-Conditions}.
	Based on well-known perturbation arguments, it is then also possible to obtain (local-in-time) strong solutions
	to the corresponding non-linear equations \eqnref*{Interior} with non-linear variants of the dynamic outflow boundary conditions.
\end{remark}
\begin{theorem}
	\thmlabel{Linear-Wedge-Inflow}
	Let $a > 0$, let $J := (0,\,a)$ and let $\Omega = \bR^n_{+\!\!\!+}$
	with $\Gamma^{\textrm{\upshape wall}} := \partial_y \bR^n_{+\!\!\!+}$, and $\Gamma^{\textrm{\upshape in}} := \partial_z \bR^n_{+\!\!\!+}$.
	Let $1 < p < \infty$ with $p \neq \frac{3}{2},\,3$.
	Moreover, let $\sigma \geq 0$, and let
	\begin{itemize}
		\item $f \in L_p(J \times \Omega)^n$,
		\item $g \in H^{1/2}_p(J,\,L_p(\Omega)) \cap L_p(J,\,H^1_p(\Omega))$,
		\item $u^{\textrm{\upshape in}} \in W^{1 - 1/2p}_p(J,\,L_p(\Gamma^{\textrm{\upshape in}}))^n \cap L_p(J,\,W^{2 - 1/p}_p(\Gamma^{\textrm{\upshape in}}))^n$,
		\item $h^{\textrm{\upshape wall}} \in W^{1/2 - 1/2p}_p(J,\,L_p(\Gamma^{\textrm{\upshape wall}}))^n \cap L_p(J,\,W^{1 - 1/p}_p(\Gamma^{\textrm{\upshape wall}}))^n$,
		\item $h^{\textrm{\upshape wall}} \cdot \nu \in W^{1 - 1/2p}_p(J,\,L_p(\Gamma^{\textrm{\upshape wall}})) \cap L_p(J,\,W^{2 - 1/p}_p(\Gamma^{\textrm{\upshape wall}}))$,
		\item $u_0 \in W^{2 - 2/p}_p(\Omega)^n$ with $\textrm{\upshape div}\,u_0 = g(0)$ in $\Omega$ for $p \geq 2$, and \\
			$[u_0]_{\Gamma^{\textrm{\upshape in}}} = u^{\textrm{\upshape in}}(0)$ as well as
			$[u_0]_{\Gamma^{\textrm{\upshape wall}}} \cdot \nu = h^{\textrm{\upshape wall}}(0) \cdot \nu$ for $p > \frac{3}{2}$, and \\
			$\sigma P_\Gamma [u_0]_{\Gamma^{\textrm{\upshape wall}}} + \frac{2}{\textrm{\upshape Re}} P_\Gamma [D_0]_{\Gamma^{\textrm{\upshape wall}}}\,\nu = P_\Gamma h^{\textrm{wall}}(0)$ for $p > 3$.
	\end{itemize}
	Furthermore, let the compatibility condition \eqnref*{Inflow-Wall-Compatibility} be satisfied for $p \geq 2$.
	Then the system \eqnref*{Interior-Linear}, \eqnref*{Inflow-BC}, \eqnref*{Wall-BC-Linear} admits a unique maximal regular solution
	\begin{itemize}
		\item $u \in H^1_p(J,\,L_p(\Omega))^n \cap L_p(J,\,H^2_p(\Omega))^n$,
		\item $p \in L_p(J,\,\hat{H}^1_p(\Omega))$.
	\end{itemize}
	The solutions depend continuously on the data in the corresponding spaces.
\end{theorem}
The proof of \Thmref{Linear-Wedge-Inflow}, which is based on a reflection technique and \Thmref{Linear-Halfspace}, is carried out in \Secref{Wedge}.
Here, however, we have to compare it with known results. \pagebreak
\begin{remark}
	\remlabel{Linear-Wedge-Inflow}
	\Thmref{Linear-Wedge-Inflow} is contained as a special case in \cite[Theorem~8.24]{Koehne:Incompressible-Newtonian-Flows}.
	However, in order to keep this paper self-contained we give a short proof of \Thmref{Linear-Wedge-Inflow} in \Secref{Wedge},
	which is different (shorter and more descriptive) from that presented in \cite{Koehne:Incompressible-Newtonian-Flows},
	since we restrict our considerations to a special combination of boundary conditions.
\end{remark}
\begin{theorem}
	\thmlabel{Linear-Wedge-Non-Reflecting}
	Let $a > 0$, let $J := (0,\,a)$ and let $\Omega = \bR^n_{+\!\!\!+}$
	with $\Gamma^{\textrm{\upshape wall}} := \partial_y \bR^n_{+\!\!\!+}$, and $\Gamma := \partial_z \bR^n_{+\!\!\!+}$.
	Let $1 < p < \infty$ with $p \neq \frac{3}{2},\,3$.
	Moreover, let $\sigma \geq 0$, let $B \in \{\,T,\,N,\,F\,\}$, and let $\alpha,\,\textrm{\upshape Re} > 0$.
	Furthermore, let $v^{\textrm{out}} = V \nu$ with $V > - \frac{1}{\alpha \textrm{\upshape Re}}$, and let
	\begin{itemize}
		\item $f \in L_p(J \times \Omega)^3$,
		\item $g \in H^{1/2}_p(J,\,L_p(\Omega)) \cap L_p(J,\,H^1_p(\Omega))$,
		\item $h \in L_p(J,\,W^{1 - 1/p}_p(\Gamma))^3$,
		\item $h^{\textrm{\upshape wall}} \in W^{1/2 - 1/2p}_p(J,\,L_p(\Gamma^{\textrm{\upshape wall}}))^n \cap L_p(J,\,W^{1 - 1/p}_p(\Gamma^{\textrm{\upshape wall}}))^n$,
		\item $h^{\textrm{\upshape wall}} \cdot \nu \in W^{1 - 1/2p}_p(J,\,L_p(\Gamma^{\textrm{\upshape wall}})) \cap L_p(J,\,W^{2 - 1/p}_p(\Gamma^{\textrm{\upshape wall}}))$,
		\item $u_0 \in W^{2 - 2/p}_p(\Omega)^n$ with $\textrm{\upshape div}\,u_0 = g(0)$ in $\Omega$ for $p \geq 2$, and \\
			$[u_0]_{\Gamma^{\textrm{\upshape wall}}} \cdot \nu = h^{\textrm{\upshape wall}}(0) \cdot \nu$ for $p > \frac{3}{2}$, and \\
			$\sigma P_\Gamma [u_0]_{\Gamma^{\textrm{\upshape wall}}} + \frac{2}{\textrm{\upshape Re}} P_\Gamma [D_0]_{\Gamma^{\textrm{\upshape wall}}}\,\nu = P_\Gamma h^{\textrm{wall}}(0)$ for $p > 3$.
	\end{itemize}
	If $B = T$, let
	\begin{itemize}
		\item $P_\Gamma h \in W^{1/2 - 1/2p}_p(J,\,L_p(\Gamma))^n$, $P_\Gamma [u_0]_\Gamma \in W^{2 - 2/p}_p(\Gamma)^n$,
		\item $h \cdot \nu \in W^{1 - 1/2p}_p(J,\,L_p(\Gamma)) \cap L_p(J,\,W^{2 - 1/p}_p(\Gamma))$,
		\item $F(g,\,h^{\textrm{\upshape wall}} \cdot \nu,\,h \cdot \nu) \in H^1_p(J,\,\hat{H}^{-1}_p(\Omega))$,
		\item $[u_0]_\Gamma \cdot \nu = h(0) \cdot \nu$ for $p > \frac{3}{2}$;
		\item the compatibility condition \eqnref*{Non-Reflecting-Wall-Compatibility-Tangential} be satisfied for $p \geq 2$ for some \\
			$\xi \in W^{3/2 - 1/2p}_p(J,\,L_p(\Gamma,\,T\Gamma)) \cap H^1_p(J,\,W^{1 - 1/p}_p(\Gamma,\,T\Gamma)) \cap L_p(J,\,W^{2 - 1/p}_p(\Gamma,\,T\Gamma))$ \\
			with $P_\Gamma[u_0]_\Gamma = \xi(0)$ for $p > \frac{3}{2}$;
	\end{itemize}
	if $B = N$, let
	\begin{itemize}
		\item $P_\Gamma h \in W^{1 - 1/2p}_p(J,\,L_p(\Gamma))^n \cap L_p(J,\,W^{2 - 1/p}_p(\Gamma))^n$, $P_\Gamma [u_0]_\Gamma = P_\Gamma h$ for $p > \frac{3}{2}$,
		\item $F(g,\,h^{\textrm{\upshape wall}} \cdot \nu,\,\eta) \in H^1_p(J,\,\hat{H}^{-1}_p(\Omega))$ for some \\
			$\eta \in H^1_p(J,\,W^{1 - 1/p}_p(\Gamma)) \cap L_p(J,\,W^{2 - 1/p}_p(\Gamma))$
			with $[u_0]_\Gamma \cdot \nu = \eta(0)$ for $p > \frac{3}{2}$,
		\item the compatibility condition \eqnref*{Non-Reflecting-Wall-Compatibility-Normal} be satisfied for $p \geq 2$;
	\end{itemize}
	if $B = F$, let
	\begin{itemize}
		\item $P_\Gamma h \in W^{1/2 - 1/2p}_p(J,\,L_p(\Gamma))^n$, $P_\Gamma [u_0]_\Gamma \in W^{2 - 2/p}_p(\Gamma)^n$,
		\item $F(g,\,h^{\textrm{\upshape wall}} \cdot \nu,\,\eta) \in H^1_p(J,\,\hat{H}^{-1}_p(\Omega))$ for some \\
			$\eta \in H^1_p(J,\,W^{1 - 1/p}_p(\Gamma)) \cap L_p(J,\,W^{2 - 1/p}_p(\Gamma))$
			with $[u_0]_\Gamma \cdot \nu = \eta(0)$ for $p > \frac{3}{2}$,
		\item the compatibility condition \eqnref*{Non-Reflecting-Wall-Compatibility-Full} be satisfied for $p \geq 2$ for some \\
			$\xi \in W^{3/2 - 1/2p}_p(J,\,L_p(\Gamma,\,T\Gamma)) \cap H^1_p(J,\,W^{1 - 1/p}_p(\Gamma,\,T\Gamma)) \cap L_p(J,\,W^{2 - 1/p}_p(\Gamma,\,T\Gamma))$ \\
			with $P_\Gamma[u_0]_\Gamma = \xi(0)$ for $p > \frac{3}{2}$;
	\end{itemize}
	Then the system \eqnref*{Interior-Linear}, $(\textrm{\upshape BDO})^{v^{\textrm{out}}, h}_{\alpha, \textrm{\upshape Re}}$, \eqnref*{Wall-BC-Linear} admits a unique maximal regular solution
	\begin{itemize}
		\item $u \in H^1_p(J,\,L_p(\Omega))^n \cap L_p(J,\,H^2_p(\Omega))^n$,
		\item $p \in L_p(J,\,\hat{H}^1_p(\Omega))$ for $B = T$, or \\
			$p \in L_p(J,\,\dot{H}^1_p(\Omega))$ with $[p]_\Gamma \in L_p(J,\,W^{1 - 1/p}_p(\Gamma))$ for $B \in \{\,N,\,F\,\}$.
	\end{itemize}
	If $B \in \{\,T,\,F\,\}$, then we additionally have
	\begin{itemize}
		\item $P_\Gamma [u]_\Gamma \in W^{3/2 - 1/2p}_p(J,\,L_p(\Gamma))^n \cap H^1_p(J,\,W^{1 - 1/p}_p(\Gamma))^n \cap L_p(J,\,W^{2 - 1/p}_p(\Gamma))^n$;
	\end{itemize}
	if $B \in \{\,N,\,F\,\}$, then we additionally have
	\begin{itemize}
		\item $[u]_\Gamma \cdot \nu \in H^1_p(J,\,W^{1 - 1/p}_p(\Gamma)) \cap L_p(J,\,W^{2 - 1/p}_p(\Gamma))$.
	\end{itemize}
	The solutions depend continuously on the data in the corresponding spaces.
\end{theorem}
The proof of \Thmref{Linear-Wedge-Inflow}, which is based on a reflection technique and Theorems~\ref{thm:Linear-Halfspace} and \ref{thm:Linear-Wedge-Inflow}, is carried out in \Secref{Wedge}.
Here, however, some remarks seem to be in order.
\begin{remark}
	\remlabel{Linear-Wedge-Non-Reflecting}
	Again there are some immediate corollaries of \Thmref{Linear-Wedge-Non-Reflecting},
	which we want to mention without elaborate proofs, cf.~\Remref{Linear-Halfspace}.
	\begin{enumerate}[(a)]
		\item If $B \in \{\,N,\,F\,\}$, then the assumptions on the right-hand side of the boundary condition may be relaxed to
			$h \cdot \nu \in L_p(J,\,\dot{W}^{1 - 1/p}_p(\Gamma))$ to obtain a maximal regular solution
			as in \Thmref{Linear-Wedge-Non-Reflecting} with $[p]_\Gamma \in L_p(J,\,\dot{W}^{1 - 1/p}_p(\Gamma))$.
			The argument here is the same as used in \Remref{Linear-Halfspace}~(a) and both formulations of \Thmref{Linear-Wedge-Non-Reflecting} are again equivalent.
		\item One may assume $v^{\textrm{out}}$ to be given based on
			\begin{equation*}
				V \in W^{1 - 1/2p}_p(J,\,L_p(\Gamma)) \cap L_p(J,\,W^{2 - 1/p}_p(\Gamma))
			\end{equation*}
			such that \eqnref{Parameter-Compatibility} is satisfied.
			Indeed, this problem may be reduced to \Thmref{Linear-Wedge-Non-Reflecting} via a localization procedure.
	\end{enumerate}
	Of course, Corollaries (a) and (b) are independent of each other and may be applied simultaneously.
\end{remark}
\begin{remark}
	\remlabel{Non-Linear-Wedge}
	Theorems~\ref{thm:Linear-Wedge-Inflow} and \ref{thm:Linear-Wedge-Non-Reflecting} and the variants in \Remref{Linear-Wedge-Non-Reflecting}
	are the cornerstones to handle realistic models in weakly singular domains $\Omega \subseteq \bR^n$ like the tube in Figure~1
	via localization procedures as presented e.\,g.\ in \cite[Chapter~8]{Koehne:Incompressible-Newtonian-Flows}.
	Based on well-known perturbation arguments, it is then also possible to obtain (local-in-time) strong solutions
	to the corresponding non-linear equations \eqnref*{Interior} with non-linear variants of the dynamic outflow boundary conditions.
\end{remark}

\section{The Halfspace Case}
\seclabel{Halfspace}
This section is devoted to the first step of the proof of \Thmref{Linear-Halfspace},
where the halfspace $\Omega := \bR^n_+$ is considered with $\Gamma := \partial \Omega$.
We assume $a > 0$, set $J := (0,\,a)$, and assume $1 < p < \infty$ with $p \neq \frac{3}{2},\,3$.
Furthermore, we assume $\alpha, \textrm{Re} > 0$ as well as $v^{\textrm{out}} = V \nu$
with $\sigma := \alpha V + \frac{2}{\textrm{Re}} > \kappa := \alpha V + \frac{1}{\textrm{Re}} > 0$.
We exploit the simple geometry of the halfspace and denote by $(x,\,y) \in \bR^{n - 1} \times \bR_+$ the generic point in $\bR^n_+$,
decomposed in its {\itshape tangential} part $x \in \bR^{n - 1}$ and its {\itshape normal} part $y > 0$.
Moreover, we employ the notation $u = (v,\,w)$ to decompose the unkown velocity field into its tangential part
$v: J \times \bR^n_+ \longrightarrow \bR^{n - 1}$ and its normal part $w: J \times \bR^n_+ \longrightarrow \bR$.
Finally, we denote by $[\,\cdot\,]_y: \bR^n_+ \longrightarrow \partial \bR^n_+$ the trace operator for the halfspace and frequently
employ the identification $\partial \bR^n_+ \simeq \bR^{n - 1}$, whenever this seems to be convenient.
The right hand side of the boundary condition is decomposed as $h = (h_v,\,h_w)$ into a tangential part $h_v$ and a normal part $h_w$.
The same splitting is employed for the initial velocity, where we let $u_0 = (v_0,\,w_0)$.

\subsection{The Condition TDO}
\subseclabel{Linear-Halfspace-Proof-Tangential}
We first consider the Stokes equations subject to a dynamic outflow boundary condition in tangential directions,
i.\,e.\ the system \eqnref*{Interior-Linear}, \eqnref*{Tangential-NRBC} which reads
\begin{equation}
	\eqnlabel{Linear-Halfspace-Tangential}
	\begin{array}{rclll}
		                                                            \partial_t u - \frac{1}{\textrm{Re}} \Delta u + \nabla p & = & f   & \quad \mbox{in} & J \times \bR^n_+,     \\[0.5em]
		                                                                                                       \mbox{div}\,u & = & g   & \quad \mbox{in} & J \times \bR^n_+,     \\[0.5em]
		\alpha \partial_t [v]_y - (\alpha V + \frac{1}{\textrm{Re}}) [\partial_y v]_y - \frac{1}{\textrm{Re}} \nabla_x [w]_y & = & h_v & \quad \mbox{on} & J \times \bR^{n - 1}, \\[0.5em]
		                                                                                                               [w]_y & = & h_w & \quad \mbox{on} & J \times \bR^{n - 1}, \\[0.5em]
		                                                                                                                u(0) & = & u_0 & \quad \mbox{in} & \bR^n_+.
	\end{array}
\end{equation}
Here, we require the data to satisfy the regularity and compatibility conditions as stated in \Thmref{Linear-Halfspace}, i.\,e.\ we have
\begin{itemize}
	\item $f \in L_p(J \times \bR^n_+)^n$,
	\item $g \in H^{1/2}_p(J,\,L_p(\bR^n_+)) \cap L_p(J,\,H^1_p(\bR^n_+))$,
	\item $h_v \in W^{1/2 - 1/2p}_p(J,\,L_p(\bR^{n - 1}))^{n - 1} \cap L_p(J,\,W^{1 - 1/p}_p(\bR^{n - 1}))^{n - 1}$,
	\item $h_w \in W^{1 - 1/2p}_p(J,\,L_p(\bR^{n - 1})) \cap L_p(J,\,W^{2 - 1/p}_p(\bR^{n - 1}))$,
	\item $u_0 \in W^{2 - 2/p}_p(\bR^n_+)^n$ with $\textrm{\upshape div}\,u_0 = g(0)$ in $\bR^n_+$ for $p \geq 2$,
	\item $F(g,\,- h_w) \in H^1_p(J,\,\hat{H}^{-1}_p(\bR^n_+))$,
	\item $[v_0]_y \in W^{2 - 2/p}_p(\bR^{n - 1})^{n - 1}$, and $[w_0]_y = h_w(0)$ for $p > \frac{3}{2}$.
\end{itemize}
The construction of a solution to \eqnref{Linear-Halfspace-Tangential} requires several Steps.

\subsection*{Step 1}
As a first step we show that we may w.\,l.\,o.\,g.\ assume $f = 0$, $g = 0$, $h_v = 0$ and $u_0 = 0$ in the following.
Indeed, based on the compatibility condition between $g$ and $- h_w$,
we may employ \cite[Proposition~3.6]{Bothe-Koehne-Pruess:Energy-Preserving-Boundary-Conditions}
to obtain $q \in L_p(J,\,{}_0 \dot{H}^1_p(\bR^n_+))$ such that $- \mbox{div}\,\nabla q = (\partial_t - \frac{1}{Re} \Delta) g$ in the sense of distributions.
Then we solve the parabolic system with dynamic boundary conditions
\begin{equation*}
	\begin{array}{rclll}
		                                \partial_t u - \frac{1}{\textrm{Re}} \Delta u & = & \cW_p f - \nabla q                         & \quad \mbox{in} & J \times \bR^n_+      \\[0.5em]
		\alpha \partial_t [v]_y - (\alpha V + \frac{1}{\textrm{Re}}) [\partial_y v]_y & = & h_v + \frac{1}{\textrm{Re}} \nabla_x [w]_y & \quad \mbox{on} & J \times \bR^{n - 1}, \\[0.5em]
		                                                             [\partial_y w]_y & = & [g]_y - \mbox{div}_x\,[v]_y                & \quad \mbox{on} & J \times \bR^{n - 1}, \\[0.5em]
		                                                                         u(0) & = & u_0                                        & \quad \mbox{in} & \bR^n_+

	\end{array}
\end{equation*}
to obtain a unique solution $u$ in the desired regularity class via \Propref{Parabolic-DBC-Standard}.
Here, we employ the Weyl projection $\cW: L_p(\bR^n_+)^n \longrightarrow L_p(\bR^n_+)$ that belongs to the topological decomposition
$L_p(\bR^n_+)^n = L_{p, s}(\bR^n_+) \oplus \nabla {}_0 \dot{H}^1_p(\bR^n_+)$ into
\begin{equation*}
	L_{p, s}(\bR^n_+) := \Big\{\,\phi \in L_p(\bR^n_+)^n\,:\,\mbox{div}\,\phi = 0\,\Big\}, \qquad
	{}_0 \dot{H}^1_p(\bR^n_+) := \Big\{\,\psi \in \dot{H}^1_p(\bR^n_+)\,:\,[\psi]_y = 0\,\Big\},
\end{equation*}
see e.\,g.\ \cite[Section~3]{Bothe-Koehne-Pruess:Energy-Preserving-Boundary-Conditions}.
If we then define $p \in L_p(J,\,{}_0 \dot{H}^1_p(\bR^n))$ via $\nabla p = \nabla q + (1 - \cW_p) f$, then
\begin{equation*}
	\begin{array}{rclll}
		                                                            \partial_t u - \frac{1}{\textrm{Re}} \Delta u + \nabla p & = & f   & \quad \mbox{in} & J \times \bR^n_+,     \\[0.5em]
		\alpha \partial_t [v]_y - (\alpha V + \frac{1}{\textrm{Re}}) [\partial_y v]_y - \frac{1}{\textrm{Re}} \nabla_x [w]_y & = & h_v & \quad \mbox{on} & J \times \bR^{n - 1}, \\[0.5em]
		                                                                                                                u(0) & = & u_0 & \quad \mbox{in} & \bR^n_+.
	\end{array}
\end{equation*}
Moreover, we have by construction
\begin{equation*}
	\begin{array}{rclll}
		\partial_t \gamma - \frac{1}{\textrm{Re}} \Delta \gamma & = & 0 & \quad \mbox{in} & J \times \bR^n_+,     \\[0.5em]
		                                             [\gamma]_y & = & 0 & \quad \mbox{on} & J \times \bR^{n - 1}, \\[0.5em]
		                                              \gamma(0) & = & 0 & \quad \mbox{in} & \bR^n_+
	\end{array}
\end{equation*}
for $\gamma = \mbox{div}\,u - g \in BC(J,\,W^{1 - 1/p}_p(\bR^n_+)) \hookrightarrow BC(J,\,L_p(\bR^n_+))$,
which implies $\gamma = 0$ by uniqueness of weak solutions to the diffusion equation with Dirichlet boundary condition,
see also the proof of \cite[Theorem~3.6]{Bothe-Koehne-Pruess:Energy-Preserving-Boundary-Conditions}.
Thus, $\mbox{div}\,u = g$.
Hence, we may assume $f = 0$, $g = 0$, $h_v = 0$ and $u_0 = 0$.
Note that in this case the compatibility condition between the right-hand side of the divergence equation
and the normal boundary condition implies
\begin{equation*}
	h_w \in {}_0 H^1_p(J,\,\dot{W}^{-1/p}_p(\bR^{n - 1})) \cap {}_0 W^{1 - 1/2p}_p(J,\,L_p(\bR^{n - 1})) \cap L_p(J,\,W^{2 - 1/p}_p(\bR^{n - 1})),
\end{equation*}
which we will assume from now on.

\subsection*{Step 2}
In order to solve the remaining problem, we will employ a Laplace transformation in time and a Fourier transformation in the tangential part of the spatial variables.
Since this is only possible for an unbounded time interval, we will from now on consider the shifted problem
\begin{equation}
	\eqnlabel{Linear-Halfspace-Tangential-Reduced}
	\begin{array}{rclll}
		                                                                       \epsilon u + \partial_t u - \frac{1}{\textrm{Re}} \Delta u + \nabla p & = & 0   & \quad \mbox{in} & \bR_+ \times \bR^n_+,     \\[0.5em]
		                                                                                                                               \mbox{div}\,u & = & 0   & \quad \mbox{in} & \bR_+ \times \bR^n_+,     \\[0.5em]
		\alpha \epsilon [v]_y + \alpha \partial_t [v]_y - (\alpha V + \frac{1}{\textrm{Re}}) [\partial_y v]_y - \frac{1}{\textrm{Re}} \nabla_x [w]_y & = & 0   & \quad \mbox{on} & \bR_+ \times \bR^{n - 1}, \\[0.5em]
		                                                                                                                                       [w]_y & = & h_w & \quad \mbox{on} & \bR_+ \times \bR^{n - 1}, \\[0.5em]
		                                                                                                                                        u(0) & = & 0   & \quad \mbox{in} & \bR^n_+.
	\end{array}
\end{equation}
for an arbitrary $\epsilon > 0$.
Note that maximal regularity for this problem is equivalent to maximal regularity of the original problem (i.\,e.\ for $\epsilon = 0$) on finite time intervals $J = (0,\,a)$.
The strategy to construct a solution to \eqnref{Linear-Halfspace-Tangential-Reduced} is as follows:
We compute the pressure derivative $- [\partial_y p]_y = \Pi h_w$ and show that it is given based on a bounded linear operator
\begin{equation}
	\eqnlabel{Halfspace-Mappings-Tangential}
	\begin{array}{l}
		\Pi: {}_0 H^1_p(\bR_+,\,\dot{W}^{-1/p}_p(\bR^{n - 1})) \cap {}_0 W^{1 - 1/2p}_p(\bR_+,\,L_p(\bR^{n - 1})) \cap L_p(\bR_+,\,W^{2 - 1/p}_p(\bR^{n - 1})) \\[0.5em]
			\qquad \qquad \qquad \qquad \longrightarrow L_p(\bR_+,\,\dot{W}^{-1/p}_p(\bR^{n - 1})).
	\end{array}
\end{equation}
Then we obtain the pressure $p \in L_p(\bR_+,\,\hat{H}^1_p(\bR^n_+))$ as a solution to the (weak) elliptic problem
\begin{equation*}
	\begin{array}{rclll}
		        - \Delta p & = & 0       & \quad \mbox{in} & \bR_+ \times \bR^n_+,     \\[0.5em]
		- [\partial_y p]_y & = & \Pi h_w & \quad \mbox{on} & \bR_+ \times \bR^{n - 1},
	\end{array}
\end{equation*}
cf.~\cite[Proposition~3.3]{Bothe-Koehne-Pruess:Energy-Preserving-Boundary-Conditions}.
Finally, we obtain $u$ as a maximal regular solution to the parabolic problem
\begin{equation*}
	\begin{array}{rclll}
		                                           \epsilon u + \partial_t u - \frac{1}{\textrm{Re}} \Delta u & = & - \nabla p                         & \quad \mbox{in} & \bR_+ \times \bR^n_+,     \\[0.5em]
		\alpha \epsilon [v]_y + \alpha \partial_t [v]_y - (\alpha V + \frac{1}{\textrm{Re}}) [\partial_y v]_y & = & \frac{1}{\textrm{Re}} \nabla_x h_w & \quad \mbox{on} & \bR_+ \times \bR^{n - 1}, \\[0.5em]
		                                                                                                [w]_y & = & h_w                                & \quad \mbox{on} & \bR_+ \times \bR^{n - 1}, \\[0.5em]
		                                                                                                 u(0) & = & 0                                  & \quad \mbox{in} & \bR^n_+
	\end{array}
\end{equation*}
via \Propref{Parabolic-DBC-Standard}.

\subsubsection*{Step 2.1}
We compute the Laplace-Fourier symbol of $\Pi$.
The transformed equations \eqnref{Linear-Halfspace-Tangential-Reduced} read:
\begin{equation*}
	\begin{array}{rcllll}
		                                 \omega^2 \hat v - \frac{1}{\textrm{Re}} \partial_y^2 \hat v + i\xi \hat p & = & 0        & \quad \lambda \in \Sigma_{\pi - \theta}, & \xi \in \bR^{n - 1}, & y > 0, \\[0.5em]
		                           \omega^2 \hat w - \frac{1}{\textrm{Re}} \partial_y^2 \hat w + \partial_y \hat p & = & 0        & \quad \lambda \in \Sigma_{\pi - \theta}, & \xi \in \bR^{n - 1}, & y > 0, \\[0.5em]
		                                                                    i \xi \cdot \hat v + \partial_y \hat w & = & 0        & \quad \lambda \in \Sigma_{\pi - \theta}, & \xi \in \bR^{n - 1}, & y > 0, \\[0.5em]
		\alpha \lambda_\epsilon [\hat v]_y - \kappa [\partial_y \hat v]_y - \frac{1}{\textrm{Re}} i \xi [\hat w]_y & = & 0        & \quad \lambda \in \Sigma_{\pi - \theta}, & \xi \in \bR^{n - 1},          \\[0.5em]
		                                                                                                [\hat w]_y & = & \hat h_w & \quad \lambda \in \Sigma_{\pi - \theta}, & \xi \in \bR^{n - 1},
	\end{array}
\end{equation*}
where $\hat v$, $\hat w$, $\hat p$ and $\hat h_w$ denote the transformed quantities, $\lambda \in \Sigma_{\pi - \theta}$ with $0 < \theta < \frac{\pi}{2}$
denotes the Laplace co-variable of $t$, where $\Sigma_\psi := \{\,z \in \bC \setminus \{\,0\,\}\,:\,|\mbox{arg}\,z| < \psi\,\}$ for $0 < \psi < \pi$,
and $\xi \in \bR^{n - 1}$ denotes the Fourier co-variable of $x$.
Moreover, we use the abbreviations
\begin{equation*}
	\lambda_\epsilon := \epsilon + \lambda, \qquad \omega := \sqrt{\lambda_\epsilon + |\zeta|^2}, \qquad \zeta := {\textstyle \frac{1}{\sqrt{\textrm{Re}}}} \xi
\end{equation*}
The first three equations above are ordinary differential equations for $y>0$, whose solutions admit a representation by linear combinations of fundamental solutions as
\begin{subequations}
\begin{equation}
	\eqnlabel{Halfspace-Generic-Solution}
	\left[ \begin{array}{c} \hat v(\lambda,\,\xi,\,y) \\[0.5em] \hat w(\lambda,\,\xi,\,y) \\[0.5em] \hat p(\lambda,\,\xi,\,y) \end{array} \right]
		= \left[ \begin{array}{rr} \omega & - i \zeta \\[0.5em] i \zeta^\sfT & |\zeta| \\[0.5em] 0 & \quad \frac{1}{\sqrt{\textrm{Re}}} \lambda_\epsilon \end{array} \right]
		  \left[ \begin{array}{c} \hat{\tau}_v(\lambda,\,\xi) e^{- \sqrt{\textrm{Re}}\,\omega y} \\[0.5em] \hat{\tau}_w(\lambda,\,\xi) e^{- \sqrt{\textrm{Re}}\,|\zeta| y} \end{array} \right]
\end{equation}
for a function $\tau = (\tau_v,\,\tau_w): \bR_+ \times \bR^{n - 1} \longrightarrow \bR^n$, which has to be determined based on the boundary conditions.
Due to \eqnref{Halfspace-Generic-Solution} we have
\begin{equation}
	\eqnlabel{Halfspace-Generic-Boundary}
	\begin{array}{rclcrcl}
		           [\hat{v}]_y & = & \omega \hat{\tau}_v - i \zeta \hat{\tau}_w,                                                     & \qquad &              [\hat{w}]_y & = & i \zeta^\sfT \hat{\tau}_v + |\zeta| \hat{\tau}_w,                                                             \\[0.5em]
		[\partial_y \hat{v}]_y & = & - \sqrt{\textrm{Re}}\,\omega^2 \hat{\tau}_v + \sqrt{\textrm{Re}}\,|\zeta| i \zeta \hat{\tau}_w, & \qquad &   [\partial_y \hat{w}]_y & = & - \sqrt{\textrm{Re}}\,\omega i \zeta^\sfT \hat{\tau}_v - \sqrt{\textrm{Re}}\,|\zeta|^2 \hat{\tau}_w,          \\[0.5em]
		           [\hat{p}]_y & = & \frac{1}{\sqrt{\textrm{Re}}} \lambda_\epsilon \hat{\tau}_w,                                     & \qquad & \widehat{\nabla_x [w]_y} & = & \sqrt{\textrm{Re}}\,(i \zeta \otimes i \zeta) \hat{\tau}_v + \sqrt{\textrm{Re}}\,|\zeta| i \zeta \hat{\tau}_w
	\end{array}
\end{equation}
and, thus, the boundary conditions read
\begin{equation*}
	\left[ \begin{array}{rr} \underbrace{\alpha \sqrt{\textrm{Re}}\,\lambda_\epsilon \omega + \textrm{Re}\,\kappa \omega^2 - (i \zeta \otimes i \zeta)}_{=: B(\lambda, |\zeta|)} & \quad - \underbrace{(\alpha \sqrt{\textrm{Re}}\,\lambda_\epsilon + \textrm{Re}\,\kappa |\zeta| + |\zeta|)}_{=: \beta(\lambda, |\zeta|)} i \zeta \\[3.0em] i \zeta^\sfT & |\zeta| \end{array} \right]
		\left[ \begin{array}{c} \hat{\tau}_v \\[0.5em] \hat{\tau}_w \end{array} \right]
		= \left[ \begin{array}{c} 0 \\[0.5em] \hat{h}_w \end{array} \right]
\end{equation*}
and we obtain $\hat{\tau}_w = (|\zeta| + \beta i \zeta^\sfT B^{-1} i \zeta)^{-1} \hat{h}_w$.
Now,
\begin{equation}
	\eqnlabel{Halfspace-Generic-Inverse}
	B^{-1}(\lambda,\,|\zeta|) = \frac{1}{\phi(\lambda,\,|\zeta|)} \left( 1 + \frac{i \zeta \otimes i \zeta}{\phi(\lambda,\,|\zeta|) + |\zeta|^2} \right), \qquad
		\phi(\lambda,\,|\zeta|) = \alpha \sqrt{\textrm{Re}}\,\lambda_\epsilon \omega + \textrm{Re}\,\kappa \omega^2,
\end{equation}
\end{subequations}
which implies that
\begin{equation*}
	|\zeta| + \beta i \zeta^\sfT B^{-1} i \zeta
		= |\zeta| + \frac{\beta}{\phi} \left( \frac{|\zeta|^4}{\phi + |\zeta|^2} - |\zeta|^2 \right)
		= |\zeta| - \frac{\beta |\zeta|^2}{\phi + |\zeta|^2}
		= |\zeta| \frac{\phi + |\zeta|^2 - \beta |\zeta|}{\phi + |\zeta|^2}
\end{equation*}
and, hence,
\begin{equation*}
	\begin{array}{rcl}
		\widehat{\Pi h_w}
			& = & \lambda_\epsilon |\zeta| \hat{\tau}_w = \frac{\displaystyle \alpha \sqrt{\textrm{Re}}\,\lambda_\epsilon \omega + \textrm{Re}\,\kappa \omega^2 + |\zeta|^2}{\displaystyle \alpha \sqrt{\textrm{Re}}\,\lambda_\epsilon + \textrm{Re}\,\kappa(\omega + |\zeta|)} (\omega + |\zeta|) \hat{h}_w \\[1.5em]
			& = & \bigg[\,\underbrace{\frac{\displaystyle \alpha \sqrt{\textrm{Re}}\,\lambda_\epsilon}{\displaystyle \alpha \sqrt{\textrm{Re}}\,\lambda_\epsilon + \textrm{Re}\,\kappa(\omega + |\zeta|)}}_{=: m_1(\lambda,\,|\zeta|)} \ + \ \underbrace{\frac{\displaystyle \textrm{Re}\,\kappa \omega}{\displaystyle \alpha \sqrt{\textrm{Re}}\,\lambda_\epsilon + \textrm{Re}\,\kappa(\omega + |\zeta|)}}_{=: m_2(\lambda,\,|\zeta|)}\,\bigg] \omega (\omega + |\zeta|) \hat{h}_w \\[3.5em]
			&   & \qquad \qquad + \ \bigg[\,\underbrace{\frac{\displaystyle \textrm{Re}\,\kappa |\zeta|}{\displaystyle \alpha \sqrt{\textrm{Re}}\,\lambda_\epsilon + \textrm{Re}\,\kappa(\omega + |\zeta|)}}_{=: m_3(\lambda,\,|\zeta|)}\,\bigg] \bigg[ \underbrace{\displaystyle \frac{|\zeta|}{\phantom{\sqrt{\textrm{Re}} (\lambda_\epsilon + |\zeta|)} \hspace*{-6em} \textrm{Re}\,\kappa \omega}}_{=: \mu(\lambda,\,|\zeta|)} \bigg] \omega (\omega + |\zeta|) \hat{h}_w.
	\end{array}
\end{equation*}
Therefore, on a symbolic level we have
\begin{equation*}
	\Pi \sim \big( m_1(\lambda,\,|\zeta|) + m_2(\lambda,\,|\zeta|) + m_3(\lambda,\,|\zeta|) \mu(\lambda,\,|\zeta|) \big) \omega (\omega + |\zeta|) =: M(\lambda,\,|\zeta|),
\end{equation*}
which is the desired representation of $\Pi$.

\subsubsection*{Step 2.2}
Based on the above considerations we have $\Pi = \mathrm{Op}(M)$ and, thus,
the mapping properties of $\Pi$ may be derived by studying its Fourier-Laplace symbol $M$.
First note that
\begin{equation*}
	G := \mathrm{Op}(\lambda) = \partial_t \qquad \textrm{and} \qquad D := \mathrm{Op}(|\zeta|) = \sqrt{- {\textstyle \frac{1}{\textrm{Re}}} \Delta_{\bR^{n - 1}}}
\end{equation*}
each admit an $\cR$-bounded $\cH^\infty$-calculus with $\cR\cH^\infty$-angles $\phi^\infty_G = \frac\pi2$ and $\phi^\infty_D = 0$, respectively,
within the scales ${}_0 \cJ^r_p(\bR_+,\,\cK^s_p(\bR^{n - 1}))$ and
${}_0 \cJ^r_p(\bR_+,\,\dot{\cK}^s_p(\bR^{n - 1}))$ for $\cJ,\,\cK \in \{\,H,\,W\,\}$, $r \geq 0$, and $s \in \bR$
see e.\,g.\ \cite[Corollary~2.10]{Denk-Saal-Seiler:Newton-Polygon}.
This combined with \cite[Theorem~6.1]{Kalton-Weis:Operator-Sums} implies that the pair $(G,\,D)$
admits a joint $\cH^\infty(\Sigma_{\pi - \theta} \times \Sigma_{\theta / 2})$-calculus for every $0 < \theta < \frac{\pi}{2}$.
Now, it has been proved as part of \cite[Theorem~2.3]{Bothe-Koehne-Pruess:Energy-Preserving-Boundary-Conditions}
that the operator $\mathrm{Op}(\omega(\omega + |\zeta|))$ has the mapping properties \eqnref{Halfspace-Mappings-Tangential},
\mbox{cf.\ \cite[Section~4, The Case $\alpha = $ and $\beta = 0$]{Bothe-Koehne-Pruess:Energy-Preserving-Boundary-Conditions}}.
Thus, it remains to prove that the functions
\begin{equation*}
	(\lambda,\,z) \mapsto m_j(\lambda,\,z),\ \mu(\lambda,\,z): \Sigma_{\pi - \theta} \times \Sigma_{\theta / 2} \longrightarrow \bC, \qquad j = 1,\,2,\,3
\end{equation*}
are bounded and holomorphic for some $0 < \theta < \frac{\pi}{2}$;
this implies the operators $\mathrm{Op}(m_j) = m_j(G,\,D)$ for $j = 1,\,2,\,3$ and $\mathrm{Op}(\mu) = \mu(G,\,D)$ to be bounded
within the above scales of function spaces.

Since $\mu$ is clearly bounded and holomorphic in $\Sigma_{\pi - \theta} \times \Sigma_{\theta / 2}$,
we restrict our considerations to the $m_j$ for $j = 1,\,2,\,3$.
It sufficies to prove that the reciprocals
\begin{equation*}
	m^{-1}_1 = 1 + {\textstyle \frac{\textrm{Re}\,\kappa(\omega(z) + z)}{\alpha \sqrt{\textrm{Re}}\,\lambda_\epsilon}}, \qquad
	m^{-1}_2 = 1 + {\textstyle \frac{\alpha \sqrt{\textrm{Re}}\,\lambda_\epsilon + \textrm{Re}\,\kappa z}{\textrm{Re}\,\kappa \omega(z)}}, \qquad
	m^{-1}_3 = 1 + {\textstyle \frac{\alpha \sqrt{\textrm{Re}}\,\lambda_\epsilon + \textrm{Re}\,\kappa \omega(z)}{\textrm{Re}\,\kappa z}},
\end{equation*}
with $\omega(z) = \sqrt{\lambda_\epsilon + z^2}$ are uniformly away from the origin for $(\lambda,\,z) \in \Sigma_{\pi - \theta} \times \Sigma_{\theta / 2}$.
Here we exploit the following elementary but useful fact:
if $|\arg z_1|,\,|\arg z_2|,\,|\arg z_1 - \arg z_2| < \pi$ then
\begin{equation*}
	\min\{\,\arg z_1,\,\arg z_2\,\} \leq \arg(z_1 + z_2) \leq \max\{\,\arg z_1,\,\arg z_2\,\}.
\end{equation*}
Thus, for $(\lambda,\,z) \in \Sigma_{\pi - \theta} \times \Sigma_{\theta / 2}$ with $\arg \lambda \geq 0$ we have
\begin{equation*}
	- \theta < \arg(\lambda_\epsilon + z),\,\arg(\lambda_\epsilon + z^2) < \pi - \theta, \qquad
	- {\textstyle \frac{\theta}{2}} < \arg \omega(z),\,\arg(\omega(z) + z) < {\textstyle \frac{\pi}{2}} - {\textstyle \frac{\theta}{2}},
\end{equation*}
which implies
\begin{equation*}
	\begin{array}{c}
		-\pi + {\textstyle \frac{\theta}{2}} < \arg {\textstyle \frac{\textrm{Re}\,\kappa(\omega(z) + z)}{\alpha \sqrt{\textrm{Re}}\,\lambda_\epsilon}} < {\textstyle \frac{\pi}{2}} - {\textstyle \frac{\theta}{2}}, \qquad
		- {\textstyle \frac{\pi}{2}} - {\textstyle \frac{\theta}{2}} < \arg {\textstyle \frac{\alpha \sqrt{\textrm{Re}}\,\lambda_\epsilon + \textrm{Re}\,\kappa z}{\textrm{Re}\,\kappa \omega(z)}} < \pi - {\textstyle \frac{\theta}{2}}, \\[1.5em]
		- {\textstyle \frac{3 \theta}{2}} < \arg {\textstyle \frac{\alpha \sqrt{\textrm{Re}}\,\lambda_\epsilon + \textrm{Re}\,\kappa \omega(z)}{\textrm{Re}\,\kappa z}} < \pi - {\textstyle \frac{\theta}{2}},
	\end{array}
\end{equation*}
i.\,e.\ $m^{-1}_j$ is indeed uniformly away from the origin for $j = 1,\,2,\,3$.
By symmetry, we obtain the same behavior for $\arg \lambda \leq 0$, which shows that $\Pi$ has the mapping properties \eqnref{Halfspace-Mappings-Tangential}.
This finishes the proof of \Thmref{Linear-Halfspace} for the boundary condition \eqnref*{Tangential-NRBC}.

\subsection{The Condition NDO}
\subseclabel{Linear-Halfspace-Proof-Normal}
Now we consider the Stokes equations subject to a dynamic outflow boundary condition in normal directions,
i.\,e.\ the system \eqnref*{Interior-Linear}, \eqnref*{Normal-NRBC} which reads
\begin{equation}
	\eqnlabel{Linear-Halfspace-Normal}
	\begin{array}{rclll}
		                             \partial_t u - \frac{1}{\textrm{Re}} \Delta u + \nabla p & = & f   & \quad \mbox{in} & J \times \bR^n_+,     \\[0.5em]
		                                                                        \mbox{div}\,u & = & g   & \quad \mbox{in} & J \times \bR^n_+,     \\[0.5em]
		                                                                                [v]_y & = & h_v & \quad \mbox{on} & J \times \bR^{n - 1}, \\[0.5em]
		\alpha \partial_t [w]_y - (\alpha V + \frac{2}{\textrm{Re}}) [\partial_y w]_y + [p]_y & = & h_w & \quad \mbox{on} & J \times \bR^{n - 1}, \\[0.5em]
		                                                                                 u(0) & = & u_0 & \quad \mbox{in} & \bR^n_+.
	\end{array}
\end{equation}
We again require the regularity and compatibility conditions as stated in \Thmref{Linear-Halfspace}, i.\,e.\ we have
\begin{itemize}
	\item $f \in L_p(J \times \bR^n_+)^n$,
	\item $g \in H^{1/2}_p(J,\,L_p(\bR^n_+)) \cap L_p(J,\,H^1_p(\bR^n_+))$,
	\item $h_v \in W^{1 - 1/2p}_p(J,\,L_p(\bR^{n - 1}))^{n - 1} \cap L_p(J,\,W^{2 - 1/p}_p(\bR^{n - 1}))^{n - 1}$,
	\item $h_w \in L_p(J,\,W^{1 - 1/p}_p(\bR^{n - 1}))$,
	\item $u_0 \in W^{2 - 2/p}_p(\bR^n_+)^n$ with $\textrm{\upshape div}\,u_0 = g(0)$ in $\bR^n_+$ for $p \geq 2$,
		\item $F(g,\,\eta) \in H^1_p(J,\,\hat{H}^{-1}_p(\Omega))$ for some \\
			$\eta \in H^1_p(J,\,W^{1 - 1/p}_p(\Gamma)) \cap L_p(J,\,W^{2 - 1/p}_p(\Gamma))$
			with $[w_0]_y = \eta(0)$ for $p > \frac{3}{2}$,
	\item $[v_0]_y = h_v(0)$ for $p > \frac{3}{2}$.
\end{itemize}
The construction of a solution to \eqnref{Linear-Halfspace-Normal} requires several Steps.

\subsection*{Step 1}
As a first step we again show that we may w.\,l.\,o.\,g.\ assume $f = 0$, $g = 0$, $h_v = 0$ and $u_0 = 0$ in the following.
Indeed, we may solve the Stokes equations with Dirichlet boundary conditions
\begin{equation*}
	\begin{array}{rclll}
		\partial_t u - \frac{1}{\textrm{Re}} \Delta u + \nabla p & = & f    & \quad \mbox{in} & J \times \bR^n_+,     \\[0.5em]
		                                           \mbox{div}\,u & = & g    & \quad \mbox{in} & J \times \bR^n_+,     \\[0.5em]
		                                                   [v]_y & = & h_v  & \quad \mbox{on} & J \times \bR^{n - 1}, \\[0.5em]
		                                                   [w]_y & = & \eta & \quad \mbox{on} & J \times \bR^{n - 1}, \\[0.5em]
		                                                    u(0) & = & u_0  & \quad \mbox{in} & \bR^n_+.
	\end{array}
\end{equation*}
to obtain a unique solution in the desired regularity class using well-known results on maximal regularity of the Stokes equations,
see e.\,g.\ \cite[Theorem~2.3]{Bothe-Koehne-Pruess:Energy-Preserving-Boundary-Conditions}.
This immediately leads to the desired reduction.
Note, however, that we now have to assume $h_w \in L_p(J,\,\dot{W}^{1 - 1/p}_p(\bR^{n - 1}))$
to obtain a pressure $p \in L_p(J,\,\dot{H}^1_p(\bR^n_+))$ without additional regularity for $[p]_y$.

\subsection*{Step 2}
In order to solve the remaining problem, we will again employ a Laplace transformation in time and a Fourier transformation in the tangential part of the spatial variables.
This is again only possible for an unbounded time interval, i.\,e.\ we will from now on consider the shifted problem
\begin{equation}
	\eqnlabel{Linear-Halfspace-Normal-Reduced}
	\begin{array}{rclll}
		                                        \epsilon u + \partial_t u - \frac{1}{\textrm{Re}} \Delta u + \nabla p & = & 0   & \quad \mbox{in} & \bR_+ \times \bR^n_+,     \\[0.5em]
		                                                                                                \mbox{div}\,u & = & 0   & \quad \mbox{in} & \bR_+ \times \bR^n_+,     \\[0.5em]
		                                                                                                        [v]_y & = & 0   & \quad \mbox{on} & \bR_+ \times \bR^{n - 1}, \\[0.5em]
		\alpha \epsilon [w]_y + \alpha \partial_t [w]_y - (\alpha V + \frac{2}{\textrm{Re}}) [\partial_y w]_y + [p]_y & = & h_w & \quad \mbox{on} & \bR_+ \times \bR^{n - 1}, \\[0.5em]
		                                                                                                         u(0) & = & 0   & \quad \mbox{in} & \bR^n_+
	\end{array}
\end{equation}
for an arbitrary $\epsilon > 0$.
Note that maximal regularity for this problem is again equivalent to maximal regularity of the original problem (i.\,e.\ for $\epsilon = 0$) on finite time intervals $J = (0,\,a)$.
The strategy to construct a solution to \eqnref{Linear-Halfspace-Normal-Reduced} is as follows:
We compute the pressure trace $[p]_y = \Pi h_w$ as well as $h_w - [p]_y = \Sigma h_w$ and show that these are given based on bounded linear operators
\begin{equation}
	\eqnlabel{Halfspace-Mappings-Normal}
	\begin{array}{rl}
		\Sigma: & L_p(\bR_+,\,\dot{W}^{1 - 1/p}_p(\bR^{n - 1})) \longrightarrow L_p(\bR_+,\,W^{1 - 1/p}_p(\bR^{n - 1})),      \\[0.5em]
		   \Pi: & L_p(\bR_+,\,\dot{W}^{1 - 1/p}_p(\bR^{n - 1})) \longrightarrow L_p(\bR_+,\,\dot{W}^{1 - 1/p}_p(\bR^{n - 1})).
	\end{array}
\end{equation}
Then we obtain the pressure $p \in L_p(\bR_+,\,\dot{H}^1_p(\bR^n_+))$ as a solution to the (weak) elliptic problem
\begin{equation*}
	\begin{array}{rclll}
		- \Delta p & = & 0       & \quad \mbox{in} & \bR_+ \times \bR^n_+,     \\[0.5em]
		     [p]_y & = & \Pi h_w & \quad \mbox{on} & \bR_+ \times \bR^{n - 1},
	\end{array}
\end{equation*}
cf.~\cite[Proposition~3.1]{Bothe-Koehne-Pruess:Energy-Preserving-Boundary-Conditions}.
Finally, we obtain $u$ as a maximal regular solution to the parabolic problem
\begin{equation*}
	\begin{array}{rclll}
		                                           \epsilon u + \partial_t u - \frac{1}{\textrm{Re}} \Delta u & = & - \nabla p & \quad \mbox{in} & \bR_+ \times \bR^n_+,     \\[0.5em]
		                                                                                                [v]_y & = & 0          & \quad \mbox{on} & \bR_+ \times \bR^{n - 1}, \\[0.5em]
		\alpha \epsilon [w]_y + \alpha \partial_t [w]_y - (\alpha V + \frac{2}{\textrm{Re}}) [\partial_y w]_y & = & \Sigma h_w & \quad \mbox{on} & \bR_+ \times \bR^{n - 1}, \\[0.5em]
		                                                                                                 u(0) & = & 0          & \quad \mbox{in} & \bR^n_+
	\end{array}
\end{equation*}
via \Propref{Parabolic-DBC-Low-Regularity}.

\subsubsection*{Step 2.1}
We compute the symbols of $\Sigma$ and $\Pi$.
The transformed equations \eqnref{Linear-Halfspace-Normal-Reduced} read:
\begin{equation*}
	\begin{array}{rcllll}
		      \omega^2 \hat v - \frac{1}{\textrm{Re}} \partial_y^2 \hat v + i\xi \hat p & = & 0         & \quad \lambda \in \Sigma_{\pi - \theta}, & \xi \in \bR^{n - 1}, & y > 0, \\[0.5em]
		\omega^2 \hat w - \frac{1}{\textrm{Re}} \partial_y^2 \hat w + \partial_y \hat p & = & 0         & \quad \lambda \in \Sigma_{\pi - \theta}, & \xi \in \bR^{n - 1}, & y > 0, \\[0.5em]
		                                         i \xi \cdot \hat v + \partial_y \hat w & = & 0         & \quad \lambda \in \Sigma_{\pi - \theta}, & \xi \in \bR^{n - 1}, & y > 0, \\[0.5em]
		                                                                     [\hat v]_y & = & 0         & \quad \lambda \in \Sigma_{\pi - \theta}, & \xi \in \bR^{n - 1},          \\[0.5em]
		 \alpha \lambda_\epsilon [\hat w]_y - \sigma [\partial_y \hat w]_y + [\hat p]_y & = & \hat{h}_w & \quad \lambda \in \Sigma_{\pi - \theta}, & \xi \in \bR^{n - 1},          \\[0.5em]
	\end{array}
\end{equation*}
where we used the same notations as in the \Subsecref{Linear-Halfspace-Proof-Tangential}.
We again employ the ansatz \eqnref{Halfspace-Generic-Solution}
and due to \eqnref{Halfspace-Generic-Boundary} and the divergence equation the boundary conditions read
\begin{equation*}
	\left[ \begin{array}{rr} \omega & - i \zeta \\[0.5em] \alpha \lambda_\epsilon i \zeta^\sfT & \quad \alpha \lambda_\epsilon |\zeta| + \frac{1}{\sqrt{\textrm{Re}}} \lambda_\epsilon \end{array} \right]
		\left[ \begin{array}{c} \hat{\tau}_v \\[0.5em] \hat{\tau}_w \end{array} \right]
		= \left[ \begin{array}{c} 0 \\[0.5em] \hat{h}_w \end{array} \right]
\end{equation*}
and we obtain
\begin{equation*}
	{\textstyle \frac{1}{\sqrt{\textrm{Re}}}} \lambda_\epsilon \hat{\tau}_w = \left( 1 + \sqrt{\textrm{Re}}\,\alpha |\zeta| \left( 1 - \frac{|\zeta|}{\omega} \right) \right)^{-1} \hat{h}_w.
\end{equation*}
This implies
\begin{equation*}
	\widehat{\Sigma h_w} = \frac{\sqrt{\textrm{Re}}\,\alpha |\zeta| \left( 1 - \frac{|\zeta|}{\omega} \right)}{1 + \sqrt{\textrm{Re}}\,\alpha |\zeta| \left( 1 - \frac{|\zeta|}{\omega} \right)} \hat{h}_w, \qquad \qquad
	   \widehat{\Pi h_w} = \frac{1}{1 + \sqrt{\textrm{Re}}\,\alpha |\zeta| \left( 1 - \frac{|\zeta|}{\omega} \right)} \hat{h}_w,
\end{equation*}
which are the desired representations of $\Sigma$ and $\Pi$.

\subsubsection*{Step 2.2}
In order to derive the mapping properties \eqnref{Halfspace-Mappings-Normal} based on the representations obtained above
we employ the same techniques as in Step 2.2 of \Subsecref{Linear-Halfspace-Proof-Tangential}.
By the very same arguments as used there we obtain that the symbol of $\Pi$ is bounded and holomorphic in $\Sigma_{\pi - \theta} \times \Sigma_{\theta / 2}$
for some $0 < \theta < \frac{\pi}{2}$.
This yields the desired mapping properties of $\Pi$.
Moreover, based on its symbol, $\Sigma$ has the same mapping properties as
\begin{equation*}
	\mathrm{Op}({\textstyle \frac{|\zeta|}{1 + |\zeta|}}): L_p(\bR_+,\,\dot{W}^{1 - 1/p}_p(\bR^{n - 1})) \longrightarrow L_p(\bR_+,\,W^{1 - 1/p}_p(\bR^{n - 1})),
\end{equation*}
which yields the mapping properties \eqnref{Halfspace-Mappings-Normal}.
This finishes the proof of \Thmref{Linear-Halfspace} for the boundary condition \eqnref*{Normal-NRBC}.
\begin{remark}
	\remlabel{Normal-Pressure-Effect}
	Note that we have
	\begin{equation*}
		\mathrm{Sym}(\Sigma) \rightarrow 0, \qquad \qquad \mathrm{Sym}(\Pi) \rightarrow 1 \qquad \qquad \mbox{as} \ \alpha \rightarrow 0,
	\end{equation*}
	which are the symbols of the corresponding operators for the boundary condition
	\begin{equation}
		\eqnlabel{Normal-Pressure-Effect}
		\begin{array}{rclll}
		                                             [v]_y & = & 0   & \quad \mbox{on} & J \times \bR^{n - 1}, \\[0.5em]
			- \frac{2}{\textrm{Re}} [\partial_y w]_y + [p]_y & = & h_w & \quad \mbox{on} & J \times \bR^{n - 1},
	\end{array}
\end{equation}
which is one of the energy preserving boundary conditions considered in \cite{Bothe-Koehne-Pruess:Energy-Preserving-Boundary-Conditions}.
Now, one can either employ $[\partial_y w]_y = -\nabla_x \cdot [v]_y = 0$ or the fact that $\Sigma = 0$ and $\Pi = 1$ for this limit case
to recognize that the absence of the pressure trace in \eqnref{Normal-Pressure-Effect} leads to an ill-posed problem,
since one boundary condition would be missing then.
A similar defect applies to the dynamic outflow condition \eqnref*{Normal-NRBC} without the pressure trace,
which would lead to an ill-posed problem.
\end{remark}

\subsection{The Condition FDO}
\subseclabel{Linear-Halfspace-Proof-Full}
Finally, we consider the Stokes equations subject to a fully dynamic outflow boundary condition,
i.\,e. the system \eqnref*{Interior-Linear}, \eqnref*{Fully-NRBC} which reads
\begin{equation}
	\eqnlabel{Linear-Halfspace-Full}
	\begin{array}{rclll}
		                                                            \partial_t u - \frac{1}{\textrm{Re}} \Delta u + \nabla p & = & f   & \quad \mbox{in} & J \times \bR^n_+,     \\[0.5em]
		                                                                                                       \mbox{div}\,u & = & g   & \quad \mbox{in} & J \times \bR^n_+,     \\[0.5em]
		\alpha \partial_t [v]_y - (\alpha V + \frac{1}{\textrm{Re}}) [\partial_y v]_y - \frac{1}{\textrm{Re}} \nabla_x [w]_y & = & h_v & \quad \mbox{on} & J \times \bR^{n - 1}, \\[0.5em]
		                               \alpha \partial_t [w]_y - (\alpha V + \frac{2}{\textrm{Re}}) [\partial_y w]_y + [p]_y & = & h_w & \quad \mbox{on} & J \times \bR^{n - 1}, \\[0.5em]
		                                                                                                                u(0) & = & u_0 & \quad \mbox{in} & \bR^n_+.
	\end{array}
\end{equation}
As in the previous steps we require the regularity and compatibility conditions as stated in \Thmref{Linear-Halfspace}, i.\,e.\ we assume that
\begin{itemize}
	\item $f \in L_p(J \times \bR^n_+)^n$,
	\item $g \in H^{1/2}_p(J,\,L_p(\bR^n_+)) \cap L_p(J,\,H^1_p(\bR^n_+))$,
	\item $h_v \in W^{1/2 - 1/2p}_p(J,\,L_p(\bR^{n - 1}))^{n - 1} \cap L_p(J,\,W^{1 - 1/p}_p(\bR^{n - 1}))^{n - 1}$,
	\item $h_w \in L_p(J,\,W^{1 - 1/p}_p(\bR^{n - 1}))$,
	\item $u_0 \in W^{2 - 2/p}_p(\bR^n_+)^n$ with $\textrm{\upshape div}\,u_0 = g(0)$ in $\bR^n_+$ for $p \geq 2$,
		\item $F(g,\,\eta) \in H^1_p(J,\,\hat{H}^{-1}_p(\Omega))$ for some \\
			$\eta \in H^1_p(J,\,W^{1 - 1/p}_p(\Gamma)) \cap L_p(J,\,W^{2 - 1/p}_p(\Gamma))$
			with $[w_0]_y = \eta(0)$ for $p > \frac{3}{2}$,
	\item $[v_0]_y \in W^{2 - 2/p}_p(\bR^{n - 1})^{n - 1}$.
\end{itemize}
The construction of a solution to \eqnref{Linear-Halfspace-Full} requires several Steps.

\subsection*{Step 1}
As a first step we again show that we may w.\,l.\,o.\,g.\ assume $f = 0$, $g = 0$, $h_v = 0$ and $u_0 = 0$ in the following.
Indeed, since the proof \Thmref{Linear-Halfspace} concerning the boundary condition \eqnref*{Tangential-NRBC} has already been given,
we may now solve the Stokes equations subject to a tangential dynamic outflow boundary condition
\begin{equation*}
	\begin{array}{rclll}
		                                                            \partial_t u - \frac{1}{\textrm{Re}} \Delta u + \nabla p & = & f    & \quad \mbox{in} & J \times \bR^n_+,     \\[0.5em]
		                                                                                                       \mbox{div}\,u & = & g    & \quad \mbox{in} & J \times \bR^n_+,     \\[0.5em]
		\alpha \partial_t [v]_y - (\alpha V + \frac{1}{\textrm{Re}}) [\partial_y v]_y - \frac{1}{\textrm{Re}} \nabla_x [w]_y & = & h_v  & \quad \mbox{on} & J \times \bR^{n - 1}, \\[0.5em]
		                                                                                                               [w]_y & = & \eta & \quad \mbox{on} & J \times \bR^{n - 1}, \\[0.5em]
		                                                                                                                u(0) & = & u_0  & \quad \mbox{in} & \bR^n_+.
	\end{array}
\end{equation*}
to obtain a unique solution in the desired regularity class.
This immediately leads to the desired reduction.
Note, however, that we now have to assume $h_w \in L_p(J,\,\dot{W}^{1 - 1/p}_p(\bR^{n - 1}))$
to obtain a pressure $p \in L_p(J,\,\dot{H}^1_p(\bR^n_+))$ without additional regularity for $[p]_y$.

\subsection*{Step 2}
In order to solve the remaining problem, we will again employ a Laplace transformation in time and a Fourier transformation in the tangential part of the spatial variables.
This is again only possible for an unbounded time interval, i.\,e.\ we will from now on consider the shifted problem
\begin{equation}
	\eqnlabel{Linear-Halfspace-Full-Reduced}
	\begin{array}{rclll}
		                                                                       \epsilon u + \partial_t u - \frac{1}{\textrm{Re}} \Delta u + \nabla p & = & 0   & \quad \mbox{in} & \bR_+ \times \bR^n_+,     \\[0.5em]
		                                                                                                                               \mbox{div}\,u & = & 0   & \quad \mbox{in} & \bR_+ \times \bR^n_+,     \\[0.5em]
		\alpha \epsilon [v]_y + \alpha \partial_t [v]_y - (\alpha V + \frac{1}{\textrm{Re}}) [\partial_y v]_y - \frac{1}{\textrm{Re}} \nabla_x [w]_y & = & 0   & \quad \mbox{on} & \bR_+ \times \bR^{n - 1}, \\[0.5em]
		                               \alpha \epsilon [w]_y + \alpha \partial_t [w]_y - (\alpha V + \frac{2}{\textrm{Re}}) [\partial_y w]_y + [p]_y & = & h_w & \quad \mbox{on} & \bR_+ \times \bR^{n - 1}, \\[0.5em]
		                                                                                                                                        u(0) & = & 0   & \quad \mbox{in} & \bR^n_+
	\end{array}
\end{equation}
for an arbitrary $\epsilon > 0$.
Note that maximal regularity for this problem is again equivalent to maximal regularity of the original problem (i.\,e.\ for $\epsilon = 0$) on finite time intervals $J = (0,\,a)$.
The strategy to construct a solution to \eqnref{Linear-Halfspace-Full-Reduced} is the same as for the problem \eqnref{Linear-Halfspace-Normal-Reduced}:
We compute the pressure trace $[p]_y = \Pi h_w$ as well as $h_w - [p]_y = \Sigma h_w$ and show that these operators are bounded and linear in the setting \eqnref{Halfspace-Mappings-Normal}.
Then we obtain the pressure $p \in L_p(\bR_+,\,\dot{H}^1_p(\bR^n_+))$ as a solution to the (weak) elliptic problem
\begin{equation*}
	\begin{array}{rclll}
		- \Delta p & = & 0       & \quad \mbox{in} & \bR_+ \times \bR^n_+,     \\[0.5em]
		     [p]_y & = & \Pi h_w & \quad \mbox{on} & \bR_+ \times \bR^{n - 1},
	\end{array}
\end{equation*}
cf.~\cite[Proposition~3.1]{Bothe-Koehne-Pruess:Energy-Preserving-Boundary-Conditions}.
Finally, we obtain $u$ as a maximal regular solution to the parabolic problem
\begin{equation*}
	\begin{array}{rclll}
		                                           \epsilon u + \partial_t u - \frac{1}{\textrm{Re}} \Delta u & = & - \nabla p                           & \quad \mbox{in} & \bR_+ \times \bR^n_+,     \\[0.5em]
		\alpha \epsilon [v]_y + \alpha \partial_t [v]_y - (\alpha V + \frac{1}{\textrm{Re}}) [\partial_y v]_y & = & \frac{1}{\textrm{Re}} \nabla_x [w]_y & \quad \mbox{on} & \bR_+ \times \bR^{n - 1}, \\[0.5em]
		\alpha \epsilon [w]_y + \alpha \partial_t [w]_y - (\alpha V + \frac{2}{\textrm{Re}}) [\partial_y w]_y & = & \Sigma h_w                           & \quad \mbox{on} & \bR_+ \times \bR^{n - 1}, \\[0.5em]
		                                                                                                 u(0) & = & 0                                    & \quad \mbox{in} & \bR^n_+
	\end{array}
\end{equation*}
via Propositions~\ref{prop:Parabolic-DBC-Standard}~and~\ref{prop:Parabolic-DBC-Low-Regularity}.

\subsubsection*{Step 2.1}
We compute the symbols of $\Sigma$ and $\Pi$.
The transformed equations \eqnref{Linear-Halfspace-Full-Reduced} read:
\begin{equation*}
	\begin{array}{rcllll}
		                                 \omega^2 \hat v - \frac{1}{\textrm{Re}} \partial_y^2 \hat v + i\xi \hat p & = & 0         & \quad \lambda \in \Sigma_{\pi - \theta}, & \xi \in \bR^{n - 1}, & y > 0, \\[0.5em]
		                           \omega^2 \hat w - \frac{1}{\textrm{Re}} \partial_y^2 \hat w + \partial_y \hat p & = & 0         & \quad \lambda \in \Sigma_{\pi - \theta}, & \xi \in \bR^{n - 1}, & y > 0, \\[0.5em]
		                                                                    i \xi \cdot \hat v + \partial_y \hat w & = & 0         & \quad \lambda \in \Sigma_{\pi - \theta}, & \xi \in \bR^{n - 1}, & y > 0, \\[0.5em]
		\alpha \lambda_\epsilon [\hat v]_y - \kappa [\partial_y \hat v]_y - \frac{1}{\textrm{Re}} i \xi [\hat w]_y & = & 0         & \quad \lambda \in \Sigma_{\pi - \theta}, & \xi \in \bR^{n - 1},          \\[0.5em]
		                            \alpha \lambda_\epsilon [\hat w]_y - \sigma [\partial_y \hat w]_y + [\hat p]_y & = & \hat{h}_w & \quad \lambda \in \Sigma_{\pi - \theta}, & \xi \in \bR^{n - 1},          \\[0.5em]

	\end{array}
\end{equation*}
where we used the same notations as in the previous subsection.
We again employ the ansatz \eqnref{Halfspace-Generic-Solution}
and due to \eqnref{Halfspace-Generic-Boundary} the boundary conditions read
\begin{equation*}
	\raisebox{1em}{\BigE[} \begin{array}{rr} \underbrace{\alpha \sqrt{\textrm{Re}}\,\lambda_\epsilon \omega + \textrm{Re}\,\kappa \omega^2 - (i \zeta \otimes i \zeta)}_{=: B(\lambda, |\zeta|)} & \quad - \underbrace{(\alpha \sqrt{\textrm{Re}}\,\lambda_\epsilon + \textrm{Re}\,\kappa |\zeta| + |\zeta|)}_{=: \beta_v(\lambda, |\zeta|)} i \zeta \\[3.0em] \underbrace{(\alpha \sqrt{\textrm{Re}}\,\lambda_\epsilon + \textrm{Re}\,\sigma \omega)}_{=: \beta_w(\lambda, |\zeta|)} i \zeta^\sfT & \quad \underbrace{\alpha \sqrt{\textrm{Re}}\,\lambda_\epsilon |\zeta| + \textrm{Re}\,\sigma |\zeta|^2 + \lambda_\epsilon}_{=: \beta(\lambda, |\zeta|)} \end{array} \raisebox{1em}{\BigE]}
		\raisebox{1em}{$\left[ \begin{array}{c} \hat{\tau}_v \\[0.5em] \hat{\tau}_w \end{array} \right] = \left[ \begin{array}{c} 0 \\[0.5em] \sqrt{\textrm{Re}}\,\hat{h}_w \end{array} \right]$}
\end{equation*}
and we obtain $[\hat{p}]_y = \frac{1}{\sqrt{\textrm{Re}}} \lambda_\epsilon \hat{\tau}_w = \lambda_\epsilon\,(\beta + \beta_v \beta_w i \zeta^\sfT B^{-1} i \zeta)^{-1} \hat{h}_w$.
Now, using \eqnref{Halfspace-Generic-Inverse} we have
\begin{equation*}
	\beta + \beta_v \beta_w i \zeta^\sfT B^{-1} i \zeta
		= \beta + \frac{\beta_v \beta_w}{\phi} \left( \frac{|\zeta|^4}{\phi + |\zeta|^2} - |\zeta|^2 \right)
		= \beta - \frac{\beta_v \beta_w |\zeta|^2}{\phi + |\zeta|^2}
		= \frac{\beta (\phi + |\zeta|^2) - \beta_v \beta_w |\zeta|^2}{\phi + |\zeta|^2}
\end{equation*}
and $1 + \textrm{Re}\,\kappa = \textrm{Re}\,\sigma$ together with
\begin{equation*}
	\beta_v(\lambda,\,|\zeta|) = \alpha \sqrt{\textrm{Re}}\,\lambda_\epsilon + \textrm{Re}\,\sigma |\zeta|, \qquad \qquad
	  \beta(\lambda,\,|\zeta|) = \lambda_\epsilon + \beta_v(\lambda,\,|\zeta|) |\zeta|
\end{equation*}
implies
\begin{equation*}
	(\beta + \beta_v \beta_w i \zeta^\sfT B^{-1} i \zeta)^{-1} = \frac{\phi + |\zeta|^2}{\lambda_\epsilon(\phi + |\zeta|^2) + \beta_v (\phi + |\zeta|^2) |\zeta| - \beta_v \beta_w |\zeta|^2}.
\end{equation*}
In order to obtain a suitable representation of the symbols of $\Sigma$ and $\Pi$ we first observe that
\begin{equation*}
	\phi + |\zeta|^2 = \alpha \sqrt{\textrm{Re}}\,\lambda_\epsilon \omega + \textrm{Re}\,\kappa \omega^2 + |\zeta|^2
		= \alpha \sqrt{\textrm{Re}}\,\lambda_\epsilon \omega + \textrm{Re}\,\kappa \lambda_\epsilon + \textrm{Re}\,\sigma |\zeta|^2
\end{equation*}
while
\begin{equation*}
	\begin{array}{l}
		\textrm{Re}\,\sigma \lambda_\epsilon |\zeta|^2 + \beta_v (\textrm{Re}\,\kappa \lambda_\epsilon |\zeta| + \textrm{Re}\,\sigma |\zeta|^3 - \beta_w |\zeta|^2) \\[0.5em]
			\qquad = \textrm{Re}\,\sigma \lambda_\epsilon |\zeta|^2 + \beta_v (\textrm{Re}\,\kappa \lambda_\epsilon |\zeta| + \textrm{Re}\,\sigma (|\zeta| - \omega) |\zeta|^2 - \alpha \sqrt{\textrm{Re}}\,\lambda_\epsilon |\zeta|^2) \\[0.5em]
			\qquad = \textrm{Re}\,\sigma \lambda_\epsilon |\zeta|^2 + \textrm{Re}\,\kappa\,\textrm{Re}\,\sigma \lambda_\epsilon |\zeta|^2 + \textrm{Re}\,\sigma (|\zeta| - \omega) \textrm{Re}\,\sigma |\zeta|^3 \\[0.5em]
			\qquad \qquad \qquad \qquad + \ \alpha \sqrt{\textrm{Re}}\,\lambda_\epsilon (\textrm{Re}\,\kappa \lambda_\epsilon + \textrm{Re}\,\sigma (|\zeta| - \omega) |\zeta| - \beta_v |\zeta|) |\zeta| \\[1.0em]
			\qquad = {\displaystyle \frac{\textrm{Re}\,\sigma \lambda_\epsilon \textrm{Re}\,\sigma (\omega + |\zeta|) |\zeta|^2 - \textrm{Re}\,\sigma \lambda_\epsilon \textrm{Re}\,\sigma |\zeta|^3}{\omega + |\zeta|}} \\[1.0em]
			\qquad \qquad \qquad \qquad + \ \alpha \sqrt{\textrm{Re}}\,\lambda_\epsilon (\textrm{Re}\,\kappa \lambda_\epsilon + \textrm{Re}\,\sigma (|\zeta| - \omega) |\zeta| - \beta_v |\zeta|) |\zeta| \\[1.0em]
			\qquad = \textrm{Re}\,\sigma \lambda_\epsilon {\displaystyle \frac{\omega}{\omega + |\zeta|}} \textrm{Re}\,\sigma |\zeta|^2 + \alpha \sqrt{\textrm{Re}}\,\lambda_\epsilon (\textrm{Re}\,\kappa \lambda_\epsilon + \textrm{Re}\,\sigma (|\zeta| - \omega) |\zeta| - \beta_v |\zeta|) |\zeta|
	\end{array}
\end{equation*}
and
\begin{equation*}
	\begin{array}{l}
		\beta_v \alpha \sqrt{\textrm{Re}}\,\lambda_\epsilon \omega |\zeta| + \alpha \sqrt{\textrm{Re}}\,\lambda_\epsilon (\textrm{Re}\,\kappa \lambda_\epsilon + \textrm{Re}\,\sigma (|\zeta| - \omega) |\zeta| - \beta_v |\zeta|) |\zeta| \\[0.5em]
			\qquad = \alpha \sqrt{\textrm{Re}}\,\lambda_\epsilon (\beta_v (\omega - |\zeta|) + \textrm{Re}\,\kappa \lambda_\epsilon + \textrm{Re}\,\sigma (|\zeta| - \omega) |\zeta|) |\zeta| \\[1.5em]
			\qquad = \alpha \sqrt{\textrm{Re}}\,\lambda_\epsilon \left[ {\displaystyle \frac{\beta_v \lambda_\epsilon - \textrm{Re}\,\sigma \lambda_\epsilon |\zeta|}{\omega + |\zeta|}} + \textrm{Re}\,\kappa \lambda_\epsilon \right] |\zeta| \\[1.5em]
			\qquad = \alpha \sqrt{\textrm{Re}}\,\lambda_\epsilon \left[ {\displaystyle \frac{\alpha \sqrt{\textrm{Re}}\,\lambda^2_\epsilon}{\omega + |\zeta|}} + \textrm{Re}\,\kappa \lambda_\epsilon \right] |\zeta|,
	\end{array}
\end{equation*}
which implies
\begin{equation*}
	\widehat{\Pi h_w}
		= \frac{\alpha \sqrt{\textrm{Re}}\,\lambda_\epsilon \omega + \textrm{Re}\,\kappa \lambda_\epsilon + \textrm{Re}\,\sigma |\zeta|^2}{\alpha \sqrt{\textrm{Re}}\,\lambda_\epsilon \omega + \textrm{Re}\,\kappa \lambda_\epsilon + \alpha \sqrt{\textrm{Re}}\,\lambda_\epsilon \!\! \left[ \alpha \sqrt{\textrm{Re}}\,\lambda_\epsilon + \textrm{Re}\,\kappa (\omega + |\zeta|) \right] \!\! \frac{|\zeta|}{\omega + |\zeta|} + \textrm{Re}\,\sigma {\frac{\omega}{\omega + |\zeta|}} \textrm{Re}\,\sigma |\zeta|^2} \hat{h}_w.
\end{equation*}
Based on this representation of $\Pi$ we also obtain
\begin{equation*}
	\widehat{\Sigma h_w}
		= \frac{\alpha \sqrt{\textrm{Re}}\,\lambda_\epsilon \!\! \left[ \alpha \sqrt{\textrm{Re}}\,\lambda_\epsilon + \textrm{Re}\,\kappa (\omega + |\zeta|) \right] \!\! \frac{|\zeta|}{\omega + |\zeta|} + (\textrm{Re}\,\kappa \omega - |\zeta|) \textrm{Re}\,\sigma |\zeta| \frac{|\zeta|}{\omega + |\zeta|}}{\alpha \sqrt{\textrm{Re}}\,\lambda_\epsilon \omega + \textrm{Re}\,\kappa \lambda_\epsilon + \alpha \sqrt{\textrm{Re}}\,\lambda_\epsilon \!\! \left[ \alpha \sqrt{\textrm{Re}}\,\lambda_\epsilon + \textrm{Re}\,\kappa (\omega + |\zeta|) \right] \!\! \frac{|\zeta|}{\omega + |\zeta|} + \textrm{Re}\,\sigma {\frac{\omega}{\omega + |\zeta|}} \textrm{Re}\,\sigma |\zeta|^2} \hat{h}_w.
\end{equation*}
These are the desired representations of $\Sigma$ and $\Pi$.

\subsubsection*{Step 2.2}
In order to derive the mapping properties \eqnref{Halfspace-Mappings-Normal} based on the representations obtained above
we employ the same techniques as in Step 2.2 of \Subsecref{Linear-Halfspace-Proof-Tangential}.
By the very same arguments as used there we obtain that the symbol of $\Pi$ is bounded and holomorphic in $\Sigma_{\pi - \theta} \times \Sigma_{\theta / 2}$
for some $0 < \theta < \frac{\pi}{2}$.
This yields the desired mapping properties of $\Pi$.
Moreover, based on its symbol, $\Sigma$ has the same mapping properties as
\begin{equation*}
	\mathrm{Op}({\textstyle \frac{|\zeta|}{\omega + |\zeta|}}): L_p(\bR_+,\,\dot{W}^{1 - 1/p}_p(\bR^{n - 1})) \longrightarrow L_p(\bR_+,\,W^{1 - 1/p}_p(\bR^{n - 1})),
\end{equation*}
which yields the mapping properties \eqnref{Halfspace-Mappings-Normal}.
This finishes the proof of \Thmref{Linear-Halfspace} for the boundary condition \eqnref*{Fully-NRBC}.

\section{The Wedge Case}
\seclabel{Wedge}
This section is devoted to the proofs of Theorems~\ref{thm:Linear-Wedge-Inflow}, and \ref{thm:Linear-Wedge-Non-Reflecting}.
Here, we first note that we can always assume $\sigma = 0$, since the corresponding term is of lower order and may be added using a standard perturbation argument.
Moreover, we assume $a > 0$, set $J := (0,\,a)$, and assume $1 < p < \infty$ with $p \neq \frac{3}{2},\,3$.
Furthermore, we assume $\alpha,\,\mbox{Re} > 0$.
Since we solve the Stokes equations in the wedge $\Omega := \bR^n_{+\!\!\!+}$,
it is convenient to denote the velocity field as $(u,\,v,\,w): J \times \Omega \longrightarrow \bR^n$,
i.\,e.\ we employ a decomposition into a {\itshape purely tangential part} $u: J \times \Omega \longrightarrow \bR^{n - 2}$,
and two {\itshape normal parts} $v,\,w: J \times \Omega \longrightarrow \bR$.
The spatial coordinates are denoted by $(x,\,y,\,z) \in \bR^n_{+\!\!\!+}$ with $x \in \bR^{n - 2}$, and $y,\,z > 0$.
Finally, $\cE = \overline{\partial_y \bR^n_{+\!\!\!+}} \cap \overline{\partial_z \bR^n_{+\!\!\!+}}$ and
$[\,\cdot\,]_y$, and $[\,\cdot\,]_z$ denote the trace of a quantity defined in $\bR^n_{+\!\!\!+}$
on the boundaries $\partial_y \bR^n_{+\!\!\!+}$ and $\partial_z \bR^n_{+\!\!\!+}$, respectively.

\subsection{Combination of Inflow/Navier Conditions}
In order to prove \Thmref{Linear-Wedge-Inflow} we have to study the model problem
\begin{equation}
	\eqnlabel{Wedge-Problem}
	\begin{array}{rclll}
		                     \partial_t u - \frac{1}{\textrm{Re}} \Delta u +   \nabla_x p & = & f_u                          & \quad \mbox{in} & J \times \bR^n_{+\!\!\!+},            \\[0.5em]
		                     \partial_t v - \frac{1}{\textrm{Re}} \Delta v + \partial_y p & = & f_v                          & \quad \mbox{in} & J \times \bR^n_{+\!\!\!+},            \\[0.5em]
		                     \partial_t w - \frac{1}{\textrm{Re}} \Delta w + \partial_z p & = & f_w                          & \quad \mbox{in} & J \times \bR^n_{+\!\!\!+},            \\[0.5em]
		                                    \mbox{div}_x\,u + \partial_y v + \partial_z w & = & g                            & \quad \mbox{in} & J \times \bR^n_{+\!\!\!+},            \\[0.5em]
		  - \frac{1}{\textrm{Re}} [\partial_y u]_y - \frac{1}{\textrm{Re}} \nabla_x [v]_y & = & h^{\textrm{\upshape wall}}_u & \quad \mbox{on} & J \times \partial_y \bR^n_{+\!\!\!+}, \\[0.5em]
		                                                                            [v]_y & = & h^{\textrm{\upshape wall}}_v & \quad \mbox{on} & J \times \partial_y \bR^n_{+\!\!\!+}, \\[0.5em]
		- \frac{1}{\textrm{Re}} [\partial_y w]_y - \frac{1}{\textrm{Re}} \partial_z [v]_y & = & h^{\textrm{\upshape wall}}_w & \quad \mbox{on} & J \times \partial_y \bR^n_{+\!\!\!+}, \\[0.5em]
		                                                                            [u]_z & = & u^{\textrm{\upshape in}}_u   & \quad \mbox{on} & J \times \partial_z \bR^n_{+\!\!\!+}, \\[0.5em]
		                                                                            [v]_z & = & u^{\textrm{\upshape in}}_v   & \quad \mbox{on} & J \times \partial_z \bR^n_{+\!\!\!+}, \\[0.5em]
		                                                                            [w]_z & = & u^{\textrm{\upshape in}}_w   & \quad \mbox{on} & J \times \partial_z \bR^n_{+\!\!\!+}, \\[0.5em]
		                                         u(0) = u_0, \quad v(0) = v_0, \quad w(0) & = & w_0                          & \quad \mbox{in} & \bR^n_{+\!\!\!+},
	\end{array}
\end{equation}
where the data $f = (f_u,\,f_v,\,f_w)$, $g$,
$u^{\textrm{\upshape in}} = (u^{\textrm{\upshape in}}_u,\,u^{\textrm{\upshape in}}_v,\,u^{\textrm{\upshape in}}_w)$,
$h^{\textrm{\upshape wall}} = (h^{\textrm{\upshape wall}}_u,\,h^{\textrm{\upshape wall}}_v,\,h^{\textrm{\upshape wall}}_w)$, and the initial data
$(u_0,\,v_0,\,w_0)$ are subject to the regularity/compatibility conditions stated in \Thmref{Linear-Wedge-Inflow}, i.\,e.
\begin{itemize}
	\item $f \in L_p(J \times \bR^n_{+\!\!\!+})^n$,
	\item $g \in H^{1/2}_p(J,\,L_p(\bR^n_{+\!\!\!+})) \cap L_p(J,\,H^1_p(\bR^n_{+\!\!\!+}))$,
	\item $u^{\textrm{\upshape in}} \in W^{1 - 1/2p}_p(J,\,L_p(\partial_z \bR^n_{+\!\!\!+}))^n \cap L_p(J,\,W^{2 - 1/p}_p(\partial_z \bR^n_{+\!\!\!+}))^n$,
	\item $(h^{\textrm{\upshape wall}}_u,\,h^{\textrm{\upshape wall}}_w) \in W^{1/2 - 1/2p}_p(J,\,L_p(\partial_y \bR^n_{+\!\!\!+}))^{n - 1} \cap L_p(J,\,W^{1 - 1/p}_p(\partial_y \bR^n_{+\!\!\!+}))^{n - 1}$,
	\item $h^{\textrm{\upshape wall}}_v \in W^{1 - 1/2p}_p(J,\,L_p(\partial_y \bR^n_{+\!\!\!+})) \cap L_p(J,\,W^{2 - 1/p}_p(\partial_y \bR^n_{+\!\!\!+}))$,
	\item $F(g,\,- h^{\textrm{\upshape wall}}_v,\,- u^{\textrm{\upshape in}}_w) \in H^1_p(J,\,\hat{H}^{-1}_p(\bR^n_{+\!\!\!+}))$,
	\item $(u_0,\,v_0,\,w_0) \in W^{2 - 2/p}_p(J \times \bR^n_{+\!\!\!+})^n$
\end{itemize}
with
\begin{subequations}
\begin{equation}
	\eqnlabel{Wedge-Compatibility-1}
	\mbox{div}_x\,u_0 + \partial_y v_0 + \partial_z w_0 = g(0) \quad \mbox{in} \ \bR^n_{+\!\!\!+}, \quad \mbox{if} \ p \geq 2
\end{equation}
as well as
\begin{equation}
	\eqnlabel{Wedge-Compatibility-2a}
	\begin{array}{rclllll}
		  - \frac{1}{\textrm{Re}} [\partial_y u_0]_y - \frac{1}{\textrm{Re}} \nabla_x [v_0]_y & = & h^{\textrm{\upshape wall}}_u(0) & \quad \mbox{on} & \partial_y \bR^n_{+\!\!\!+}, & \quad \mbox{if} & p > 3,           \\[0.5em]
		                                                                              [v_0]_y & = & h^{\textrm{\upshape wall}}_v(0) & \quad \mbox{on} & \partial_y \bR^n_{+\!\!\!+}, & \quad \mbox{if} & p > \frac{3}{2}, \\[0.5em]
		- \frac{1}{\textrm{Re}} [\partial_y w_0]_y - \frac{1}{\textrm{Re}} \partial_z [v_0]_y & = & h^{\textrm{\upshape wall}}_w(0) & \quad \mbox{on} & \partial_y \bR^n_{+\!\!\!+}, & \quad \mbox{if} & p > 3,           \\[0.5em]
	\end{array}
\end{equation}
together with
\begin{equation}
	\eqnlabel{Wedge-Compatibility-2b}
	\begin{array}{rclllll}
		[u_0]_z & = & u^{\textrm{\upshape in}}_u(0)   & \quad \mbox{on} & \partial_z \bR^n_{+\!\!\!+}, & \quad \mbox{if} & p > \frac{3}{2}, \\[0.5em]
		[v_0]_z & = & u^{\textrm{\upshape in}}_v(0)   & \quad \mbox{on} & \partial_z \bR^n_{+\!\!\!+}, & \quad \mbox{if} & p > \frac{3}{2}, \\[0.5em]
		[w_0]_z & = & u^{\textrm{\upshape in}}_w(0)   & \quad \mbox{on} & \partial_z \bR^n_{+\!\!\!+}, & \quad \mbox{if} & p > \frac{3}{2},
	\end{array}
\end{equation}
and, due to \eqnref*{Inflow-Wall-Compatibility}, with
\begin{equation}
	\eqnlabel{Wedge-Compatibility-3}
	\begin{array}{rclll}
		  - \frac{1}{\textrm{Re}} [\partial_y u^{\textrm{\upshape in}}_u]_y - \frac{1}{\textrm{Re}} \nabla_x [h^{\textrm{\upshape wall}}_v]_z & = & [h^{\textrm{\upshape wall}}_u]_z & \quad \mbox{on} & J \times \cE, \\[0.5em]
		                                                                                                       [u^{\textrm{\upshape in}}_v]_y & = & [h^{\textrm{\upshape wall}}_v]_z & \quad \mbox{on} & J \times \cE, \\[0.5em]
		- \frac{1}{\textrm{Re}} [\partial_y u^{\textrm{\upshape in}}_w]_y - \frac{1}{\textrm{Re}} [\partial_z h^{\textrm{\upshape wall}}_v]_z & = & [h^{\textrm{\upshape wall}}_w]_z & \quad \mbox{on} & J \times \cE
	\end{array}
\end{equation}
\end{subequations}
for $p \geq 2$.
The construction of a solution to \eqnref{Wedge-Problem} requires several steps.

\subsection*{Step 1}
We first show that we can w.\,l.\,o.\,g.\ assume $f = 0$, as well as $g(0) = 0$, if $p \geq 2$, and
$h^{\textrm{\upshape wall}}_v(0) = 0$, if $p > \frac{3}{2}$, and $h^{\textrm{\upshape wall}}_u(0) = h^{\textrm{\upshape wall}}_w(0) = 0$, if $p > 3$,
as well as $u^{\textrm{\upshape in}}(0) = 0$, if $p > \frac{3}{2}$,
together with $u_0 = v_0 = w_0 = 0$.
Indeed, we may choose
\begin{equation*}
	\hat{f} \in L_p(J \times \bR^n), \quad \hat{u}_0 \in W^{2 - 2/p}_p(\bR^n)^{n - 2}, \quad \hat{v}_0,\,\hat{w}_0 \in W^{2 - 2/p}_p(\bR^n)
\end{equation*}
as extentions of $f$, $u_0$, $v_0$, and $w_0$, respectively.
Note that such extensions may be constructed using a linear extension operator as provided e.\,g.\ by \cite[Theorem~4.32]{Adams-Fournier:Sobolev-Spaces}.
Then the problems
\begin{equation*}
	\begin{array}{rclll}
		             \partial_t \hat{u} - \frac{1}{\textrm{Re}} \Delta \hat{u} & = & \hat{f}_u & \quad \mbox{in} & J \times \bR^n, \\[0.5em]
		             \partial_t \hat{v} - \frac{1}{\textrm{Re}} \Delta \hat{v} & = & \hat{f}_v & \quad \mbox{in} & J \times \bR^n, \\[0.5em]
		             \partial_t \hat{w} - \frac{1}{\textrm{Re}} \Delta \hat{w} & = & \hat{f}_w & \quad \mbox{in} & J \times \bR^n, \\[0.5em]
		\hat{u}(0) = \hat{u}_0, \quad \hat{v}(0) = \hat{v}_0, \quad \hat{w}(0) & = & \hat{w}_0 & \quad \mbox{in} & \bR^n
	\end{array}
\end{equation*}
admit unique solutions
\begin{equation*}
	\begin{array}{rclcl}
		          \hat{u} & \in & H^1_p(J,\,\bR^n))^{n - 2} & \cap & L_p(J,\,H^2_p(\bR^n))^{n - 2}, \\[0.5em]
		\hat{v},\,\hat{w} & \in & H^1_p(J,\,\bR^n))         & \cap & L_p(J,\,H^2_p(\bR^n)).
	\end{array}
\end{equation*}
Now, if we define $u$, $v$, and $w$ to be the restrictions of $\hat{u}$, $\hat{v}$, and $\hat{w}$ to $\bR^n_{+\!\!\!+}$,
then $(u,\,v,\,w)$ together with $p := 0$ belong to the desired regularity class and solve
\begin{equation*}
	\begin{array}{rclll}
		\partial_t u - \frac{1}{\textrm{Re}} \Delta u +   \nabla_x p & = & f_u & \quad \mbox{in} & J \times \bR^n_{+\!\!\!+}, \\[0.5em]
		\partial_t v - \frac{1}{\textrm{Re}} \Delta v + \partial_y p & = & f_v & \quad \mbox{in} & J \times \bR^n_{+\!\!\!+}, \\[0.5em]
		\partial_t w - \frac{1}{\textrm{Re}} \Delta w + \partial_z p & = & f_w & \quad \mbox{in} & J \times \bR^n_{+\!\!\!+}, \\[0.5em]
		                    u(0) = u_0, \quad v(0) = v_0, \quad w(0) & = & w_0 & \quad \mbox{in} & \bR^n_{+\!\!\!+}.
	\end{array}
\end{equation*}
Hence, we may assume $f = u_0 = v_0 = w_0 = 0$ together with the assumptions on $g(0)$, $h^{\textrm{\upshape wall}}(0)$, and $u^{\textrm{\upshape in}}(0)$ stated above.
Note that this reduction of the problem does not affect the regularity and compatibility assumptions of \Thmref{Linear-Wedge-Inflow}, i.\,e.\ we may still assume
\begin{itemize}
	\item $g \in H^{1/2}_p(J,\,L_p(\bR^n_{+\!\!\!+})) \cap L_p(J,\,H^1_p(\bR^n_{+\!\!\!+}))$,
	\item $u^{\textrm{\upshape in}} \in W^{1 - 1/2p}_p(J,\,L_p(\partial_z \bR^n_{+\!\!\!+}))^n \cap L_p(J,\,W^{2 - 1/p}_p(\partial_z \bR^n_{+\!\!\!+}))^n$,
	\item $(h^{\textrm{\upshape wall}}_u,\,h^{\textrm{\upshape wall}}_w) \in W^{1/2 - 1/2p}_p(J,\,L_p(\partial_y \bR^n_{+\!\!\!+}))^{n - 1} \cap L_p(J,\,W^{1 - 1/p}_p(\partial_y \bR^n_{+\!\!\!+}))^{n - 1}$,
	\item $h^{\textrm{\upshape wall}}_v \in W^{1 - 1/2p}_p(J,\,L_p(\partial_y \bR^n_{+\!\!\!+})) \cap L_p(J,\,W^{2 - 1/p}_p(\partial_y \bR^n_{+\!\!\!+}))$,
	\item $F(g,\,- h^{\textrm{\upshape wall}}_v,\,- u^{\textrm{\upshape in}}_w) \in H^1_p(J,\,\hat{H}^{-1}_p(\bR^n_{+\!\!\!+}))$
\end{itemize}
as well as compatibility condition \eqnref{Wedge-Compatibility-3} in the remaining part of the proof.

\subsection*{Step 2}
We show that we may w.\,l.\,o.\,g.\ assume $g \in H^1_p(J,\,{}_0 \dot{H}^{-1}_p(\bR^n_{+\!\!\!+}))$,
where we employ the notation ${}_0 \dot{H}^{-1}_p(\bR^n_{+\!\!\!+}) := \dot{H}^1_{p^\prime}(\bR^n_{+\!\!\!+})^\prime$ for $\frac{1}{p} + \frac{1}{p^\prime} = 1$,
as well as $h^{\textrm{\upshape wall}}_u = h^{\textrm{\upshape wall}}_v = u^{\textrm{\upshape in}}_w = 0$.
Indeed, we may choose
\begin{equation*}
	\begin{array}{rclcl}
		\hat{h}^{\textrm{\upshape wall}}_u & \in &                   W^{1/2 - 1/2p}_p(J,\,L_p(\partial \bR^n_{[y > 0]}))^{n - 2} & \cap & L_p(J,\,W^{1 - 1/p}_p(\partial \bR^n_{[y > 0]}))^{n - 2}, \\[0.5em]
		\hat{h}^{\textrm{\upshape wall}}_v & \in & {}^{\phantom{/2}} W^{1 - 1/2p}_p(J,\,L_p(\partial \bR^n_{[y > 0]}))           & \cap & L_p(J,\,W^{2 - 1/p}_p(\partial \bR^n_{[y > 0]})),         \\[0.5em]
		  \hat{u}^{\textrm{\upshape in}}_w & \in & {}^{\phantom{/2}} W^{1 - 1/2p}_p(J,\,L_p(\partial \bR^n_{[z > 0]}))           & \cap & L_p(J,\,W^{2 - 1/p}_p(\partial \bR^n_{[z > 0]}))
	\end{array}
\end{equation*}
as extensions of $h^{\textrm{\upshape wall}}_u$, $h^{\textrm{\upshape wall}}_v$, and $u^{\textrm{\upshape in}}_w$, respectively,
where we denote by
\begin{equation*}
	\bR^n_{[y > 0]} := \Big\{\,(x,\,y,\,z) \in \bR^n\,:\,x \in \bR^{n - 2},\ y > 0,\ z \in \bR\,\Big\},
\end{equation*}
and $\bR^n_{[z > 0]}$, which is defined analogously, the two halfspaces,
whose intersection is given by $\bR^n_{+\!\!\!+}$.
Note that such extensions may be constructed using a linear extension operator as provided e.\,g.\ by \cite[Theorem~4.26]{Adams-Fournier:Sobolev-Spaces}.
Then the problems
\begin{equation*}
	\begin{array}{rclll}
		\partial_t \hat{u} - \frac{1}{\textrm{Re}} \Delta \hat{u} & = & 0                                                                                                      & \quad \mbox{in} & J \times \bR^n_{[y > 0]},          \\[0.5em]
		\partial_t \hat{v} - \frac{1}{\textrm{Re}} \Delta \hat{v} & = & 0                                                                                                      & \quad \mbox{in} & J \times \bR^n_{[y > 0]},          \\[0.5em]
		\partial_t \hat{w} - \frac{1}{\textrm{Re}} \Delta \hat{w} & = & 0                                                                                                      & \quad \mbox{in} & J \times \bR^n_{[z > 0]},          \\[0.5em]
		           - \frac{1}{\textrm{Re}} [\partial_y \hat{u}]_y & = & \hat{h}^{\textrm{\upshape wall}}_u + \frac{1}{\textrm{Re}} \nabla_x \hat{h}^{\textrm{\upshape wall}}_v & \quad \mbox{on} & J \times \partial \bR^n_{[y > 0]}, \\[0.5em]
		                                              [\hat{v}]_y & = & \hat{h}^{\textrm{\upshape wall}}_v                                                                     & \quad \mbox{on} & J \times \partial \bR^n_{[y > 0]}, \\[0.5em]
		                                              [\hat{w}]_z & = & \hat{u}^{\textrm{\upshape in}}_w                                                                       & \quad \mbox{on} & J \times \partial \bR^n_{[z > 0]}, \\[0.5em]
		                         \hat{u}(0) = 0, \quad \hat{v}(0) & = & 0                                                                                                      & \quad \mbox{in} & \bR^n_{[y > 0]},                   \\[0.5em]
		                                               \hat{w}(0) & = & 0                                                                                                      & \quad \mbox{in} & \bR^n_{[z > 0]}
	\end{array}
\end{equation*}
admit unique solutions
\begin{equation*}
	\begin{array}{rclcl}
		\hat{u} & \in & H^1_p(J,\,L_p(\bR^n_{[y > 0]}))^{n - 2} & \cap & L_p(J,\,H^2_p(\bR^n_{[y > 0]}))^{n - 2}, \\[0.5em]
		\hat{v} & \in & H^1_p(J,\,L_p(\bR^n_{[y > 0]}))         & \cap & L_p(J,\,H^2_p(\bR^n_{[y > 0]})),         \\[0.5em]
		\hat{w} & \in & H^1_p(J,\,L_p(\bR^n_{[z > 0]}))         & \cap & L_p(J,\,H^2_p(\bR^n_{[z > 0]})).
	\end{array}	
\end{equation*}
Now, if we define $u$, $v$, and $w$ to be the restrictions of $\hat{u}$, $\hat{v}$, and $\hat{w}$ to $\bR^n_{+\!\!\!+}$,
then $(u,\,v,\,w)$ together with $p := 0$ belong to the desired regularity class and solve
\begin{equation*}
	\begin{array}{rclll}
		                   \partial_t u - \frac{1}{\textrm{Re}} \Delta u +   \nabla_x p & = & 0                            & \quad \mbox{in} & J \times \bR^n_{+\!\!\!+},            \\[0.5em]
		                   \partial_t v - \frac{1}{\textrm{Re}} \Delta v + \partial_y p & = & 0                            & \quad \mbox{in} & J \times \bR^n_{+\!\!\!+},            \\[0.5em]
		                   \partial_t w - \frac{1}{\textrm{Re}} \Delta w + \partial_z p & = & 0                            & \quad \mbox{in} & J \times \bR^n_{+\!\!\!+},            \\[0.5em]
		- \frac{1}{\textrm{Re}} [\partial_y u]_y - \frac{1}{\textrm{Re}} \nabla_x [v]_y & = & h^{\textrm{\upshape wall}}_u & \quad \mbox{on} & J \times \partial_y \bR^n_{+\!\!\!+}, \\[0.5em]
		                                                                          [v]_y & = & h^{\textrm{\upshape wall}}_v & \quad \mbox{on} & J \times \partial_y \bR^n_{+\!\!\!+}, \\[0.5em]
		                                                                          [w]_z & = & u^{\textrm{\upshape in}}_w   & \quad \mbox{on} & J \times \partial_z \bR^n_{+\!\!\!+}, \\[0.5em]
		                                       u(0) = u_0, \quad v(0) = v_0, \quad w(0) & = & 0                            & \quad \mbox{in} & \bR^n_{+\!\!\!+}.
	\end{array}
\end{equation*}
Hence, we may assume $h^{\textrm{\upshape wall}}_u = h^{\textrm{\upshape wall}}_v = u^{\textrm{\upshape in}}_w = 0$ together with the assumption on $g$ stated above.
Note that this reduction of the problem neither affects the regularity and compatibility assumptions of \Thmref{Linear-Wedge-Inflow},
nor the simplifications obtained in Step~1, i.\,e.\ we may still assume
\begin{itemize}
	\item $g \in H^{1/2}_p(J,\,L_p(\bR^n_{+\!\!\!+})) \cap L_p(J,\,H^1_p(\bR^n_{+\!\!\!+}))$,
	\item $(u^{\textrm{\upshape in}}_u,\,u^{\textrm{\upshape in}}_v) \in W^{1 - 1/2p}_p(J,\,L_p(\partial_z \bR^n_{+\!\!\!+}))^{n - 1} \cap L_p(J,\,W^{2 - 1/p}_p(\partial_z \bR^n_{+\!\!\!+}))^{n - 1}$,
	\item $h^{\textrm{\upshape wall}}_w \in W^{1/2 - 1/2p}_p(J,\,L_p(\partial_y \bR^n_{+\!\!\!+})) \cap L_p(J,\,W^{1 - 1/p}_p(\partial_y \bR^n_{+\!\!\!+}))$
\end{itemize}
as well as the compatibility condition
\begin{equation}
	\eqnlabel{Wedge-Compatibility-4}
	\begin{array}{rclll}
		- \frac{1}{\textrm{Re}} [\partial_y u^{\textrm{\upshape in}}_u]_y & = & 0 & \quad \mbox{on} & J \times \cE, \\[0.5em]
		                                   [u^{\textrm{\upshape in}}_v]_y & = & 0 & \quad \mbox{on} & J \times \cE, \\[0.5em]
		                                 [h^{\textrm{\upshape wall}}_w]_z & = & 0 & \quad \mbox{on} & J \times \cE,
	\end{array}
\end{equation}
which stems from \eqnref{Wedge-Compatibility-3},
in the remaining part of the proof.

\subsection*{Step 3}
We show that we may w.\,l.\,o.\,g.\ assume $h^{\textrm{\upshape wall}}_w = 0$.
To accomplish this we define
\begin{equation*}
	\hat{h}^{\textrm{\upshape wall}}_w := E^-_z h^{\textrm{\upshape wall}}_w \in W^{1/2 - 1/2p}_p(J,\,L_p(\partial \bR^n_{[y > 0]})) \cap L_p(J,\,W^{1 - 1/p}_p(\partial \bR^n_{[y > 0]})),
\end{equation*}
where we denote by $E^-_z$ the odd extension operator w.\,r.\,t.\ $z$.
Note that $[h^{\textrm{\upshape wall}}_w]_z = 0$ thanks to the compatibility condition \eqnref{Wedge-Compatibility-4}.
This ensures that the odd extension of $h^{\textrm{\upshape wall}}_w$ w.\,r.\,t.\ $z$ has the desired spatial regularity.
Now, the problem
\begin{equation*}
	\begin{array}{rclll}
		\partial_t \hat{w} - \frac{1}{\textrm{Re}} \Delta \hat{w} & = & 0                                  & \quad \mbox{in} & J \times \bR^n_{[y > 0]},          \\[0.5em]
		           - \frac{1}{\textrm{Re}} [\partial_y \hat{w}]_y & = & \hat{h}^{\textrm{\upshape wall}}_w & \quad \mbox{on} & J \times \partial \bR^n_{[y > 0]}, \\[0.5em]
		                                               \hat{w}(0) & = & 0                                  & \quad \mbox{in} & \bR^n_{[y > 0]}
	\end{array}
\end{equation*}
admits a unique solution
\begin{equation*}
	\hat{w} \in H^1_p(J,\,L_p(\bR^n_{[y > 0]})) \cap L_p(J,\,H^2_p(\bR^n_{[y > 0]})),
\end{equation*}
which is odd w.\,r.\,t.\ $z$ by construction.
Hence, if we set $u = v = 0$ and define $w$ to be the restriction of $\hat{w}$ to $\bR^n_{+\!\!\!+}$,
then $(u,\,v,\,w)$ together with $p := 0$ belong to the desired regularity class and solve
\begin{equation*}
	\begin{array}{rclll}
		                     \partial_t u - \frac{1}{\textrm{Re}} \Delta u +   \nabla_x p & = & 0                            & \quad \mbox{in} & J \times \bR^n_{+\!\!\!+},            \\[0.5em]
		                     \partial_t v - \frac{1}{\textrm{Re}} \Delta v + \partial_y p & = & 0                            & \quad \mbox{in} & J \times \bR^n_{+\!\!\!+},            \\[0.5em]
		                     \partial_t w - \frac{1}{\textrm{Re}} \Delta w + \partial_z p & = & 0                            & \quad \mbox{in} & J \times \bR^n_{+\!\!\!+},            \\[0.5em]
		  - \frac{1}{\textrm{Re}} [\partial_y u]_y - \frac{1}{\textrm{Re}} \nabla_x [v]_y & = & 0                            & \quad \mbox{on} & J \times \partial_y \bR^n_{+\!\!\!+}, \\[0.5em]
		                                                                            [v]_y & = & 0                            & \quad \mbox{on} & J \times \partial_y \bR^n_{+\!\!\!+}, \\[0.5em]
		- \frac{1}{\textrm{Re}} [\partial_y w]_y - \frac{1}{\textrm{Re}} \partial_z [v]_y & = & h^{\textrm{\upshape wall}}_w & \quad \mbox{on} & J \times \partial_y \bR^n_{+\!\!\!+}, \\[0.5em]
		                                                                            [w]_z & = & 0                            & \quad \mbox{on} & J \times \partial_z \bR^n_{+\!\!\!+}, \\[0.5em]
		                                         u(0) = u_0, \quad v(0) = v_0, \quad w(0) & = & 0                            & \quad \mbox{in} & \bR^n_{+\!\!\!+}.
	\end{array}
\end{equation*}
Hence, we may assume $h^{\textrm{\upshape wall}}_w = 0$.
Note that this reduction of the problem neither affects the regularity and compatibility assumptions of \Thmref{Linear-Wedge-Inflow},
nor the simplifications obtained in Steps~1, and 2, i.\,e.\ we may still assume
\begin{itemize}
	\item $g \in H^1_p(J,\,{}_0 \dot{H}^{-1}_p(\bR^n_{+\!\!\!+})) \cap H^{1/2}_p(J,\,L_p(\bR^n_{+\!\!\!+})) \cap L_p(J,\,H^1_p(\bR^n_{+\!\!\!+}))$,
	\item $(u^{\textrm{\upshape in}}_u,\,u^{\textrm{\upshape in}}_v) \in W^{1 - 1/2p}_p(J,\,L_p(\partial_z \bR^n_{+\!\!\!+}))^{n - 1} \cap L_p(J,\,W^{2 - 1/p}_p(\partial_z \bR^n_{+\!\!\!+}))^{n - 1}$
\end{itemize}
as well as the compatibility condition
\begin{equation}
	\eqnlabel{Wedge-Compatibility-5}
	\begin{array}{rclll}
		- \frac{1}{\textrm{Re}} [\partial_y u^{\textrm{\upshape in}}_u]_y & = & 0 & \quad \mbox{on} & J \times \cE, \\[0.5em]
		                                   [u^{\textrm{\upshape in}}_v]_y & = & 0 & \quad \mbox{on} & J \times \cE,
	\end{array}
\end{equation}
which stems from \eqnref{Wedge-Compatibility-4},
in the remaining part of the proof.

\subsection*{Step 4}
Finally, we solve the reduced problem as obtained by Steps~1, 2, and 3.
To accomplish this we define
\begin{equation*}
	\hat{g} := E^+_y g \in H^1_p(J,\,{}_0 \dot{H}^{-1}_p(\bR^n_{[z > 0]})) \cap H^{1/2}_p(J,\,L_p(\bR^n_{[z > 0]})) \cap L_p(J,\,H^1_p(\bR^n_{[z > 0]}))
\end{equation*}
as well as
\begin{equation*}
	\begin{array}{rcrclcl}
		\hat{u}^{\textrm{\upshape in}}_u & := & E^+_y u^{\textrm{\upshape in}}_u & \in & W^{1 - 1/2p}_p(J,\,L_p(\partial \bR^n_{[z > 0]}))^{n - 2} & \cap & L_p(J,\,W^{2 - 1/p}_p(\partial \bR^n_{[z > 0]}))^{n - 2}, \\[0.5em]
		\hat{u}^{\textrm{\upshape in}}_v & := & E^-_y u^{\textrm{\upshape in}}_v & \in & W^{1 - 1/2p}_p(J,\,L_p(\partial \bR^n_{[z > 0]}))         & \cap & L_p(J,\,W^{2 - 1/p}_p(\partial \bR^n_{[z > 0]})),
	\end{array}
\end{equation*}
where we denote by $E^\pm_y$ the even, and odd extension operator w.\,r.\,t.\ $y$, respectively.
Note that
\begin{equation*}
	\langle \phi,\,F(\hat{g},\,0) \rangle = \!\!\! \int_{\bR^n_{[z > 0]}} \!\!\!\! \phi\,\hat{g}\,\mbox{d}V = \int_{\bR^n_{+\!\!\!+}} (1 + R^+_y) \phi\,g\,\mbox{d}V = \langle [(1 + R^+_y) \phi]|_{\bR^n_{+\!\!\!+}},\,F(g,\,0,\,0) \rangle
\end{equation*}
for $\phi \in H^1_{p^\prime}(\bR^n_{[z > 0]})$,
where $R^+_y$ denotes the even reflection operator w.\,r.\,t.\ $y$.
This implies that $\hat{g}$ has the desired temporal and spatial regularity.
Also note that $[\partial_y u^{\textrm{\upshape in}}_u]_y = [u^{\textrm{\upshape in}}_v]_y = 0$ thanks to the compatibility condition \eqnref{Wedge-Compatibility-5}.
This ensures that the even extension of $u^{\textrm{\upshape in}}_u$,
and the odd extension of $u^{\textrm{\upshape in}}_v$ w.\,r.\,t.\ $y$ have the desired spatial regularity.
Now, the Stokes equations in the halfspace
\begin{equation*}
	\begin{array}{rclll}
		\partial_t \hat{u} - \frac{1}{\textrm{Re}} \Delta \hat{u} +   \nabla_x \hat{p} & = & 0                                & \quad \mbox{in} & J \times \bR^n_{[z > 0]},            \\[0.5em]
		\partial_t \hat{v} - \frac{1}{\textrm{Re}} \Delta \hat{v} + \partial_y \hat{p} & = & 0                                & \quad \mbox{in} & J \times \bR^n_{[z > 0]},            \\[0.5em]
		\partial_t \hat{w} - \frac{1}{\textrm{Re}} \Delta \hat{w} + \partial_z \hat{p} & = & 0                                & \quad \mbox{in} & J \times \bR^n_{[z > 0]},            \\[0.5em]
		               \mbox{div}_x\,\hat{u} + \partial_y \hat{v} + \partial_z \hat{w} & = & \hat{g}                          & \quad \mbox{in} & J \times \bR^n_{[z > 0]},            \\[0.5em]
		                                                                   [\hat{u}]_z & = & \hat{u}^{\textrm{\upshape in}}_u & \quad \mbox{on} & J \times \partial_z \bR^n_{[z > 0]}, \\[0.5em]
		                                                                   [\hat{v}]_z & = & \hat{u}^{\textrm{\upshape in}}_v & \quad \mbox{on} & J \times \partial_z \bR^n_{[z > 0]}, \\[0.5em]
		                                                                   [\hat{w}]_z & = & 0                                & \quad \mbox{on} & J \times \partial_z \bR^n_{[z > 0]}, \\[0.5em]
		                        \hat{u}(0) = 0, \quad \hat{v}(0) = 0, \quad \hat{w}(0) & = & 0                                & \quad \mbox{in} & \bR^n_{[z > 0]},
	\end{array}
\end{equation*}
admit a unique solution
\begin{equation*}
	\begin{array}{rcl}
		(\hat{u},\,\hat{v},\,\hat{w}) & \in & H^1_p(J,\,L_p(\bR^n_{[z > 0]}))^n \cap L_p(J,\,H^2_p(\bR^n_{[z > 0]}))^n, \\[0.5em]
		                      \hat{p} & \in & L_p(J,\,\hat{H}^1_p(\bR^n_{[z > 0]}))
	\end{array}
\end{equation*}
thanks to \cite[Theorem~2.3]{Bothe-Koehne-Pruess:Energy-Preserving-Boundary-Conditions},
where $\hat{u}$, $\hat{w}$, and $\hat{p}$ are even while $\hat{v}$ is odd w.\,r.\,t.\ $y$ by construction.
Hence, the restrictions $u$, $v$, $w$, and $p$ of $\hat{u}$, $\hat{v}$, $\hat{w}$, and $\hat{p}$ to $\bR^n_{+\!\!\!+}$
belong to the desired regularity class and solve
\begin{equation*}
	\begin{array}{rclll}
		                     \partial_t u - \frac{1}{\textrm{Re}} \Delta u +   \nabla_x p & = & 0                          & \quad \mbox{in} & J \times \bR^n_{+\!\!\!+},            \\[0.5em]
		                     \partial_t v - \frac{1}{\textrm{Re}} \Delta v + \partial_y p & = & 0                          & \quad \mbox{in} & J \times \bR^n_{+\!\!\!+},            \\[0.5em]
		                     \partial_t w - \frac{1}{\textrm{Re}} \Delta w + \partial_z p & = & 0                          & \quad \mbox{in} & J \times \bR^n_{+\!\!\!+},            \\[0.5em]
		                                    \mbox{div}_x\,u + \partial_y v + \partial_z w & = & g                          & \quad \mbox{in} & J \times \bR^n_{+\!\!\!+},            \\[0.5em]
		  - \frac{1}{\textrm{Re}} [\partial_y u]_y - \frac{1}{\textrm{Re}} \nabla_x [v]_y & = & 0                          & \quad \mbox{on} & J \times \partial_y \bR^n_{+\!\!\!+}, \\[0.5em]
		                                                                            [v]_y & = & 0                          & \quad \mbox{on} & J \times \partial_y \bR^n_{+\!\!\!+}, \\[0.5em]
		- \frac{1}{\textrm{Re}} [\partial_y w]_y - \frac{1}{\textrm{Re}} \partial_z [v]_y & = & 0                          & \quad \mbox{on} & J \times \partial_y \bR^n_{+\!\!\!+}, \\[0.5em]
		                                                                            [u]_z & = & u^{\textrm{\upshape in}}_u & \quad \mbox{on} & J \times \partial_z \bR^n_{+\!\!\!+}, \\[0.5em]
		                                                                            [v]_z & = & u^{\textrm{\upshape in}}_v & \quad \mbox{on} & J \times \partial_z \bR^n_{+\!\!\!+}, \\[0.5em]
		                                                                            [w]_z & = & 0                          & \quad \mbox{on} & J \times \partial_z \bR^n_{+\!\!\!+}, \\[0.5em]
		                                             u(0) = 0, \quad v(0) = 0, \quad w(0) & = & 0                          & \quad \mbox{in} & \bR^n_{+\!\!\!+},
	\end{array}
\end{equation*}
which is the reduced form of problem \eqnref{Wedge-Problem} after Steps~1, 2, and 3.
This completes the proof of \Thmref{Linear-Wedge-Inflow}.

\subsection{Combination of Dynamic Outflow/Navier Conditions}
In order to prove \Thmref{Linear-Wedge-Non-Reflecting},
in addition to the assumptions and the notation introduced at the beginning of this section,
we assume $v^{\textrm{out}} = V \nu$ with $\alpha V + \frac{2}{\textrm{Re}} > \alpha V + \frac{1}{\textrm{Re}} > 0$.
We have to study the model problem
\begin{equation}
	\eqnlabel{Wedge-Problem-NR}
	\begin{array}{rclll}
		                     \partial_t u - \frac{1}{\textrm{Re}} \Delta u +   \nabla_x p & = & f_u                          & \quad \mbox{in} & J \times \bR^n_{+\!\!\!+},            \\[0.5em]
		                     \partial_t v - \frac{1}{\textrm{Re}} \Delta v + \partial_y p & = & f_v                          & \quad \mbox{in} & J \times \bR^n_{+\!\!\!+},            \\[0.5em]
		                     \partial_t w - \frac{1}{\textrm{Re}} \Delta w + \partial_z p & = & f_w                          & \quad \mbox{in} & J \times \bR^n_{+\!\!\!+},            \\[0.5em]
		                                    \mbox{div}_x\,u + \partial_y v + \partial_z w & = & g                            & \quad \mbox{in} & J \times \bR^n_{+\!\!\!+},            \\[0.5em]
		  - \frac{1}{\textrm{Re}} [\partial_y u]_y - \frac{1}{\textrm{Re}} \nabla_x [v]_y & = & h^{\textrm{\upshape wall}}_u & \quad \mbox{on} & J \times \partial_y \bR^n_{+\!\!\!+}, \\[0.5em]
		                                                                            [v]_y & = & h^{\textrm{\upshape wall}}_v & \quad \mbox{on} & J \times \partial_y \bR^n_{+\!\!\!+}, \\[0.5em]
		- \frac{1}{\textrm{Re}} [\partial_y w]_y - \frac{1}{\textrm{Re}} \partial_z [v]_y & = & h^{\textrm{\upshape wall}}_w & \quad \mbox{on} & J \times \partial_y \bR^n_{+\!\!\!+}, \\[0.5em]
		                                         u(0) = u_0, \quad v(0) = v_0, \quad w(0) & = & w_0                          & \quad \mbox{in} & \bR^n_{+\!\!\!+}
	\end{array}
\end{equation}
together with a dynamic outflow boundary condition in tangential directions
\begin{subequations}
\begin{equation}
	\eqnlabel{Wedge-Problem-TNR}
	\begin{array}{rclll}
		  \alpha \partial_t [u]_z - (\alpha V + \frac{1}{\textrm{Re}}) [\partial_z u]_z - \frac{1}{\textrm{Re}} \nabla_x [w]_z & = & h_u & \quad \mbox{on} & J \times \partial_z \bR^n_{+\!\!\!+}, \\[0.5em]
		\alpha \partial_t [v]_z - (\alpha V + \frac{1}{\textrm{Re}}) [\partial_z v]_z - \frac{1}{\textrm{Re}} \partial_y [w]_z & = & h_v & \quad \mbox{on} & J \times \partial_z \bR^n_{+\!\!\!+}, \\[0.5em]
		                                                                                                                 [w]_z & = & h_w & \quad \mbox{on} & J \times \partial_z \bR^n_{+\!\!\!+},
	\end{array}
\end{equation}
or a dynamic outflow boundary condition in normal direction
\begin{equation}
	\eqnlabel{Wedge-Problem-NNR}
	\begin{array}{rclll}
		                                                                                [u]_z & = & h_u & \quad \mbox{on} & J \times \partial_z \bR^n_{+\!\!\!+}, \\[0.5em]
		                                                                                [v]_z & = & h_v & \quad \mbox{on} & J \times \partial_z \bR^n_{+\!\!\!+}, \\[0.5em]
		\alpha \partial_t [w]_z - (\alpha V + \frac{2}{\textrm{Re}}) [\partial_z w]_z + [p]_z & = & h_w & \quad \mbox{on} & J \times \partial_z \bR^n_{+\!\!\!+},
	\end{array}
\end{equation}
or a full dynamic outflow boundary condition
\begin{equation}
	\eqnlabel{Wedge-Problem-FNR}
	\begin{array}{rclll}
		  \alpha \partial_t [u]_z - (\alpha V + \frac{1}{\textrm{Re}}) [\partial_z u]_z - \frac{1}{\textrm{Re}} \nabla_x [w]_z & = & h_u & \quad \mbox{on} & J \times \partial_z \bR^n_{+\!\!\!+}, \\[0.5em]
		\alpha \partial_t [v]_z - (\alpha V + \frac{1}{\textrm{Re}}) [\partial_z v]_z - \frac{1}{\textrm{Re}} \partial_y [w]_z & = & h_v & \quad \mbox{on} & J \times \partial_z \bR^n_{+\!\!\!+}, \\[0.5em]
		                                 \alpha \partial_t [w]_z - (\alpha V + \frac{2}{\textrm{Re}}) [\partial_z w]_z + [p]_z & = & h_w & \quad \mbox{on} & J \times \partial_z \bR^n_{+\!\!\!+},
	\end{array}
\end{equation}
\end{subequations}
where the data $f = (f_u,\,f_v,\,f_w)$, $g$, $h = (h_u,\,h_v,\,h_w)$,
$h^{\textrm{\upshape wall}} = (h^{\textrm{\upshape wall}}_u,\,h^{\textrm{\upshape wall}}_v,\,h^{\textrm{\upshape wall}}_w)$, and
$(u_0,\,v_0,\,w_0)$ are subject to the regularity/compatibility conditions stated in \Thmref{Linear-Wedge-Non-Reflecting}, i.\,e.
\begin{itemize}
	\item $f \in L_p(J \times \bR^n_{+\!\!\!+})^n$,
	\item $g \in H^{1/2}_p(J,\,L_p(\bR^n_{+\!\!\!+})) \cap L_p(J,\,H^1_p(\bR^n_{+\!\!\!+}))$,
	\item $(h_u,\,h_v) \in W^{\kappa/2 - 1/2p}_p(J,\,L_p(\partial_z \bR^n_{+\!\!\!+}))^{n - 1} \cap L_p(J,\,W^{\kappa - 1/p}_p(\partial_z \bR^n_{+\!\!\!+}))^{n - 1}$,
	\item $h_w \in L_p(J,\,W^{1 - 1/p}_p(\partial_z \bR^n_{+\!\!\!+}))$,
	\item $(h^{\textrm{\upshape wall}}_u,\,h^{\textrm{\upshape wall}}_w) \in W^{1/2 - 1/2p}_p(J,\,L_p(\partial_y \bR^n_{+\!\!\!+}))^{n - 1} \cap L_p(J,\,W^{1 - 1/p}_p(\partial_y \bR^n_{+\!\!\!+}))^{n - 1}$,
	\item $h^{\textrm{\upshape wall}}_v \in W^{1 - 1/2p}_p(J,\,L_p(\partial_y \bR^n_{+\!\!\!+})) \cap L_p(J,\,W^{2 - 1/p}_p(\partial_y \bR^n_{+\!\!\!+}))$,
	\item $F(g,\,- h^{\textrm{\upshape wall}}_v,\,- \eta) \in H^1_p(J,\,\hat{H}^{-1}_p(\bR^n_{+\!\!\!+}))$,
	\item $(u_0,\,v_0,\,w_0) \in W^{2 - 2/p}_p(\bR^n_{+\!\!\!+})^n$,
	\item $([u_0]_z,\,[v_0]_z) \in W^{2 - (\kappa + 1)/p}_p(\partial_z \bR^n_{+\!\!\!+})^n$
\end{itemize}
with $\kappa = 1$ for conditions \eqnref{Wedge-Problem-TNR}, and \eqnref{Wedge-Problem-FNR}, while $\kappa = 2$ for condition \eqnref{Wedge-Problem-NNR}.
Furthermore, $\eta = h_w \in W^{1 - 1/2p}_p(J,\,L_p(\partial_z \bR^n_{+\!\!\!+})) \cap L_p(J,\,W^{2 - 1/p}_p(\partial_z \bR^n_{+\!\!\!+}))$ for condition \eqnref{Wedge-Problem-TNR},
while otherwise $\eta \in H^1_p(J,\,W^{1 - 1/p}_p(\partial_z \bR^n_{+\!\!\!+})) \cap L_p(J,\,W^{2 - 1/p}_p(\partial_z \bR^n_{+\!\!\!+}))$ is given by assumption.
In all cases the compatibility conditions \eqnref{Wedge-Compatibility-1}, \eqnref{Wedge-Compatibility-2a} are satisfied,
and, due to \eqnref*{Non-Reflecting-Wall-Compatibility-Tangential}, \eqnref*{Non-Reflecting-Wall-Compatibility-Normal}, or \eqnref*{Non-Reflecting-Wall-Compatibility-Full} we have
\begin{subequations}
\begin{equation}
	\eqnlabel{Wedge-Compatibility-6}
	\begin{array}{rclll}
		 - \frac{1}{\textrm{Re}} [\partial_y \xi_u]_y - \frac{1}{\textrm{Re}} \nabla_x [h^{\textrm{\upshape wall}}_v]_z & = & [h^{\textrm{\upshape wall}}_u]_z & \quad \mbox{on} & J \times \cE, \\[0.5em]
		                                                                                                      [\xi_v]_y & = & [h^{\textrm{\upshape wall}}_v]_z & \quad \mbox{on} & J \times \cE, \\[0.5em]
		- \frac{1}{\textrm{Re}} [\partial_y \eta]_y - \frac{1}{\textrm{Re}} [\partial_z h^{\textrm{\upshape wall}}_v]_z & = & [h^{\textrm{\upshape wall}}_w]_z & \quad \mbox{on} & J \times \cE
	\end{array}
\end{equation}
for $p \geq 2$ with $(\xi_u,\,\xi_v) = (h_u,\,h_v)$ for condition \eqnref{Wedge-Problem-NNR} while otherwise
\begin{equation*}
	(\xi_u,\,\xi_v) \in W^{3/2 - 1/2p}_p(J,\,L_p(\partial_z \bR^n_{+\!\!\!+}))^{n - 1} \cap H^1_p(J,\,W^{1 - 1/p}_p(\partial_z \bR^n_{+\!\!\!+}))^{n - 1} \cap L_p(J,\,W^{2 - 2/p}_p(\partial_z \bR^n_{+\!\!\!+}))^{n - 1}
\end{equation*}
is given by assumption such that we always have
\begin{equation}
	\eqnlabel{Wedge-Compatibility-2c}
	\begin{array}{rclllll}
		[u_0]_z & = & \xi_u(0) & \quad \mbox{on} & \partial_z \bR^n_{+\!\!\!+}, & \quad \mbox{if} & p > \frac{3}{2}, \\[0.5em]
		[v_0]_z & = & \xi_v(0) & \quad \mbox{on} & \partial_z \bR^n_{+\!\!\!+}, & \quad \mbox{if} & p > \frac{3}{2}, \\[0.5em]
		[w_0]_z & = & \eta(0)  & \quad \mbox{on} & \partial_z \bR^n_{+\!\!\!+}, & \quad \mbox{if} & p > \frac{3}{2}.
	\end{array}
\end{equation}
In case of the boundary conditions \eqnref{Wedge-Problem-TNR} and \eqnref{Wedge-Problem-FNR} the compatibility conditions
\eqnref*{Non-Reflecting-Wall-Compatibility-Tangential} and \eqnref*{Non-Reflecting-Wall-Compatibility-Full} imply
\begin{equation}
	\eqnlabel{Wedge-Compatibility-7}
	\alpha \partial_t [h^{\textrm{\upshape wall}}_v]_z - \alpha V [\partial_z h^{\textrm{\upshape wall}}_v]_z + [h^{\textrm{\upshape wall}}_w]_z = [h_v]_y \qquad \mbox{on} \ J \times \cE.
\end{equation}
\end{subequations}
The construction of solutions to this problems requires two steps.

\subsection*{Step 1}
We first show that we can w.\,l.\,o.\,g.\ assume $f = g = h^{\textrm{\upshape wall}} = \xi = \eta = 0$
together with $u_0 = v_0 = w_0 = 0$ as well as $h_u = h_v = 0$ in case condition \eqnref{Wedge-Problem-NNR} is applied,
and $h_w = 0$ in case condition \eqnref{Wedge-Problem-TNR} is applied.
Indeed, thanks to \eqnref{Wedge-Compatibility-1}, \eqnref{Wedge-Compatibility-2a}, \eqnref{Wedge-Compatibility-6} and \eqnref{Wedge-Compatibility-2c}
all necessary regularity and compatibility conditions are satisfied in order to apply \Thmref{Linear-Wedge-Inflow} to solve
\begin{equation*}
	\begin{array}{rclll}
		                     \partial_t u - \frac{1}{\textrm{Re}} \Delta u +   \nabla_x p & = & f_u                          & \quad \mbox{in} & J \times \bR^n_{+\!\!\!+},            \\[0.5em]
		                     \partial_t v - \frac{1}{\textrm{Re}} \Delta v + \partial_y p & = & f_v                          & \quad \mbox{in} & J \times \bR^n_{+\!\!\!+},            \\[0.5em]
		                     \partial_t w - \frac{1}{\textrm{Re}} \Delta w + \partial_z p & = & f_w                          & \quad \mbox{in} & J \times \bR^n_{+\!\!\!+},            \\[0.5em]
		                                    \mbox{div}_x\,u + \partial_y v + \partial_z w & = & g                            & \quad \mbox{in} & J \times \bR^n_{+\!\!\!+},            \\[0.5em]
		  - \frac{1}{\textrm{Re}} [\partial_y u]_y - \frac{1}{\textrm{Re}} \nabla_x [v]_y & = & h^{\textrm{\upshape wall}}_u & \quad \mbox{on} & J \times \partial_y \bR^n_{+\!\!\!+}, \\[0.5em]
		                                                                            [v]_y & = & h^{\textrm{\upshape wall}}_v & \quad \mbox{on} & J \times \partial_y \bR^n_{+\!\!\!+}, \\[0.5em]
		- \frac{1}{\textrm{Re}} [\partial_y w]_y - \frac{1}{\textrm{Re}} \partial_z [v]_y & = & h^{\textrm{\upshape wall}}_w & \quad \mbox{on} & J \times \partial_y \bR^n_{+\!\!\!+}, \\[0.5em]
		                                                                            [u]_z & = & \xi_u                        & \quad \mbox{on} & J \times \partial_z \bR^n_{+\!\!\!+}, \\[0.5em]
		                                                                            [v]_z & = & \xi_v                        & \quad \mbox{on} & J \times \partial_z \bR^n_{+\!\!\!+}, \\[0.5em]
		                                                                            [w]_z & = & \eta                         & \quad \mbox{on} & J \times \partial_z \bR^n_{+\!\!\!+}, \\[0.5em]
		                                         u(0) = u_0, \quad v(0) = v_0, \quad w(0) & = & w_0                          & \quad \mbox{in} & \bR^n_{+\!\!\!+}.
	\end{array}
\end{equation*}
Then $(u,\,v,\,w)$, and $p$ belong to the desired regularity class,
except that the pressure trace only satisfies $[p]_z \in L_p(J,\,\dot{W}^1_p(\partial_z \bR^n_{+\!\!\!+}))$.
This shows that we can assume $f = g = h^{\textrm{\upshape wall}} = \xi = \eta = 0$
together with $u_0 = v_0 = w_0 = 0$ as well as $h_u = h_v = 0$ in case condition \eqnref{Wedge-Problem-NNR} is applied,
and $h_w = 0$ in case condition \eqnref{Wedge-Problem-TNR} is applied.
Note that this reduction of the problem does not affect the regularity and compatibility assumptions of \Thmref{Linear-Wedge-Non-Reflecting},
except for a potential lower regularity of $h_w$ that stems from the potential lower regularity of $[p]_z$, i.\,e.\ we may now assume
\begin{itemize}
	\item $(h_u,\,h_v) \in W^{1/2 - 1/2p}_p(J,\,L_p(\partial_z \bR^n_{+\!\!\!+}))^{n - 1} \cap L_p(J,\,W^{1 - 1/p}_p(\partial_z \bR^n_{+\!\!\!+}))^{n - 1}$,
	\item $h_w \in L_p(J,\,\dot{W}^{1 - 1/p}_p(\partial_z \bR^n_{+\!\!\!+}))$
\end{itemize}
as well as the compatibility condition
\begin{equation}
	\eqnlabel{Wedge-Compatibility-8}
	[h_v]_y = 0 \qquad \mbox{on} \ J \times \cE,
\end{equation}
which stems from \eqnref{Wedge-Compatibility-7},
with $h_u = h_v = 0$ in case of condition \eqnref{Wedge-Problem-NNR}, and $h_w = 0$ in case of condition \eqnref{Wedge-Problem-TNR}.

\subsection*{Step 2}
Finally, we solve the reduced problem as obtained by Step~1.
To accomplish this we define
\begin{equation*}
	\begin{array}{rcrclcl}
		\hat{h}_u & := & E^+_y h_u & \in & W^{1/2 - 1/2p}_p(J,\,L_p(\partial \bR^n_{[z > 0]}))^{n - 2} & \cap & L_p(J,\,W^{1 - 1/p}_p(\partial \bR^n_{[z > 0]}))^{n - 2}, \\[0.5em]
		\hat{h}_v & := & E^-_y h_v & \in & W^{1/2 - 1/2p}_p(J,\,L_p(\partial \bR^n_{[z > 0]}))         & \cap & L_p(J,\,W^{1 - 1/p}_p(\partial \bR^n_{[z > 0]})),
	\end{array}
\end{equation*}
in case condition \eqnref{Wedge-Problem-TNR} or \eqnref{Wedge-Problem-FNR} is applied as well as
\begin{equation*}
	\hat{h}_w := E^+_y h_w \in L_p(J,\,\dot{W}^{1 - 1/p}_p(\partial \bR^n_{[z > 0]}))
\end{equation*}
in case condition \eqnref{Wedge-Problem-NNR} or \eqnref{Wedge-Problem-FNR} is applied.
Note that $[h_v]_y = 0$ thanks to the compatibility condition \eqnref{Wedge-Compatibility-8}.
This ensures that the odd extension of $h_v$ w.\,r.\,t.\ $y$ have the desired spatial regularity.
Now, the Stokes equations in the halfspace
\begin{equation*}
	\begin{array}{rclll}
		\partial_t \hat{u} - \frac{1}{\textrm{Re}} \Delta \hat{u} +   \nabla_x \hat{p} & = & 0 & \quad \mbox{in} & J \times \bR^n_{[z > 0]}, \\[0.5em]
		\partial_t \hat{v} - \frac{1}{\textrm{Re}} \Delta \hat{v} + \partial_y \hat{p} & = & 0 & \quad \mbox{in} & J \times \bR^n_{[z > 0]}, \\[0.5em]
		\partial_t \hat{w} - \frac{1}{\textrm{Re}} \Delta \hat{w} + \partial_z \hat{p} & = & 0 & \quad \mbox{in} & J \times \bR^n_{[z > 0]}, \\[0.5em]
		               \mbox{div}_x\,\hat{u} + \partial_y \hat{v} + \partial_z \hat{w} & = & 0 & \quad \mbox{in} & J \times \bR^n_{[z > 0]}, \\[0.5em]
		                        \hat{u}(0) = 0, \quad \hat{v}(0) = 0, \quad \hat{w}(0) & = & 0 & \quad \mbox{in} & \bR^n_{[z > 0]},
	\end{array}
\end{equation*}
together with the dynamic outflow boundary condition in tangential directions
\begin{equation*}
	\begin{array}{rclll}
		  \alpha \partial_t [\hat{u}]_z - (\alpha V + \frac{1}{\textrm{Re}}) [\partial_z \hat{u}]_z - \frac{1}{\textrm{Re}} \nabla_x [\hat{w}]_z & = & \hat{h}_u & \quad \mbox{on} & J \times \partial \bR^n_{[z > 0]}, \\[0.5em]
		\alpha \partial_t [\hat{v}]_z - (\alpha V + \frac{1}{\textrm{Re}}) [\partial_z \hat{v}]_z - \frac{1}{\textrm{Re}} \partial_y [\hat{w}]_z & = & \hat{h}_v & \quad \mbox{on} & J \times \partial \bR^n_{[z > 0]}, \\[0.5em]
		                                                                                                                             [\hat{w}]_z & = & 0         & \quad \mbox{on} & J \times \partial \bR^n_{[z > 0]},
	\end{array}
\end{equation*}
the dynamic outflow boundary condition in normal direction
\begin{equation*}
	\begin{array}{rclll}
		                                                                                            [\hat{u}]_z & = & 0         & \quad \mbox{on} & J \times \partial \bR^n_{[z > 0]}, \\[0.5em]
		                                                                                            [\hat{v}]_z & = & 0         & \quad \mbox{on} & J \times \partial \bR^n_{[z > 0]}, \\[0.5em]
		\alpha \partial_t [\hat{w}]_z - (\alpha V + \frac{2}{\textrm{Re}}) [\partial_z \hat{w}]_z + [\hat{p}]_z & = & \hat{h}_w & \quad \mbox{on} & J \times \partial \bR^n_{[z > 0]},
	\end{array}
\end{equation*}
the full dynamic outflow boundary condition
\begin{equation*}
	\begin{array}{rclll}
		  \alpha \partial_t [\hat{u}]_z - (\alpha V + \frac{1}{\textrm{Re}}) [\partial_z \hat{u}]_z - \frac{1}{\textrm{Re}} \nabla_x [\hat{w}]_z & = & \hat{h}_u & \quad \mbox{on} & J \times \partial \bR^n_{[z > 0]}, \\[0.5em]
		\alpha \partial_t [\hat{v}]_z - (\alpha V + \frac{1}{\textrm{Re}}) [\partial_z \hat{v}]_z - \frac{1}{\textrm{Re}} \partial_y [\hat{w}]_z & = & \hat{h}_v & \quad \mbox{on} & J \times \partial \bR^n_{[z > 0]}, \\[0.5em]
		                                 \alpha \partial_t [\hat{w}]_z - (\alpha V + \frac{2}{\textrm{Re}}) [\partial_z \hat{w}]_z + [\hat{p}]_z & = & \hat{h}_w & \quad \mbox{on} & J \times \partial \bR^n_{[z > 0]},
	\end{array}
\end{equation*}
respectively, admit a unique solution
\begin{equation*}
	\begin{array}{rcl}
		(\hat{u},\,\hat{v},\,\hat{w}) & \in & H^1_p(J,\,L_p(\bR^n_{[z > 0]}))^n \cap L_p(J,\,H^2_p(\bR^n_{[z > 0]}))^n, \\[0.5em]
		                      \hat{p} & \in & L_p(J,\,\hat{H}^1_p(\bR^n_{[z > 0]}))
	\end{array}
\end{equation*}
with increased regularity of $[\hat{u}]_z$, $[\hat{v}]_z$, and $[\hat{w}]_z$ according to the dynamic boundary condition
thanks to \Thmref{Linear-Halfspace} and \Remref{Linear-Halfspace}~(a),
where $\hat{u}$, $\hat{w}$, and $\hat{p}$ are even while $\hat{v}$ is odd w.\,r.\,t.\ $y$ by construction.
Hence, the restrictions $u$, $v$, $w$, and $p$ of $\hat{u}$, $\hat{v}$, $\hat{w}$, and $\hat{p}$ to $\bR^n_{+\!\!\!+}$
belong to the desired regularity class and solve
\begin{equation*}
	\begin{array}{rclll}
		                     \partial_t u - \frac{1}{\textrm{Re}} \Delta u +   \nabla_x p & = & 0 & \quad \mbox{in} & J \times \bR^n_{+\!\!\!+},            \\[0.5em]
		                     \partial_t v - \frac{1}{\textrm{Re}} \Delta v + \partial_y p & = & 0 & \quad \mbox{in} & J \times \bR^n_{+\!\!\!+},            \\[0.5em]
		                     \partial_t w - \frac{1}{\textrm{Re}} \Delta w + \partial_z p & = & 0 & \quad \mbox{in} & J \times \bR^n_{+\!\!\!+},            \\[0.5em]
		                                    \mbox{div}_x\,u + \partial_y v + \partial_z w & = & 0 & \quad \mbox{in} & J \times \bR^n_{+\!\!\!+},            \\[0.5em]
		  - \frac{1}{\textrm{Re}} [\partial_y u]_y - \frac{1}{\textrm{Re}} \nabla_x [v]_y & = & 0 & \quad \mbox{on} & J \times \partial_y \bR^n_{+\!\!\!+}, \\[0.5em]
		                                                                            [v]_y & = & 0 & \quad \mbox{on} & J \times \partial_y \bR^n_{+\!\!\!+}, \\[0.5em]
		- \frac{1}{\textrm{Re}} [\partial_y w]_y - \frac{1}{\textrm{Re}} \partial_z [v]_y & = & 0 & \quad \mbox{on} & J \times \partial_y \bR^n_{+\!\!\!+}, \\[0.5em]
		                                             u(0) = 0, \quad v(0) = 0, \quad w(0) & = & 0 & \quad \mbox{in} & \bR^n_{+\!\!\!+},
	\end{array}
\end{equation*}
together with the dynamic outflow boundary condition in tangential directions
\begin{equation*}
	\begin{array}{rclll}
		  \alpha \partial_t [u]_z - (\alpha V + \frac{1}{\textrm{Re}}) [\partial_z u]_z - \frac{1}{\textrm{Re}} \nabla_x [w]_z & = & h_u & \quad \mbox{on} & J \times \partial_z \bR^n_{+\!\!\!+}, \\[0.5em]
		\alpha \partial_t [v]_z - (\alpha V + \frac{1}{\textrm{Re}}) [\partial_z v]_z - \frac{1}{\textrm{Re}} \partial_y [w]_z & = & h_v & \quad \mbox{on} & J \times \partial_z \bR^n_{+\!\!\!+}, \\[0.5em]
		                                                                                                                 [w]_z & = & 0   & \quad \mbox{on} & J \times \partial_z \bR^n_{+\!\!\!+},
	\end{array}
\end{equation*}
the dynamic outflow boundary condition in normal direction
\begin{equation*}
	\begin{array}{rclll}
		                                                                                [u]_z & = & 0   & \quad \mbox{on} & J \times \partial_z \bR^n_{+\!\!\!+}, \\[0.5em]
		                                                                                [v]_z & = & 0   & \quad \mbox{on} & J \times \partial_z \bR^n_{+\!\!\!+}, \\[0.5em]
		\alpha \partial_t [w]_z - (\alpha V + \frac{2}{\textrm{Re}}) [\partial_z w]_z + [p]_z & = & h_w & \quad \mbox{on} & J \times \partial_z \bR^n_{+\!\!\!+},
	\end{array}
\end{equation*}
the full dynamic outflow boundary condition
\begin{equation*}
	\begin{array}{rclll}
		  \alpha \partial_t [u]_z - (\alpha V + \frac{1}{\textrm{Re}}) [\partial_z u]_z - \frac{1}{\textrm{Re}} \nabla_x [w]_z & = & h_u & \quad \mbox{on} & J \times \partial_z \bR^n_{+\!\!\!+}, \\[0.5em]
		\alpha \partial_t [v]_z - (\alpha V + \frac{1}{\textrm{Re}}) [\partial_z v]_z - \frac{1}{\textrm{Re}} \partial_y [w]_z & = & h_v & \quad \mbox{on} & J \times \partial_z \bR^n_{+\!\!\!+}, \\[0.5em]
		                                 \alpha \partial_t [w]_z - (\alpha V + \frac{2}{\textrm{Re}}) [\partial_z w]_z + [p]_z & = & h_w & \quad \mbox{on} & J \times \partial_z \bR^n_{+\!\!\!+},
	\end{array}
\end{equation*}
respectively, which is the reduced form of problem \eqnref{Wedge-Problem-NR}
together with \eqnref{Wedge-Problem-TNR}, \eqnref{Wedge-Problem-NNR}, \eqnref{Wedge-Problem-FNR}, respectively, after Step~1.
This completes the proof of \Thmref{Linear-Wedge-Non-Reflecting}.

\appendix
\section*{Appendix A \\ Parabolic Equations subject to Dynamic Boundary conditions}
\seclabel{Appendix-Parabolic}\renewcommand{\thesection}{A}
In this appendix we collect some useful results on parabolic equations subject to dynamic boundary conditions.
The first result is essentially contained in \cite{Denk-Pruess-Zacher:Dynamic-Boundary-Conditions}.
\begin{proposition}
\proplabel{Parabolic-DBC-Standard}
Let $0 < a \leq \infty$, let $J := (0,\,a)$, let $\epsilon \geq 0$ with $\epsilon > 0$, if $a = \infty$.
Let $\Omega := \bR^n_+$ with $\Gamma = \partial \Omega$.
Let $1 < p < \infty$ with $p \neq \frac{3}{2},\,3$ and let $\alpha,\,\beta,\,\mu > 0$.
Then for every
\begin{equation*}
	f \in L_p(J \times \Omega), \quad h \in W^{1/2 - 1/2p}_p(J,\,L_p(\Gamma)) \cap L_p(J,\,W^{1 - 1/p}_p(\Gamma)), \quad u_0 \in W^{2 - 2/p}_p(\Omega)
\end{equation*}
with $[u_0]_\Gamma \in W^{2 - 2/p}_p(\Gamma)$ the parabolic problem
\begin{equation*}
	\begin{array}{rclll}
		                      \epsilon u + \partial_t u - \mu \Delta u & = & f   & \quad \mbox{in} & J \times \Omega, \\[0.5em]
		\alpha \epsilon u + \alpha \partial_t u + \beta \partial_\nu u & = & h   & \quad \mbox{on} & J \times \Gamma, \\[0.5em]
		                                                          u(0) & = & u_0 & \quad \mbox{in} & \Omega
	\end{array}
\end{equation*}
admits a unique maximal regular solution
\begin{equation*}
	\begin{array}{rcl}
		         u & \in & H^1_p(J,\,L_p(\Omega)) \cap L_p(J,\,H^2_p(\Omega)), \\[0.5em]
		[u]_\Gamma & \in & W^{3/2 - 1/2p}_p(J,\,L_p(\Omega)) \cap H^1_p(J,\,W^{1 - 1/p}_p(\Gamma)) \cap L_p(J,\,W^{2 - 1/p}_p(\Gamma)).
	\end{array}
\end{equation*}
The solutions depend continuously on the data.
\end{proposition}
Concerning the proof, first note that the problem fits into the framework of \cite[Theorem~2.1]{Denk-Pruess-Zacher:Dynamic-Boundary-Conditions};
cf.~also~\cite[Example~3.1]{Denk-Pruess-Zacher:Dynamic-Boundary-Conditions}.
Strictly speaking, \cite[Theorem~2.1]{Denk-Pruess-Zacher:Dynamic-Boundary-Conditions} is formulated
for the case that $\Gamma$ is a sufficiently smooth, compact manifold, $a < \infty$, and $\epsilon = 0$.
However, the proof given in \cite{Denk-Pruess-Zacher:Dynamic-Boundary-Conditions} employs a localization procedure,
where the problem in the halfspace $\Omega = \bR^n_+$ with $a = \infty$, and $\epsilon > 0$ is a model problem,
which is dealt with in \cite[Section~4]{Denk-Pruess-Zacher:Dynamic-Boundary-Conditions}.

Now, in order to obtain maximal regular solutions for the Stokes equations subject to a dynamic boundary condition involving the pressure
it is necessary to have a result at hand that requires a lower regularity for the right-hand side of the boundary condition.
\begin{proposition}
\proplabel{Parabolic-DBC-Low-Regularity}
Let $0 < a \leq \infty$, let $J := (0,\,a)$, let $\epsilon \geq 0$ with $\epsilon > 0$, if $a = \infty$.
Let $\Omega := \bR^n_+$ with $\Gamma = \partial \Omega$.
Let $1 < p < \infty$ with $p \neq \frac{3}{2},\,3$ and let $\alpha,\,\beta,\,\mu > 0$.
Then for every
\begin{equation*}
	f \in L_p(J \times \Omega), \qquad h \in L_p(J,\,W^{1 - 1/p}_p(\Gamma)), \qquad u_0 \in W^{2 - 2/p}_p(\Omega)
\end{equation*}
with $[u_0]_\Gamma \in W^{2 - 2/p}_p(\Gamma)$ the parabolic problem
\begin{equation*}
	\begin{array}{rclll}
		                      \epsilon u + \partial_t u - \mu \Delta u & = & f   & \quad \mbox{in} & J \times \Omega, \\[0.5em]
		\alpha \epsilon u + \alpha \partial_t u + \beta \partial_\nu u & = & h   & \quad \mbox{on} & J \times \Gamma, \\[0.5em]
		                                                          u(0) & = & u_0 & \quad \mbox{in} & \Omega
	\end{array}
\end{equation*}
admits a unique maximal regular solution
\begin{equation*}
	\begin{array}{rcl}
		         u & \in & H^1_p(J,\,L_p(\Omega)) \cap L_p(J,\,H^2_p(\Omega)), \\[0.5em]
		[u]_\Gamma & \in & H^1_p(J,\,W^{1 - 1/p}_p(\Gamma)) \cap L_p(J,\,W^{2 - 1/p}_p(\Gamma)).
	\end{array}
\end{equation*}
The solutions depend continuously on the data.
\end{proposition}
\begin{proof}
It is sufficient to consider the case $a = \infty$, and $\epsilon > 0$.
Moreover, using \Propref{Parabolic-DBC-Standard} we may assume $f = u_0 = 0$ in the following.
A Laplace transformation w.\,r.\,t.\ time and a Fourier transformation w.\,r.\,t.\ the tangential spatial variables leads to
\begin{equation*}
	\begin{array}{rcllll}
		                       \omega^2 \hat u - \mu \partial_y^2 \hat u & = & 0,       & \quad \lambda \in \Sigma_{\pi - \theta}, & \xi \in \bR^{n - 1}, & y > 0, \\[0.5em]
		\alpha \lambda_\epsilon [\hat u]_y - \beta [\partial_y \hat u]_y & = & \hat{h}, & \quad \lambda \in \Sigma_{\pi - \theta}, & \xi \in \bR^{n - 1},
	\end{array}
\end{equation*}
where we employ the notation from \Subsecref{Linear-Halfspace-Proof-Tangential} with $\omega = \sqrt{\lambda_\epsilon + \mu |\xi|^2}$.
We immediately obtain $\hat{u}(\lambda,\,\xi,\,y) = \hat{\tau} e^{-(\omega / \sqrt{\mu}) y}$ for an unkown boundary value $\tau: \bR_+ \times \bR^{n - 1} \longrightarrow \bR$,
which has to be determined based on the boundary condition
\begin{equation*}
	\alpha \lambda_\epsilon [\hat u]_y - \beta [\partial_y \hat u]_y = \left( \alpha \lambda_\epsilon + \beta \frac{\omega}{\sqrt{\mu}} \right) \hat{\tau} = \hat{h}.
\end{equation*}
This implies that
\begin{equation*}
	[\hat{u}]_y = \hat{\tau} = \frac{\sqrt{\mu}}{\alpha \sqrt{\mu} \lambda_\epsilon + \beta \omega} \hat{h}
\end{equation*}
and since the symbols
\begin{equation*}
	(\lambda,\,z) \mapsto \frac{\sqrt{\mu} \lambda}{\alpha \sqrt{\mu} \lambda_\epsilon + \beta \omega(z)},\ \frac{\sqrt{\mu} z}{\alpha \sqrt{\mu} \lambda_\epsilon + \beta \omega(z)}:
		\Sigma_{\pi - \theta} \times \Sigma_{\theta / 2} \longrightarrow \bC
\end{equation*}
are bounded and holomorphic for $0 < \theta < \frac{\pi}{2}$,
we obtain the desired regularity of $[u]_y$ by the bounded $\cH^\infty$-calculus of the operators $\partial_t$ and $\sqrt{- \Delta_\Gamma}$,
cf.~\Subsecref{Linear-Halfspace-Proof-Tangential}.
\end{proof}

\section*{Acknowledgements}
T.\,K.~gratefully acknowledges financial support by the Deutsche For\-schungs\-gemeinschaft within the International Research Training Group ``Mathematical Fluid Dynamics'' (IRTG 1529).

\bibliographystyle{plain}
\bibliography{references}
\end{document}